\documentclass[11pt]{amsart}
\usepackage[a4paper]{geometry}                                              
\usepackage{fourier,mathtools}                                             
\usepackage[bb=ams, cal=cm, scr=boondox, frak=euler]{mathalpha}             
\let\amsmathbb\mathbb
\AtBeginDocument{%
    \let\mathbb\relax
    \newcommand{\mathbb}[1]{\amsmathbb{#1}}
}
\mathtoolsset{mathic=true}
\usepackage{amsmath,amsthm,amsfonts,amssymb}                                
\usepackage{mathrsfs}                                                      
\usepackage{esint}                                                          
\usepackage[all]{xy}                                                      
\usepackage{tikz-cd}                                                      
\usepackage[colorlinks=true,linkcolor=magenta,citecolor=blue]{hyperref}    
\usepackage[capitalise]{cleveref}                                          
\usepackage{colonequals}                                                   
\newcommand{\bA}{{\mathbb{A}}}

\newcommand{\bG}{{\mathbb{G}}}

 \newcommand{\bN}{{\mathbb{N}}}
 
\newcommand{\bQ}{{\mathbb{Q}}}

 \newcommand{\bZ}{{\mathbb{Z}}}
\newcommand{\cA}{{\mathcal{A}}} \newcommand{\cB}{{\mathcal{B}}}
\newcommand{\cC}{{\mathcal{C}}} \newcommand{\cD}{{\mathcal{D}}}
 
\newcommand{\cG}{{\mathcal{G}}} \newcommand{\cH}{{\mathcal{H}}}

 \newcommand{\cT}{{\mathcal{T}}}

\newcommand{\sA}{{\mathscr{A}}} 
 
 \newcommand{\sF}{{\mathscr{F}}}
\newcommand{\sG}{{\mathscr{G}}} \newcommand{\sH}{{\mathscr{H}}}
 
 \newcommand{\sL}{{\mathscr{L}}}

\DeclareMathOperator{\Char}{char}

\DeclareMathOperator{\coker}{coker}

\DeclareMathOperator{\Corr}{Cor}

\DeclareMathOperator{\Div}{Div}

\DeclareMathOperator{\Gal}{Gal}

\DeclareMathOperator{\Hom}{Hom}

\DeclareMathOperator{\id}{id}

\DeclareMathOperator{\LN}{LN}
\DeclareMathOperator{\Mor}{Mor}

\DeclareMathOperator{\NS}{NS}

\DeclareMathOperator{\Pic}{Pic}

\DeclareMathOperator{\rank}{rank}

\DeclareMathOperator{\Spec}{Spec}

\newcommand{\ab}{{\mathrm{ab}}} 
\newcommand{\aff}{{\mathrm{aff}}}

\newcommand{\eff}{{\mathrm{eff}}}
\newcommand{\et}{{\mathrm{\acute{e}t}}}

\newcommand{\free}{{\mathrm{fr}}}

\newcommand{\perf}{{\mathrm{perf}}}
\newcommand{\red}{{\mathrm{red}}}

\newcommand{\tor}{{\mathrm{tor}}}
\newcommand{\tr}{{\mathrm{tr}}}

\newcommand{\AV}{{\mathsf{AV}}}
\newcommand{\Cor}{{\mathsf{Cor}}}

\newcommand{\DM}{{\mathsf{DM}}}

\newcommand{\Et}{{\mathsf{Et}}}
\newcommand{\Grp}{{\mathsf{Grp}}} 
 
\newcommand{\HI}{{\mathsf{HI}}}

\newcommand{\M}{{\mathsf{M}}}  

\newcommand{\MM}{{\mathsf{MM}}}   
\newcommand{\Mod}{{\mathsf{Mod}}}

\newcommand{\PST}{{\mathsf{PST}}}

\newcommand{\SAV}{{\mathsf{SAV}}}
\newcommand{\Sch}{{\mathsf{Sch}}} 
\newcommand{\Shv}{{\mathsf{Shv}}}

\newcommand{\Sm}{{\mathsf{Sm}}}

\newcommand{\Tori}{{\mathsf{Tori}}}

\usepackage{enumitem}
\newcommand{\menum}{\end{enumerate}\vspace*{-2\partopsep}\begin{enumerate}[label={\rm (\arabic*)},resume]}

\theoremstyle{plain}

\newtheorem{thm}{Theorem}[subsection]  
\newtheorem{cor}[thm]{Corollary}
\newtheorem{lem}[thm]{Lemma}
\newtheorem{prop}[thm]{Proposition}

\theoremstyle{definition}
\newtheorem{defn}[thm]{Definition}          
\newtheorem{rmk}[thm]{Remark}                 
\newtheorem{eg}[thm]{Example}
\newtheorem{hyp}[thm]{Hypothesis}

\makeatletter
\def\@tocline#1#2#3#4#5#6#7{\relax
  \ifnum #1>\c@tocdepth 
  \else
    \par \addpenalty\@secpenalty\addvspace{#2}%
    \begingroup \hyphenpenalty\@M
    \@ifempty{#4}{%
      \@tempdima\csname r@tocindent\number#1\endcsname\relax
    }{%
      \@tempdima#4\relax
    }%
    \parindent\z@ \leftskip#3\relax \advance\leftskip\@tempdima\relax
    \rightskip\@pnumwidth plus4em \parfillskip-\@pnumwidth
    #5\leavevmode\hskip-\@tempdima
      \ifcase #1
       \or\or \hskip 2em \or \hskip 2em \else \hskip 3em \fi%
      #6\nobreak\relax
    \hfill\hbox to\@pnumwidth{\@tocpagenum{#7}}\par
    \nobreak
    \endgroup
  \fi}
\makeatother

\tikzcdset{scale cd/.style={every label/.append style={scale=#1},cells={nodes={scale=#1}}}}

\begin{document}
\title{Chow trace of $1$-Motives and the Lang-N\'eron Groups}                                                                              
\begin{abstract}
    We show that in the case of primary field extensions, the extension of scalars of Deligne $1$-motives admits a left adjoint, called Chow image, and a right adjoint, called Chow trace. This generalizes Chow's results on abelian varieties. Then we study the Chow trace in the framework of Voevodsky's triangulated categories of (\'etale) motives. With respect to the $1$-motivic $t$-structure on the category of Voevodsky's homological $1$-motives, the zero-th direct image of an abelian variety is given by the Chow trace, and the first direct image is the $0$-motive defined by the (geometric) Lang-N\'eron group.
\end{abstract}

\author[Long Liu]{Long Liu \smallskip\\ \MakeLowercase{\uppercase{W}ith an appendix by \uppercase{B}runo \uppercase{K}ahn}}
\address{Sorbonne Universit\'e and Universit\'e Paris Cit\'e, CNRS, IMJ-PRG, F-75005 Paris, France.}
\email{long.liu@imj-prg.fr}                                                         
\address{Sorbonne Universit\'e and Universit\'e Paris Cit\'e, CNRS, IMJ-PRG, F-75005 Paris, France.}
\email{bruno.kahn@imj-prg.fr}

\date{August 2, 2024}  
\maketitle  

\tableofcontents

\section{Introduction}
\subsection{Background}
Let $K/k$ be a field extension, $K_s$ be a separable closure of $K$ and $k_s$ be the separable closure of $k$ in $K_s$. For a finite separable field extension $K/k$, the field $k_s$ is $K_s$ itself and the absolute Galois group $\Gal(K_s/K)$ is canonically an open subgroup of $\Gal(k_s/k)$. Then the extension of scalars of discrete Galois modules admits a left adjoint and a right adjoint, both of which are given by the induced modules in the sense of \cite[Chapter I, 2.5]{Serre02GaloisCohomology}. A deeper result is the existence of a left adjoint and a right adjoint to the extension of scalars of abelian varieties, both of which are given by the Weil restriction. See, for example, \cite[Th. ~4.2~and~4.3]{Kahn18MotifsAdjoint}.

Another interesting case is when $K/k$ is a primary extension of fields, which means that the algebraic closure of $k$ in $K$ is purely inseparable over $k$. Then the canonical homomorphism of absolute Galois groups $\Gal(K_s/K)\to \Gal(k_s/k)$ is surjective. The extension of scalars of discrete Galois modules admits a left adjoint and a right adjoint, given by $\Gal(K_s/Kk_s)$ co-invariants and $\Gal(K_s/Kk_s)$ invariants respectively. A deeper result (\cite{Chow55AVfunctionfield}) is the existence of a left adjoint and a right adjoint to the extension of scalars of abelian varieties, called Chow image and Chow trace respectively. Lang and N\'eron (\cite{LangNeron59LNfg}) proved a relative version of Mordell-Weil theorem using Chow's trace: Let $K/k$ be a finitely generated regular field extension. Let $A$ be an abelian variety over $K$ and $\pi_*A$ be its $K/k$-trace. Then the Lang-N\'eron group
\[\LN(A,K/k) \colonequals A(K)/(\pi_*A)(k)\]
is a finitely generated abelian group. See \cite{Conrad06ChowLN} and \cite[Appendices A and B]{Kahn06ClassGroup} for `modern' proofs using Grothendieck's theory of schemes and fpqc descent.

\subsection{Main results}
The existence of Chow's trace can be generalized to Deligne $1$-motives. Recall \cite[10.1.10]{Deligne74HodgeIII} that a Deligne $1$-motive over $k$ is a two-term complex of group schemes $[L \to G]$, where $L$ is a lattice and $G$ is a semi-abelian variety. Here, a lattice means a commutative \'etale group scheme $L$ over $k$ such that $L(k_s)$ is a finitely generated free $\bZ$-module, and a semi-abelian variety is a commutative algebraic group which is an extension of an abelian variety by a torus. A morphism of Deligne $1$-motives is defined to be a commutative square in the obvious sense. Denote the category of Deligne $1$-motives by $\M_1(k)$. The base change of group schemes induces a base change functor of Deligne $1$-motives. The following result answers positively the expectation from \cite[bottom of p.~82]{Kahn18MotifsAdjoint}.

\begin{thm}[{Theorems~\ref{M1fullyfaithful} and~\ref{M1adjoint}}]
    Let $K$ be a primary field extension of $k$. Then 
    \begin{enumerate}[label={\rm(\arabic*)}]
        \item the extension of scalars of Deligne $1$-motives
        \begin{align*}
            \pi^* \colon  \M_1(k)        &\longrightarrow \M_1(K)\\
                [L\to G]       &\longmapsto [L_K\to G_K]
        \end{align*}
        is fully faithful;
        \item $\pi^*$ has a left adjoint $\pi_\sharp^{\M_1}$ and a right adjoint $\pi_*^{\M_1}$.
    \end{enumerate}
\end{thm}
Our functors $\pi_\sharp^{\M_1}$ and $\pi_*^{\M_1}$ recover some classical constructions (Corollaries \ref{M1AdjExample} and \ref{ZtoAtrivial}, and Proposition \ref{ZtoAAlb}), such as Chow's image and trace of abelian varieties. Thus we call these two functors Chow image and Chow trace respectively. A key ingredient of the existence of Chow image and Chow trace is the fact that primary field extensions will not bring new semi-abelian subvarieties (Theorem~\ref{SAVdescent}). 

We also want to study the derived functors of $\pi_*$. However, the category of Deligne $1$-motives is neither abelian nor big enough. Let $\Lambda$ be $\bZ[1/p]$, the localization of $\bZ$ by inverting $p$ the exponential characteristic of $k$. Thanks to the work of Voevodsky, Orgogozo, Barbieri-Viale, Kahn and Ayoub (\cite{Voevodsky00DM}, \cite{Orgogozo041Motives}, \cite{BVK16Derived1Motives}, \cite{Ayoub11Motivic-t-structure}), $\M_1(k)\otimes_\bZ\Lambda$ is a full subcategory of the heart of a $1$-motivic $t$-structure on Voevodsky's triangulated category $\DM_{\leq 1}(k,\Lambda)$ of \'etale homological $1$-motives, i.e., the localizing subcategory of $\DM_\et^\eff(k,\Lambda)$ generated by the motives $M(X)$ for $\dim X\leq 1$.

By the work of Ayoub and Barbieri-Viale \cite{ABV09MotShvLAlb}, $\DM_{\leq 1}(k,\Lambda)$ is canonically equivalent to the unbounded derived category of (\'etale) $1$-motivic sheaves $\HI_{\leq 1}(k,\Lambda)$, which is the smallest co-complete Serre subcategory of the category of \'etale sheaves with transfers containing lattices and \'etale sheaves represented by semi-abelian varieties. And every smooth curve $C$ defines a $1$-motivic sheaf $h_0^\et(C)$, which will form a system of generators of $\HI_{\leq 1}(k,\Lambda)$. The category of $1$-motivic sheaves contains the category $\HI_{\leq 0}(k,\Lambda)$ of $0$-motivic sheaves, which is equivalent to the category of sheaves of $\Lambda$-modules on the site $(\Spec k)_\et$. Ayoub and Barbieri-Viale showed that the inclusion $\delta \colon  \HI_{\leq 0}\hookrightarrow \HI_{\leq 1}$ admits a left adjoint $\pi_0$, which is constructed using the scheme of connected components. Then we will have an analogue of the connected-\'etale exact sequence 
\[0\to \sF^0 \to \sF \to \pi_0(\sF) \to 0.\]

Let $K/k$ be a field extension. Then the inverse image functor
\begin{align*}
    e^* \colon   \HI_{\leq 1}(k,\Lambda) &\longrightarrow         \HI_{\leq 1}(K,\Lambda) \\
        h_0^\et(C)               &\longmapsto  h_0^\et(C_K)
\end{align*}
admits a right adjoint $e_*$, which has a total right derived functor $Re_*$. Similarly, we also have a direct image functor $\varepsilon_*$ for $0$-motivic sheaves.

\begin{thm}[{Theorem~\ref{HI0HI1Re*}}]
    For $\sF\in\HI_{\leq 0}(K,\Lambda)$, we have a canonical isomorphism
    \[\delta R^i\varepsilon_*\sF \stackrel{\sim}{\longrightarrow} R^ie_*\delta\sF.\]
    In particular, if $K/k$ is primary, then $R^ie_*\delta\sF$ is the $0$-motivic sheaf associated with the $\Gal(k_s/k)$-module $H^i(\Gamma,\sF_{K_s})$, where $\Gamma=\Gal(K_s/Kk_s)$.
\end{thm}
The key point of the proof is a smooth base change theorem for non-torsion \'etale sheaves (Corollary~\ref{SmBCfield}), whose proof will be given in Appendix \ref{SmBC}. Besides, we shall need some knowledge about model categories to study the unbounded derived functors used in the proof.

\begin{thm}[{Theorems~\ref{ChowRevisited} and \ref{Rie*A}}]
    Let $K/k$ be a field extension.
    \begin{enumerate}[label={\rm(\arabic*)}]
        \item If $A$ is an abelian variety over $K$, then $R^ie_*A$ is a torsion $0$-motivic sheaf for $i\geq 1$.
        \item If $K/k$ is a primary field extension, then the connected-\'etale exact sequence associated with $e_*A$ is 
            \[0\to \pi_*A \to e_*A \to \LN(A,Kk_s/k_s)_\Lambda\to 0.\]
    \end{enumerate}
\end{thm}
To prove the first assertion, we shall use the fact that $H^i_\et(X,\sF)$ is torsion for $i>\dim X$, Raynaud's theorem that $H^1_\et(X,A)$ is torsion for $X$ noetherian regular and $A$ an abelian scheme over $X$,  and Suslin's rigidity theorem \cite[Theorem~7.20]{MVW06Motive}. For the second assertion, we shall check that Chow trace is the connected component of the direct image by using a structure theorem of $1$-motivic sheaves, which is due to Ayoub, Barbieri-Viale and Kahn, and using the universal property of Chow trace. The Lang-N\'eron theorem now can be used to deduce the finiteness of $e_*A$ when $K/k$ is a finitely generated regular extension (see Corollary~\ref{LNe*A}).

We shall refine the $1$-motivic $t$-structure with $\bQ$-coefficients in \cite{Ayoub11Motivic-t-structure} to $\bZ[1/p]$-integral coefficients. An object in $D(\HI_{\leq 1}(k,\Lambda))$ is in the heart of the $1$-motivic $t$-structure if and only if it is quasi-isomorphic to a two-term complex $[L\to G]$ concentrated in degrees $0,1$ with $\ker(L\to G)$ a $0$-motivic sheaf and $\coker(L\to G)$ a connected $1$-motivic sheaf. We will call it a $0$-motive if it is quasi-isomorphic to $[L\to 0]$ with $L$ a $0$-motivic sheaf. Using a proposition (\ref{tstrucExact}) comparing the two $t$-structures on $D(\HI_{\leq 1})$, we can translate the above theorems to some results on the higher direct images relative to the $1$-motivic $t$-structure.

Denote by $R^ie_*$ (resp. ${}^m\!R^ie_*=[L^i\to G^i]$) the cohomology of $Re_*$ relative to the standard (resp. $1$-motivic) $t$-structure on $D(\HI_{\leq 1}(k,\Lambda))$.
\begin{thm}[{Theorem~\ref{mRe*0}}]
    Let $K/k$ be a field extension and let $L$ be a \(0\)-motivic sheaf over $K$. Then 
    \[{}^m\!R^ie_*[L\to 0]=[R^ie_*L\to 0].\]
    In particular, ${}^m\!R^ie_*[L\to 0]$ is a torsion $0$-motive for $i\geq 1$.
\end{thm}

\begin{thm}[{Theorem~\ref{mRe*AV}}]
    Let $K/k$ be a primary field extension and let $A$ be an abelian variety over $K$. Then 
    \[{}^m\!R^ie_*[0\to A]=\left\{
        \begin{array}{ll}
            {[0\to \pi_*A]}, & \hbox{if \(i=0\);} \\
            {[\LN(A,Kk_s/k_s)\to 0]}, & \hbox{if \(i=1\);} \\
            {[R^{i-1}e_*A\to 0]}, & \hbox{if \(i\geq 2\).}
        \end{array}\right.\]
    In particular, ${}^m\!R^0e_*[0\to A]$ is a constructible $1$-motive, and ${}^m\!R^ie_*[0\to A]$ are torsion $0$-motives for $i\geq 2$. Moreover, if $K/k$ is a finitely generated regular extension, then ${}^m\!R^1e_*[0\to A]$ is a constructible $0$-motive. 
\end{thm}

\begin{thm}[Theorem \ref{mRe*Gm}]\label{GmPic}
    Let $X$ be a smooth projective and geometrically connected variety over $k$ and let $K$ be the function field of $X$. Then 
    \[{}^m\!R^ie_*[0\to \bG_m]=\left\{
        \begin{array}{ll}
            {[0\to \bG_m]}, & \hbox{if \(i=0\);} \\
            0             , & \hbox{if \(i=2\);} \\
            {[R^{i-1}e_*\bG_m \to 0]}, & \hbox{if \(i\geq 3\).}
        \end{array}\right.
    \]
    Moreover, with $\bQ$-coefficients, we have 
    \[ {}^m\!R^1e_*[0\to \bG_m] = [\Div^0(X_{k_s}) \to \Pic^0_{X/k}]. \] 
\end{thm}

The following result can be viewed as a generalization of the Lang-N\'eron theorem for certain Deligne $1$-motives.
\begin{thm}[Theorem \ref{LNLA}]
    Let $K/k$ be a finitely generated regular field extension and let $M=[L \to A]$ be a Deligne $1$-motive over $K$ where $A$ is an abelian variety. Write $\Gamma=\Gal(K_s/Kk_s)$. Then we have an exact sequence of \(0\)-motivic sheaves 
    \[ 0 \to X \to \pi_0(R^1e_*M) \to Y \to 0,\]
    where 
    \[ X =\coker(L(K_s)^\Gamma \to \LN(A,Kk_s/k_s)) \] and 
    \[ Y =\ker(H^1(\Gamma,L(K_s)) \to R^1e_*A). \]
    In particular, 
    \begin{enumerate}[label={\rm(\arabic*)}]
        \item $\pi_0(R^1e_*M)(k_s)$ is a finitely generated $\Gal(k_s/k)$-module;
        \item ${}^m\!R^1e_*M=[\pi_0(R^1e_*M) \to 0]$ is a constructible $0$-motive.
    \end{enumerate}
\end{thm}
This fails in the presence of tori; see Theorem \ref{GmPic}. One can recover a finite generation statement if one replaces the function field by a smooth model.

The story in the case when $K/k$ is a finite extension can be found in \cite[A.17 and Lemma~2.22]{PL19DA1}. In fact, he studied it in the more general setting when $f \colon  X\to Y$ is a finite \'etale morphism. 

\subsection{Conventions and notations}
We shall use the following categorical notions in the spirit of \cite[Definition 8.3.21]{KS06CategorySheaf}.
\begin{defn}\label{subcategories}
    Let $\cA$ be an abelian category and $\cB$ be a full subcategory of $\cA$.
    \begin{enumerate}[label={\rm(\arabic*)}]
        \item We say that $\cB$ is a Serre subcategory of $\cA$ if it is closed under subobjects, quotients and extensions.
        \item We say that $\cB$ is a thick subcategory of $\cA$ if it is closed under kernels, cokernels and extensions.
        \item We say that $\cB$ is a fully abelian subcategory of $\cA$ if it is additive and closed under kernels and cokernels, equivalently, $\cB$ is an abelian category and the embedding functor is exact.
    \end{enumerate}
\end{defn}
\begin{rmk}
    \begin{enumerate}[leftmargin=*,label={\rm(\arabic*)}]
        \item Clearly, a Serre subcategory is thick, and a thick subcategory is a fully abelian subcategory. \menum
        \item In \cite[\href{https://stacks.math.columbia.edu/tag/02MO}{Definition~02MO}]{stacks-project}, a (strictly full) thick subcategory is called a weak Serre subcategory.
    \end{enumerate}
\end{rmk}

We list some notations used in this chapter:
\begin{itemize}
    \item $k$, field of exponential characteristic $p$
    \item $\Lambda$, the ring $\bZ[1/p]$   
    \item $\M_0(k)$ (resp. ${}^t\M_0(k)$), category of lattices (resp. of constructible group schemes) (\ref{M0tM0})
    \item $\Tori(k)$, category of tori
    \item $\AV(k)$ (resp. $\SAV(k)$), category of abelian varieties (resp. of semi-abelian varieties)
    \item ${}_nG$, kernel of the multiplication by $n$ on a semi-abelian variety
    \item $\M_1(k)$, category of Deligne $1$-motives
    \item $\Sm/k$, category of smooth schemes separated of finite type over $k$
    \item $(\Sm/k)_{\leq n}$, full subcategory of $\Sm/k$ consisting of schemes with dimension $\leq n$
    \item $\Shv_\et(\cC,\Lambda)$, category of sheaves of $\Lambda$-modules on the \'etale site $\cC_\et$
    \item $\Shv_\et^\tr(k,\Lambda)$, category of \'etale sheaves with transfers on $\Sm/k$ (reviewed in \S \ref{EST})
    \item $\Shv_\et^\tr(k_{\leq n},\Lambda)$, category of \'etale sheaves with transfers on $(\Sm/k)_{\leq n}$ (reviewed in \S \ref{EST})
    \item $\HI_\et(k,\Lambda)$, category of homotopy invariant sheaves (reviewed in \S \ref{HI_et})
    \item $\HI_{\leq n}(k,\Lambda)$, category of $n$-motivic sheaves (reviewed in \S \ref{HI_leqn})
    \item $\DM_\et^\eff(k,\Lambda)$, Voevodsky category of effective \'etale motives
    \item $\DM_{\leq n}(k,\Lambda)$, the localizing subcategory of $\DM_\et^\eff(k,\Lambda)$ generated by the motives $M(X)$ for $\dim X\leq n$
    \item $\pi^* \colon \M_1(k) \to \M_1(K)$, extension of scalars of Deligne $1$-motives induced by base change of schemes
    \item $\pi_*$, right adjoint to $\pi^*$, i.e., Chow trace
    \item $\pi_\sharp$, left adjoint to $\pi^*$, i.e., Chow image
    \item $\gamma_* \colon \Shv_\et^\tr(k,\Lambda) \to \Shv_\et(\Sm/k,\Lambda)$, the forgetful functor
    \item $\gamma^*$, left adjoint to $\gamma_*$
    \item $\iota \colon \HI_\et^\tr(k,\Lambda) \to \Shv_\et^\tr(k,\Lambda)$, the inclusion functor
    \item $h_0^\et \colon \Shv_\et^\tr(k,\Lambda) \to \HI_\et^\tr(k,\Lambda)$, a left adjoint to $\iota$
    \item $\iota_n \colon \HI_{\leq n}(k,\Lambda) \to \HI_\et^\tr(k,\Lambda)$, the inclusion functor
    \item $\sigma_{n}^* \colon \Shv_\et^\tr(k_{\leq n},\Lambda) \leftrightarrows \Shv_\et^\tr(k,\Lambda) \colon \sigma_{n*}$, extension and restriction functors
    \item $e^*_\tr,~e_*^\tr$, inverse image and direct image functors of sheaves with transfers on $\Sm$
    \item $e^*_{\leq n},~e_*^{\leq n}$, inverse image and direct image functors of sheaves (with transfers) on $(\Sm)_{\leq n}$
    \item $e^*_\HI,~e_*^\HI$, inverse image and direct image functors of homotopy invariant sheaves
    \item $e_n^*,~e_{n*}$, inverse image and direct image functors of $n$-motivic sheaves
    \item $\pi_0 \colon \Shv_\et^\tr(k,\Lambda) \to \HI_{\leq 0}(k,\Lambda)$, left adjoint to the inclusion functor
\end{itemize}

\section{Chow image and Chow trace of Deligne \texorpdfstring{$1$}{1}-motives}
In this section, we study the extension of scalars of Deligne $1$-motives. More precisely, we will show some full faithfulness results and will construct the Chow image and Chow trace of Deligne $1$-motives.
\subsection{Commutative \'etale group schemes}
In this subsection, we fix our notations and recall some well-known and not so well-known facts on \'etale group schemes over a field.

Let $k$ be a field and let $k_s$ be a separable closure of $k$.
\begin{defn}\label{M0tM0}
    Let $L$ be a commutative \'etale group scheme over $k$.
    \begin{enumerate}[label={\rm(\arabic*)}]
        \item We say that $L$ is a constructible group scheme if $L(k_s)$ is a finitely generated abelian group.
        \item We say that $L$ is a lattice if $L(k_s)$ is a finitely generated free abelian group.
        \item Denote the category of constructible group schemes (resp. lattices) over $k$ by ${}^t\M_0(k)$ (resp. $\M_0(k)$), where morphisms are homomorphisms between group schemes.
    \end{enumerate}
\end{defn}
\begin{rmk}
    Constructible group schemes are called discrete group schemes in \cite{BVK16Derived1Motives}. We call such group schemes constructible to emphasize the finiteness and to avoid the confusion with discrete Galois modules. 
\end{rmk}

\begin{rmk}\label{SubM0}
    By \cite[Chapitre II, \S5, Proposition~1.4]{DG70AlgGrp}, a group scheme $G$ locally of finite type over $k$ is \'etale if and only if $G^0\simeq\Spec k$. Thus if $G$ is \'etale, then every subgroup scheme is also \'etale. Moreover, if $G$ is a constructible group scheme (resp. a lattice, resp. a finite \'etale group scheme), then so is every subgroup scheme.
\end{rmk}

The following result gives a concrete way to study commutative \'etale group schemes over $k$:
\begin{prop}[{\cite[Chapitre II, \S5, Proposition~1.7]{DG70AlgGrp}}]\label{EtaleModule}
    The functor $L \mapsto L(k_s)$ is an equivalence from the category of \'etale group schemes to the category of discrete $\Gal(k_s/k)$-groups. Moreover, via this functor, constructible group schemes (resp. lattices, resp. commutative finite \'etale group schemes) correspond to finitely generated (resp. finitely generated free, resp. finite) abelian groups with continuous $\Gal(k_s/k)$-actions.
\end{prop}

Let $S$ be a scheme. For a commutative group scheme $G$ over $S$, we have a group functor
\begin{align*}
    G^\vee  \colon \Sch/S &\longrightarrow \{\text{abelian groups}\} \\
                X \quad  &\longmapsto     \Hom_X(G_X,\bG_{m,X}),
\end{align*}
called Cartier duality of $G$.

Recall that a group scheme over $k$ is called a torus if $T_{\overline k}\simeq \bG_{m,\overline k}^n$ for some $n\in\bN$. Tori are commutative, connected, affine, smooth and of finite type over $k$. Denote by $\Tori(k)$ the category of tori over $k$. We have the following duality theorem: 

\begin{thm}\label{CartierDual}
    The functors $T\mapsto T^\vee$ and $L\mapsto L^\vee$ are anti-equivalences, quasi-inverses of each other, between $\Tori(k)$ and $\M_0(k)$.
\end{thm}
\begin{proof}
    By \cite[Chapitre IV, \S1, Corollaire 3.3]{DG70AlgGrp}, these two functors are anti-equivalences, quasi-inverses of each other, between the category of the group schemes of multiplicative type over $k$ and the category of commutative \'etale group schemes over $k$. Then by \cite[Chapitre IV, \S1, Corollaire 3.9 (a)]{DG70AlgGrp}, tori correspond to lattices.
\end{proof}

\subsection{Semi-abelian varieties}
In this subsection, we shall work in the abelian category of commutative algebraic groups (i.e., group schemes of finite type) over a field $k$; see {\cite[Expos\'e~VI$\rm _A$, Th\'eor\`eme~5.4.2]{SGA3I}} for a proof that this category is abelian. By \cite[Chapitre II, \S5, Th\'eor\`eme~2.1]{DG70AlgGrp}, an algebraic group over $k$ is smooth if and only if it is geometrically reduced.

\begin{lem}\label{AffineAVtrivial}
    Let $T$ be a smooth connected affine algebraic group over $k$ and let $A$ be an abelian variety over $k$. Then there is neither a nontrivial homomorphism from $T$ to $A$ nor a nontrivial homomorphism from $A$ to $T$.
\end{lem}
\begin{proof}
    For any homomorphism $f \colon  T\to A$, the quotient $T/\ker(f)$ inherits the properties of being smooth connected and affine from $T$ by \cite[Expos\'e~VI$\rm _B$, Proposition~9.2(xii) and Th\'eor\`eme~11.17]{SGA3I}. Since $A$ is proper, its closed subgroup $T/\ker(f)$ is also proper. Because $T/\ker(f)$ is both proper and affine, it is finite. Since $T/\ker(f)$ is finite \'etale and connected, it is isomorphic to $\Spec k$, which means that $f$ is trivial.

    Similarly, for any homomorphism $g \colon  A\to T$, the quotient $A/\ker(g)$ inherits the properties of being smooth connected and proper from $A$. Since $T$ is affine, its closed subgroup $A/\ker(g)$ is also affine. Because $A/\ker(g)$ is both proper and affine, it is finite. Since $A/\ker(g)$ is finite \'etale and connected, it is isomorphic to $\Spec k$, which means that $g$ is trivial.
\end{proof}
\begin{rmk}
    When $T$ is a torus, we can also use \cite[Corollary~3.9]{Milne86AV} to see that there exists no nontrivial (homo)morphism from $T$ to $A$.
\end{rmk}

Let $G$ be a smooth connected commutative algebraic group over $k$. By Chevalley's theorem (see, e.g., \cite{Conrad02ChevalleyThmAlgGrp}), $G_{\overline k}$ is uniquely an extension of an abelian variety by a smooth connected affine group $G_{\overline k}^\aff$. By Lemma~\ref{AffineAVtrivial}, these are functorial in $G$.

\begin{defn} 
    A commutative algebraic group $G$ over $k$ is called a semi-abelian variety if it can be represented by an extension
    \[0\to T\to G \to A\to 0,\]
    where $T$ is a torus and $A$ is an abelian variety. Since $T$ and $A$ are both smooth and connected, so is $G$.
\end{defn}
\begin{rmk}\label{SAV}
    \begin{enumerate}[leftmargin=*,label={\rm(\arabic*)}]
        \item We call $T$ and $A$ the toric part and abelian part of $G$ respectively. By Lemma~\ref{AffineAVtrivial}, these are functorial in $G$. In particular, the groups $T$ and $A$ are uniquely determined by $G$.  \menum
        \item By \cite[bottom of p. 178]{BLR90Neron}, a smooth connected commutative algebraic group $G$ is semi-abelian if and only if $G_{\overline k}^\aff$ is a torus.
    \end{enumerate}
\end{rmk}

\begin{lem}\label{SubQuotSAV}
    \begin{enumerate}[leftmargin=*,label={\rm(\arabic*)}]
        \item The quotients and smooth connected subgroups of a torus are tori.\menum
        \item The quotients and smooth connected subgroups of an abelian variety are abelian varieties.
        \item The quotients and the smooth connected subgroups of a semi-abelian variety are semi-abelian varieties.
    \end{enumerate}
\end{lem}
\begin{proof}
    \begin{enumerate}[leftmargin=*,label={\rm(\arabic*)}]
        \item This assertion follows from \cite[Chapitre IV, \S1, Corollaires 2.4 and 3.9 (a)]{DG70AlgGrp}.
        \item The closed subgroups of abelian varieties are proper over $k$. If they are smooth connected, then they are abelian varieties by definition. The quotients of abelian varieties inherit the properties of being smooth connected and proper over $k$ (see \cite[Expos\'e~VI$\rm _B$, Proposition~9.2(xii)]{SGA3I}). Thus they are abelian varieties. 
        \item By Remark~\ref{SAV} (2), we may and do assume that $k$ is algebraically closed. Let $G$ be a semi-abelian variety over $k$, i.e., $G^\aff$ is a torus. 
    
        If $G'$ is a smooth connected subgroup of $G$, then $G'^{\aff}$ is a closed subgroup of the torus $G^\aff$. Thus $G'^\aff$ is also a torus, which implies that $G'$ is a semi-abelian variety. 
        
        Let $f \colon  G\to G''$ be a surjection. Let $H$ be the categorical image of the induced morphism $G^\aff\to G''^\aff$ in the abelian category of commutative algebraic groups over $k$. Since $G/G^\aff$ is an abelian variety, its quotient $G''/H$ is also an abelian variety. Then the smooth connected subgroup $G''^\aff/H$ of $G''/H$ is an abelian variety. But $G''^\aff/H$ inherits the property of being affine from $G''^\aff$. Thus $G''^\aff/H$ is trivial, which implies that the morphism $G^\aff \to G''^\aff$ is an epimorphism. Then $G''^\aff$ inherits the property of being a torus from $G^\aff$. Thus $G''$ is a semi-abelian variety.\qedhere
    \end{enumerate}
\end{proof}

\begin{lem}\label{nGFiniteEtale}
    Let $k$ be a field and $G$ be a semi-abelian variety over $k$. Let $n$ be a positive integer prime to $\Char(k)$ and let ${}_nG$ be the kernel of multiplication by $n$ on $G$. Then ${}_nG$ is finite \'etale over $k$.
\end{lem}
\begin{proof}
    See \cite[\S 7.3, Lemmas 1 and 2]{BLR90Neron}.
\end{proof}

\begin{lem}\label{nGOrder}
    Let $k$ be an algebraically closed field and let $G$ be a semi-abelian variety over $k$. Let $n$ be a positive integer prime to $\Char(k)$. Then
    \[\#{}_nG(k)=n^{t+2a},\]
    where $t$ (resp. $a$) is the dimension of the toric part (resp. abelian part) of $G$.
\end{lem}
\begin{proof}
    Consider the following commutative diagram with exact rows in the abelian category of commutative algebraic groups
    \[\xymatrix{
        0\ar[r] & T \ar[r] \ar[d]_n & G \ar[r] \ar[d]_n & A \ar[r] \ar[d]_n & 0\\
        0\ar[r] & T \ar[r]          & G \ar[r]          & A \ar[r]          & 0.
    }\]
    Using the snake lemma and noting that $n \colon  T\to T$ is an epimorphism, we have the following exact sequence
    \[0\to {}_nT \to {}_nG \to {}_nA \to 0.\] 
    By Lemma~\ref{nGFiniteEtale}, the group schemes ${}_nT$, ${}_nG$ and ${}_nA$ are finite \'etale over $k$. Using Proposition~\ref{EtaleModule}, we obtain the following exact sequence of finite abelian groups
    \[0\to {}_nT(k) \to {}_nG(k) \to {}_nA(k) \to 0.\] 
    Since $k$ is algebraically closed, we have $T\simeq \bG_m^t$ and $\#{}_n\!T(k)=n^t$. By \cite[p.60]{Mumford14AV} or \cite[Theorem~8.2]{Milne86AV}, the map $n\colon A\to A$ is finite \'etale of degree $n^{2a}$. So $\#{}_nA(k)=n^{2a}$. Hence, we obtain $\#{}_nG(k)=\#{}_nT(k)\cdot \#{}_nA(k)=n^{t+2a}$.
\end{proof}

\begin{lem}\label{DimSubSAV}
    Let $G$ be a semi-abelian variety over a field $k$ and let $H$ be a semi-abelian subvariety of $G$. Let $a$ (resp. $a_0$) be the dimension of the abelian part and let $t$ (resp. $t_0$) be the dimension of the toric part of $G$ (resp. $H$). Then we have 
    \[a\geq a_0 \text{ and }t\geq t_0.\]
\end{lem}
\begin{proof}
    We may and do assume that $k$ is algebraically closed. By Remark~\ref{SAV} (1), we can write the closed immersion $i \colon  H\hookrightarrow G$ as the following commutative diagram with exact rows
    \[\xymatrix{
        0\ar[r] & S \ar[r] \ar[d]_j & H \ar[r] \ar[d]_i & B \ar[r] \ar[d]_f & 0\\
        0\ar[r] & T \ar[r]          & G \ar[r]          & A \ar[r]          & 0,
    }\]
    where $S$ and $T$ are tori, and $A$ and $B$ are abelian varieties. Using the snake lemma and noting that $i$ is a closed immersion, we obtain that $j$ is a closed immersion and that the connecting homomorphism $\delta \colon  \ker(f)\to \coker(j)$ is a closed immersion. The fact that $j$ is a closed immersion implies that $t\geq t_0$. Note that the reduced connected component $\ker(f)^0_\red$ of $\ker(f)$ is a smooth connected closed subgroup\footnote{Here, we used \cite[Chapitre II, \S5, Corollaire 2.3]{DG70AlgGrp} that if $G$ is a group scheme locally of finite type over a perfect field $k$, then $G_\red$ is a subgroup scheme of $G$. We shall show that for connected subgroups of semi-abelian varieties, it is unnecessary to assume $k$ to be perfect. See Proposition~\ref{SAVsubRed}.} of the abelian variety $A$. Using Lemma~\ref{SubQuotSAV} (2), we get that $\coker(j)$ is a torus and $\ker(f)^0_\red$ is an abelian variety. By Lemma~\ref{AffineAVtrivial}, the composition of closed immersions \[\ker(f)^0_\red\hookrightarrow\ker(f)\stackrel{\delta}{\hookrightarrow}\coker(j)\] is trivial. Thus $\ker(f)^0_\red$ is trivial and $\ker(f)$ has dimension $0$, which implies that $a_0\leq a$.
\end{proof}

We will reduce some problems about semi-abelian varieties to relevant ones of finite \'etale group schemes by using the following result.
\begin{prop}\label{SAVTorsionDense}
    Let $G$ be a semi-abelian variety over a field $k$. Then the collection of closed subschemes $\{{}_nG\}_{n\geq1,\Char(k)\nmid n}$ is (topologically) dense in $G$, where ${}_nG$ is the kernel of multiplication by $n$ on $G$.
\end{prop}
\begin{proof}
    It suffices to prove the assertion for $k=\overline k$ and from now on, we assume that $k$ is algebraically closed. Let $X\subset G(k)$ be the union of all ${}_nG(k)$, where $\Char(k)\nmid n$ and let $H$ be the reduced closed subscheme of $G$ whose underlying space is the Zariski closure of $X$. Then by our construction, it is easy to verify that $H$ is a subgroup scheme of $G$.

    The connected component $H^0$ of the unit is a smooth connected subgroup of the semi-abelian variety $G$. So $H^0$ is also semi-abelian by Lemma~\ref{SubQuotSAV}. Let $N$ be the number of connected components of $H$. Let $a$ (resp. $a_0$) be the dimension of the abelian part and let $t$ (resp. $t_0$) be the dimension of the toric part of $G$ (resp. $H^0$). By Lemma~\ref{nGOrder}, we have that
    \[\#{}_nG(k)=n^{2a+t} \text{\quad and \quad} \# {}_nH^0(k)=n^{2a_0+t_0},\]
    the second one of which implies $\#{}_nH(k)\leq Nn^{2a_0+t_0}$. By construction, $H$ contains all torsion points of $G$ of order prime to $\Char(k)$. So $\#{}_nH(k)=\#{}_nG(k)=n^{2a+t}$. Now, we have $n^{2a+t}\leq Nn^{2a_0+t_0}$ for every positive integer $n$ prime to $\Char(k)$. Taking $n$ very large, we obtain $2a+t\leq 2a_0+t_0$. By Lemma~\ref{DimSubSAV}, we have $a\geq a_0$ and $t\geq t_0$. Thus 
    \[a=a_0 \text{\quad and \quad} t=t_0.\] 
    So the irreducible variety $G$ has the same dimension as the closed subvariety $H^0$. Hence $H^0=G$.
\end{proof}

\begin{lem}\label{redSubgroup}
    Let $H$ be a group scheme over a field $k$. If $H_\red$ is geometrically reduced, then $H_\red$ is a closed subgroup scheme of $H$.
\end{lem}
\begin{proof}
    The argument of \cite[Chapitre II, \S5, Corollaire 2.3]{DG70AlgGrp} works here. For readers' convenience, we repeat it: Since $H_\red$ is geometrically reduced, the scheme $H_\red\times_k H_\red$ is reduced. Thus the restriction of the multiplication law $m \colon  H\times_k H\to H$ to $H_\red\times_k H_\red$ factors through $H_\red\hookrightarrow H$: 
    \[H_\red\times_k H_\red \stackrel{m_\red}{\longrightarrow} H_\red.\]
    Similarly, the unit morphism and the inverse morphism of $H$ induce morphisms on $H_\red$, and it follows that $(H_\red,m_\red)$ is a closed subgroup of $(H,m)$.
\end{proof}

\begin{prop}\label{SAVsubRed}
    Let $G$ be a semi-abelian variety over a field $k$ and $H$ be a connected closed subgroup of $G$. Then 
    \begin{enumerate}[label={\rm(\arabic*)}]
        \item the collection of closed subschemes $\{{}_nH\}_{n\geq1,\Char(k)\nmid n}$ is (topologically) dense in $H$;
        \item $H_\red$ is geometrically reduced;
        \item $H_\red$ is a semi-abelian subvariety of $G$.
    \end{enumerate} 
\end{prop}
\begin{proof}
    Let $n$ be a positive integer with $\Char(k)\nmid n$. By Lemma~\ref{nGFiniteEtale}, the commutative group scheme ${}_nG$ is finite  \'etale over $k$. Since ${}_nH$ is a closed subgroup of ${}_nG$, it is also finite \'etale over $k$. Let $H_0$ be the reduced closed subscheme of $G$ whose underlying space is the Zariski closure of $\bigcup_{\Char(k)\nmid n} {}_nH$. By \cite[Corollaire 11.10.7]{EGAIV3}, $(H_0)_{\overline k}$ is the reduced closed subscheme of $G_{\overline k}$ whose underlying space is the Zariski closure of $\bigcup_{\Char(k)\nmid n} {}_nH_{\overline k}$. Since $(H_{\overline k})_\red$ is a smooth connected closed subgroup of $G_{\overline k}$, it is a semi-abelian subvariety of $G_{\overline k}$. By Proposition~\ref{SAVTorsionDense}, the collection $\{{}_n ((H_{\overline k})_\red)\}_{\Char(k)\nmid n}$ is topologically dense in $(H_{\overline k})_\red$. Since $(H_{\overline k})_\red$ is a closed subgroup of $H_{\overline k}$, we have that ${}_n((H_{\overline k})_\red)$ is a closed subgroup of ${}_nH_{\overline k}$, which implies that as topological spaces
    \[(H_{\overline k})_\red\subseteq (H_0)_{\overline k}\subseteq H_{\overline k}.\]
    Because $H_{\overline k}$ and $(H_{\overline k})_\red$ have the same underlying topological space, the above three topological spaces are the same. Thus
    \[\dim H_0=\dim (H_0)_{\overline k}=\dim H_{\overline k}=\dim H.\]
    Because $H$ is irreducible, we conclude that $H_0=H$ as topological spaces, which completes the proof of the first assertion. 
 
    Clearly, $H_0=H_\red$ as schemes. Since $(H_0)_{\overline k}$ is reduced, the scheme $H_\red$ is geometrically reduced.
    
    By Lemma~\ref{redSubgroup}, $H_\red$ is a closed subgroup of $H$. Since $H_\red$ is a smooth connected closed subgroup of the semi-abelian variety $G$, it is a semi-abelian subvariety by Lemma~\ref{SubQuotSAV}. 
\end{proof}

\subsection{Base change and descent of Deligne \texorpdfstring{$1$}{1}-motives}
In the case of primary field extensions, we show that the extension of scalars of Deligne $1$-motives is fully faithful. We will also prove some descent results on Deligne $1$-motives.

\begin{defn}[{\cite[10.1.10]{Deligne74HodgeIII}}]
    \begin{enumerate}[leftmargin=*,label={\rm(\arabic*)}]
        \item A Deligne $1$-motive over $k$ is a complex of commutative group schemes
            \[M=[L\stackrel{u}{\to} G],\]
            where $L$ is a lattice and $G$ is a semi-abelian variety.\menum
        \item A morphism of Deligne $1$-motives from $M=[L\stackrel{u}{\to} G]$ to $M'=[L' \stackrel{u'}{\to} G']$ is a commutative square 
            \[\xymatrix{
                L  \ar[r]^u \ar[d]_f & G \ar[d]^g \\
                L' \ar[r]^{u'} & G'
            }\]
            in the category of group schemes. Denote by $(f,g) \colon  M\to M'$ such a morphism.
        \item Denote the category of Deligne $1$-motives over $k$ by $\M_1(k)$.
    \end{enumerate}
\end{defn}

\begin{defn}
    \begin{enumerate}[label={\rm(\arabic*)},leftmargin=*]
        \item An extension of fields $K/k$ is called primary if the algebraic closure of $k$ in $K$ is purely inseparable over $k$. \menum
        \item An extension of fields $K/k$ is called regular if $K/k$ is separable and $k$ is algebraically closed in $K$.
    \end{enumerate}
\end{defn}
\begin{rmk}
    Clearly, regular extensions are primary. If $k$ is perfect, then any primary extension of $k$ is automatically regular.
\end{rmk}

Let $K/k$ be a field extension, $K_s$ be a separable closure of $K$ and $k_s$ be the separable closure of $k$ in $K_s$. Denote $\Gal(K_s/Kk_s)$ by $\Gamma$.
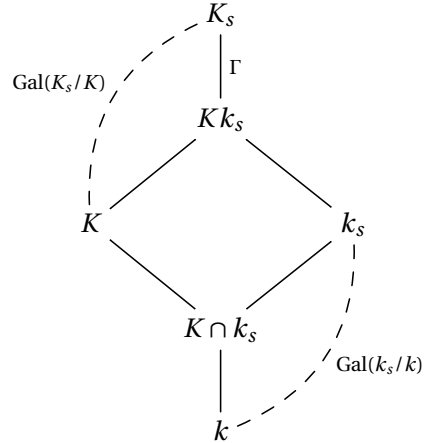
\[\xymatrix{
    & K_s  \ar@/_1.5pc/@{--}[ddl]_{\Gal(K_s/K)} \ar@{-}[d]^\Gamma   &\\
    & Kk_s \ar@{-}[dl]  \ar@{-}[dr] & \\
    K \ar@{-}[dr] && k_s \ar@{-}[dl] \ar@/^1.5pc/@{--}[ddl]^{\Gal(k_s/k)}\\
    & K\cap k_s \ar@{-}[d] &\\
    &k& 
}\]
If $K/k$ is primary, then $K\cap k_s=k$ and the restriction map 
\[\pi \colon  \Gal(K_s/K) \twoheadrightarrow \Gal(Kk_s/K) \stackrel{\simeq}{\to} \Gal(k_s/k),\quad \sigma\mapsto \sigma|_{Kk_s}\mapsto\sigma|_{k_s}\]
is continuous and surjective with kernel $\Gamma$. 

\begin{prop}\label{M0FullFaithAdj}
    Let $K/k$ be a primary extension. 
    \begin{enumerate}[label={\rm(\arabic*)}]
        \item The extension of scalars $\pi^* \colon  {}^t\M_0(k)\to {}^t\M_0(K)$ is fully faithful. The same result holds for $\M_0$ and $\Tori$.
        \item The extension of scalars $\pi^* \colon  {}^t\M_0(k)\to {}^t\M_0(K)$ has a right adjoint $\pi_*^{{}^t\M_0}$ and a left adjoint $\pi_\sharp^{{}^t\M_0}$. The same result holds for $\M_0$ and $\Tori$.
    \end{enumerate}
\end{prop}
\begin{proof}
    Thanks Proposition~\ref{EtaleModule}, we identify commutative \'etale group schemes with the associated discrete Galois modules.

    Because the restriction map 
    \[\pi \colon  \Gal(K_s/K) \twoheadrightarrow \Gal(Kk_s/K) \stackrel{\simeq}{\to} \Gal(k_s/k),\quad \sigma\mapsto \sigma|_{Kk_s}\mapsto\sigma|_{k_s}\]
    is surjective, it induces a fully faithful functor
    \begin{align*}
        \pi^* \colon  {}^t\M_0(k)    &\longrightarrow          {}^t\M_0(K)\\
        L                    &\longmapsto   L,
    \end{align*}
    which corresponds to the extension of scalars of constructible group schemes. The essential image of $\pi^*$ is the full subcategory of modules on which $\Gamma$ acts trivially. The functor $\pi^*$ admits a right adjoint
    \begin{align*}
        \pi_* \colon  {}^t\M_0(K)    &\longrightarrow         {}^t\M_0(k)\\
        L                    &\longmapsto  L^\Gamma,
    \end{align*}
    and a left adjoint 
    \begin{align*}
        \pi_\sharp \colon  {}^t\M_0(K)  &\longrightarrow         {}^t\M_0(k)\\
                L               &\longmapsto  L_\Gamma, 
    \end{align*}
    where $L^\Gamma\subset L$ denotes the submodule of $\Gamma$-invariants and $L_\Gamma \colonequals L/\langle hx-x\mid h\in \Gamma,x\in L\rangle$ is the space of $\Gamma$-coinvariants.

    Restricting to lattices, we have a fully faithful functor $\pi^* \colon  \M_0(k)\to \M_0(K)$. Since $L^\Gamma$ is a lattice over $k$ for any $L\in\M_0(K)$, the functor $\pi^*$ admits a right adjoint $\pi_* \colon  L\mapsto L^\Gamma$. On the other hand, for any $L\in \M_0(K)$, the $\Gal(k_s/k)$-module $L_\Gamma$ can be represented by a unique extension
    \[0\to (L_\Gamma)_\tor \to L_\Gamma \to (L_\Gamma)_\free \to 0,\]
    where $(L_\Gamma)_\tor$ is a finite abelian group with a continuous $\Gal(k_s/k)$-action and $(L_\Gamma)_\free$ is an object of $\M_0(k)$. Since there exists no non-trivial homomorphism from $(L_\Gamma)_\tor$ to any $L'\in\M_0(k)$, we can see that 
    \[\Hom_{\M_0(k)}\left((L_\Gamma)_\free,L'\right)\simeq\Hom_{\M_0(k)}(L_\Gamma,L')\simeq \Hom_{\M_0(K)}(L,\pi^*L').\]
    In other words, the functor $L\mapsto (L_\Gamma)_\free$ is the left adjoint of $\pi^* \colon  \M_0(k)\to \M_0(K)$.

    By Cartier duality (Theorem~\ref{CartierDual}), we get the assertions for tori from the ones for lattices.
\end{proof}

Denote the category of abelian varieties (resp. semi-abelian varieties) by $\AV(k)$ (resp. $\SAV(k)$), where morphisms are the homomorphisms between group schemes.
\begin{thm}[Chow]\label{AVfullyfaithful}
    Let $K$ be a primary field extension of $k$. Then the extension of scalars
    \begin{align*}
        \pi^* \colon  \AV(k)  &\longrightarrow  \AV(K),\\
            A       &\longmapsto A_K
    \end{align*}
    is fully faithful.
\end{thm}
\begin{proof}
    For a modern proof using fpqc descent, see \cite[Theorem~3.19]{Conrad06ChowLN}.
\end{proof}

Deligne \cite[10.2.11--13]{Deligne74HodgeIII} defined a self-duality on the category $\M_1(k)$, that he called Cartier duality. Let $M=[L\stackrel{u}{\to} G]$ be a Deligne $1$-motive, and let $T$ and $A$ be the toric and abelian part of $G$ respectively. Then 
\[M^\vee=[T^\vee\to G^u],\]
where $G^u$ is an extension of $A^\vee$ by $L^\vee$.

\begin{lem}\label{DualBase}
    Cartier duality commutes with extension of scalars.
\end{lem}
\begin{proof}
    Following \cite[p.17]{BS01AlbPic-1-motive}, we use the symmetric avatar $(L,T^\vee,A,A^\vee,u,v,\psi)$ to denote a Deligne $1$-motive $[L\to G]$, where $T$ and $A$ are the toric and abelian part of $G$ respectively. The symmetric avatar of Cartier dual is $(T^\vee,L,A^\vee,A,v,u,\psi^t)$. Since all these $L,T^\vee,A,A^\vee,u,v,\psi,\psi^t$ are compatible with the extension of scalars, we conclude that Cartier duality commutes with the extension of scalars. 
\end{proof}

\begin{lem}\label{fpqcdescent}
    Let $\pi \colon  S'\to S$ be a faithfully flat morphism of schemes. 
    \begin{enumerate}[label={\rm(\arabic*)}]
        \item The base change functor
            \begin{align*} 
                \pi^* \colon  \Sch/S &\longrightarrow \Sch/S',\\
                      X      &\longmapsto \pi^*X \colonequals X\times_S S'
            \end{align*}
            is faithful, where $\Sch/S$ (resp. $\Sch/S'$) is the category of schemes over $S$ (resp. $S'$).
        \item Assume moreover that $\pi$ is quasi-compact. If $X$ is a group scheme over $S$ and $Z$ is a closed subscheme of $X$, then $Z$ is a subgroup scheme of $X$ if and only if $Z_{S'}$ is a subgroup scheme of $X_{S'}$.
    \end{enumerate} 
\end{lem}
\begin{proof}
    The first assertion is {\cite[Corollaire 2.2.16]{EGAIV2}}, and the second one is part of \cite[Theorem~3.5]{Conrad06ChowLN}. 
\end{proof}

\begin{thm}\label{M1fullyfaithful}
    Let $K$ be a primary field extension of $k$. Then the extension of scalars of Deligne $1$-motives
    \begin{align*}
        \pi^* \colon  \M_1(k)        &\longrightarrow \M_1(K)\\
            [L\to G]       &\longmapsto [L_K\to G_K]
    \end{align*}
    is fully faithful.
\end{thm}
\begin{proof}
    By Lemma~\ref{fpqcdescent} (1), it is clear that $\pi^*$ is faithful. Now, we prove that it is also full. Let $\M_1^\ab$ be the full subcategory of $\M_1$ whose objects are $[L\to A]$ with $A$ an abelian variety.
    \begin{enumerate}[leftmargin=*,label={\rm(\alph*)}]
        \item First, we show that the induced functor $\pi^* \colon  \M_1^\ab(k) \to\M_1^\ab(K)$ is full. Consider the morphisms of the form 
        \[(f,g) \colon  [L_K \stackrel{u_K}{\to} A_K]\longrightarrow [L'_K \stackrel{v_K}{\to} A'_K]\]
        where $A$ and $A'$ are abelian varieties over $k$. By Proposition~\ref{M0FullFaithAdj} (1) and Theorem~\ref{AVfullyfaithful}, there exist homomorphisms 
        \[f_0 \colon  L\to L'\text{ and }g_0 \colon  A\to A'\]
        such that $f$ (resp. $g$) is the base change of $f_0$ (resp. $g_0$). Since $(f,g)$ is a morphism of Deligne $1$-motives, we have that
        \[(vf_0)_K=v_K f=g u_K=(g_0 u)_K.\]
        By Lemma~\ref{fpqcdescent} (1), we obtain that \[v f_0=g_0 u.\] It means that $(f_0,g_0)$ is a morphism of Deligne $1$-motives, whose base change is $(f,g)$.
        \item By Lemma~\ref{DualBase}, the Cartier duality of $\M_1$ gives an anti-equivalence between $\M_1^\ab$ and $\SAV$ which commutes with extension of scalars. It follows from (a) that the extension of scalars $\pi^* \colon  \SAV(k)\to \SAV(K)$ is fully faithful.
        \item Now, consider $\pi^* \colon  \M_1(k)\to \M_1(K)$. Repeating the argument of (a) except for replacing Chow's Theorem~\ref{AVfullyfaithful} with (b), we conclude that the extension of scalars of Deligne $1$-motives is fully faithful. \qedhere
    \end{enumerate}
\end{proof}
\begin{rmk}
    In \cite[Theorem~1.2]{Yu19ChowSAV}, Yu uses the same strategy as in \cite[Theorem~3.19]{Conrad06ChowLN} to show that extension of scalars of semi-abelian varieties is fully faithful in the case of primary extension. Our result is a generalization of theirs, and our proof is an alternative to Yu's.
\end{rmk}

\begin{lem}
    Let $K/k$ be a primary field extension. Let $L$ be a discrete $\Gal(k_s/k)$-module and $L'$ be a discrete $\Gal(K_s/K)$-submodule of $L$. Then $L'$ is also a discrete $\Gal(k_s/k)$-submodule of $L$.
\end{lem}
\begin{proof}
    Since $\Gal(K_s/Kk_s)$ acts trivially on $L$, its action on the submodule $L'$ is also trivial. Thus $L'$ is a $\Gal(k_s/k)$-submodule of $L$.
\end{proof}

\begin{prop}\label{M0descent}
    Let $K/k$ be a primary field extension. Let $L$ be a commutative \'etale group scheme over $k$ and let $L'$ be a subgroup scheme of $L_K$. Then the closed immersion $i \colon  L'\hookrightarrow L_K$ is defined over $k$.
\end{prop}
\begin{proof}
    By the equivalence between discrete Galois modules and commutative \'etale group schemes (Proposition~\ref{EtaleModule}), this proposition is a reformulation of the above lemma.
\end{proof}

We have a similar result for semi-abelian varieties.
\begin{thm}\label{SAVdescent}
    Let $K/k$ be a primary field extension. Let $G$ be a semi-abelian variety over $k$ and $H$ be a semi-abelian subvariety of $G_K$. Then the closed immersion $i \colon  H\hookrightarrow G_K$ of semi-abelian varieties is defined over $k$.
\end{thm}
\begin{proof}
    Let $n$ be a positive integer with $\Char(k)\nmid n$. By Lemma~\ref{nGFiniteEtale}, the commutative group scheme ${}_nG$ is finite \'etale over $k$. Then ${}_nH$ is a finite  \'etale closed subgroup scheme of ${}_nG_K$. By Proposition~\ref{M0descent}, the closed immersion ${}_nH\hookrightarrow {}_nG_K$ is the base change of a closed immersion ${}_nH'\hookrightarrow {}_nG$ of finite \'etale group schemes over $k$. Let $H_0$ be the reduced closed subscheme of $G$ whose underlying space is the Zariski closure of the union of such ${}_nH'$'s. By \cite[Corollaire 11.10.7]{EGAIV3}, $(H_0)_K$ is the reduced closed subscheme of $G_K$ whose underlying space is the Zariski closure of $\bigcup_{\Char(k)\nmid n} {}_nH$. Thus we obtain that $(H_0)_K=H$ by Proposition~\ref{SAVsubRed}. By Lemma~\ref{fpqcdescent} (2), the subscheme $H_0$ is a closed subgroup of $G$. Moreover, the smoothness and connectedness of $H$ descend to $H_0$. By Lemma~\ref{SubQuotSAV}, $H_0$ is a semi-abelian subvariety of $G$.
\end{proof}

The following lemma is a simple corollary of the above theorem. It is a key ingredient in the construction of Chow image and Chow trace of Deligne $1$-motives.

\begin{lem}\label{M1CommaSurj}
    Let $K/k$ be a primary extension. By Proposition~\ref{M0FullFaithAdj}, $\pi^* \colon  \M_0(k)\to \M_0(K)$ admits a left adjoint $\pi_\sharp^{\M_0}$. Let $(f,g) \colon  [L\to G]\longrightarrow [L'\to G']_K$ be a morphism of Deligne $1$-motives over $K$. Then $(f,g)$ factors as
    \[\xymatrix{ [L\to G] \ar[rr]^-{(\varepsilon,g_0)} && [\pi_\sharp^{\M_0}L \to G_0]_K \ar[rr] && [L'\to G']_K,}\]
    where $\varepsilon \colon  L\to (\pi_\sharp^{\M_0}L)_K$ is the unit morphism and $g_0 \colon  G\to (G_0)_K$ is a surjection.
\end{lem}
\begin{proof}
    By Theorem~\ref{SAVdescent}, the morphism $g \colon  G\to G'_K$ factors as 
    \[G \stackrel{g_0}{\longrightarrow} (G_0)_K \stackrel{i_K}{\longrightarrow} G'_K,\]
    where $(G_0)_K$ is the image of $g$ in the abelian category of commutative algebraic groups over $K$. Let $L_0$ be the base change of $L'$ through the closed immersion $i \colon  G_0\hookrightarrow G'$. Then $L_0$ is a closed subgroup of $L'$ and thus is a lattice by Remark~\ref{SubM0}. Using the universal property of fiber products, we can see that the morphism of Deligne $1$-motives $(f,g)$ factors as
    \[\xymatrix{
        L \ar[r]^-{f_0} \ar[d]    & (L_0)_K \ar[r] \ar[d] & L_K'  \ar[d] \\
        G \ar[r]^-{g_0}           & (G_0)_K \ar[r]^{i_K}  & G_K'.
    }\]
    Then the homomorphism $f_0 \colon  L\to (L_0)_K$ factors through the unit morphism $\varepsilon \colon  L\to (\pi_\sharp^{\M_0}L)_K$, which completes the proof.
\end{proof}

\subsection{Chow image and Chow trace of Deligne \texorpdfstring{$1$}{1}-motives}
\begin{lem}\label{PreorderBound}
    Let $(P,\leq)$ be a pre-ordered set (i.e., $\leq$ is reflexive and transitive, but not necessarily anti-symmetric) satisfying the following conditions: 
    \begin{enumerate}[label={\rm(\arabic*)}]
        \item $P$ is filtered, i.e., for $x,y\in P$, there exists an element $z\in P$ with $x\leq z$ and $y\leq z$;
        \item there exists a subset $Q\subset P$ such that 
            \begin{enumerate}[label={\rm(\alph*)}]
                \item for any $x\in P$, there exists $y\in Q$ with $x\leq y$;
                \item $Q$ has a maximal element $m$ in the following sense: if $t\in Q$ and $m\leq t$, then $t\leq m$.
            \end{enumerate}
    \end{enumerate}
    Then $m$ is an upper bound of $P$, i.e., $x\leq m$ for all $x\in P$.
\end{lem}
\begin{proof}
    For any $x\in P$, there exists an element $y\in P$ with $m\leq y$ and $x\leq y$ by (1). Then there is an element $z\in Q$ such that $y\leq z$ by (2a). Thus $m\leq z$. It follows from (2b) that $z\leq m$. Hence $x\leq m$.
\end{proof}

\begin{thm}\label{M1adjoint}
    Let $K/k$ be a primary extension. Then the extension of scalars $\pi^* \colon  \M_1(k)\to \M_1(K)$ has a left adjoint $\pi_\sharp^{\M_1}$ and a right adjoint $\pi_*^{\M_1}$, called Chow image and Chow trace respectively.
\end{thm}
\begin{proof}
    Let $M=[L\to G]$ be a Deligne $1$-motive over $K$. Let $P$ be the set of morphisms of Deligne $1$-motives $\varphi \colon  M\to \pi^*N$ with $N$ a Deligne $1$-motive over $k$. Let $Q$ be the subset consisting of morphisms of the form
    \[(\varepsilon,g) \colon  [L\to G] \longrightarrow [\pi_\sharp^{\M_0} L\to G']_K,\] 
    with $g \colon  G\to G_K'$ surjective. Consider the following pre-order:  $\varphi'\leq \varphi$ if there exists a morphism $\psi \colon  N\to N'$ such that $\varphi'=\pi^*\psi\circ\varphi$. 
    
    The pre-ordered set $(P,\leq)$ satisfies all the conditions in Lemma~\ref{PreorderBound}:
    \begin{enumerate}[leftmargin=*,label={\rm(\arabic*)}]
        \item Let $\varphi_1 \colon  M\to \pi^*N_1$ and $\varphi_2 \colon  M\to \pi^*N_2$ be two elements in $P$. Then the induced map 
            \[(\varphi_1,\varphi_2) \colon  M\to (N_1)_K\times (N_2)_K \simeq (N_1\times N_2)_K\] 
            is a supremum of these two morphisms.
        \item \begin{enumerate}[label={\rm(\alph*)}]
                \item This is Lemma~\ref{M1CommaSurj}.
                \item Consider the set of closed subgroup schemes of $G$ which are of the form $\ker(g)$ for some $(\varepsilon,g)$ in $Q$. Because $G$ is a noetherian scheme, this set has a minimal element $m$ with respect to inclusion. Say $G/m=G'_K$, where $G'$ is a semi-abelian variety over $k$. Then the corresponding morphism $(\varepsilon,g) \colon  [L\to G] \to [\pi_\sharp^{\M_0}L\to G']_K$ is a maximal element of $Q$.
            \end{enumerate}
    \end{enumerate}
    Hence by Lemma~\ref{PreorderBound}, the set $P$ has an upper bound which is an element of $Q$. In other words, there exists a Deligne $1$-motive $N$ over $k$ and a morphism $\varphi \colon  M\to \pi^*N$ such that for every morphism $\varphi' \colon  M\to \pi^*N'$ there exists a morphism $\psi \colon  N\to N'$ such that $\varphi'=\pi^*\psi\circ\varphi$. Note that $\pi^*$ is fully faithful by Theorem~\ref{M1fullyfaithful}, and that such $\varphi$ has the form $(\varepsilon,g)$ with $\varepsilon$ the unit morphism and $g$ a surjection. Thus the morphisms $\psi$ satisfying $\varphi'=\pi^*\psi\circ\varphi$ is unique. It means that the functor $\pi^*$ has a left adjoint $\pi_\sharp^{\M_1}$.

    By Lemma~\ref{DualBase}, Cartier duality for Deligne $1$-motives commutes with extension of scalars. Thus the existence of the right adjoint is obvious by dualizing $\pi_\sharp(M^\vee)$ and using the dual of its universal morphism.
\end{proof}

\begin{cor}\label{M1adjIsom}
    Let $K/k$ be a primary field extension and $M\in\M_1(k)$. Then we have the following canonical isomorphisms
    \[\pi_\sharp^{\M_1} \pi^*M\stackrel{\sim}{\longrightarrow} M\]
    and \[M\stackrel{\sim}{\longrightarrow}\pi_*^{\M_1}\pi^*M.\] 
\end{cor}
\begin{proof}
    This result holds because $\pi^*$ is fully faithful according to Theorem~\ref{M1fullyfaithful}.
\end{proof}

The functors $\pi_\sharp^{\M_1}$ and $\pi_*^{\M_1}$ recover some classical constructions, such as Chow's image and trace of abelian varieties. 
From now on, we sometimes write them simply as $\pi_\sharp$ and $\pi_*$ respectively.
\begin{cor}\label{M1AdjExample}
    Let $K$ be a primary field extension of $k$.
    \begin{enumerate}[label={\rm(\arabic*)}]
        \item For any $L\in\M_0(K)$, we have 
            \[\pi_\sharp([L\to 0])=[\pi_\sharp^{\M_0}(L)\to 0] \text{\quad and \quad} \pi_*([L\to 0])=[\pi_*^{\M_0}(L)\to 0],\]
            where $\pi_\sharp^{\M_0}$ and $\pi_*^{\M_0}$ are left and right adjoints to extension of scalars of lattices in Proposition~\ref{M0FullFaithAdj}.
        \item The extension of scalars $\pi^* \colon  \SAV(k)\to \SAV(K)$ has a left adjoint $\pi_\sharp^\SAV$ and a right adjoint $\pi_*^\SAV$. Moreover, for any $G\in\SAV(K)$, we have 
            \[\pi_\sharp([0\to G])=[0\to \pi_\sharp^\SAV(G)] \text {\quad and \quad } \pi_*([0\to G])=[0\to \pi_*^\SAV(G)].\]
        \item The extension of scalars $\pi^* \colon  \AV(k)\to \AV(K)$ has a left adjoint $\pi_\sharp^\AV$ and a right adjoint $\pi_*^\AV$. Moreover, for any $A\in\AV(K)$, we have 
            \[\pi_\sharp([0\to A])=[0\to \pi_\sharp^\AV(A)] \text{\quad and \quad} \pi_*([0\to A])=[0\to \pi_*^\AV(A)].\]
        \item The extension of scalars $\pi^* \colon  \Tori(k)\to \Tori(K)$ has a left adjoint $\pi_\sharp^\Tori$ and a right adjoint $\pi_*^\Tori$. Moreover, for any $T\in\Tori(K)$, we have 
            \[\pi_\sharp([0\to T])=[0\to \pi_\sharp^\Tori(T)] \text{\quad and \quad} \pi_*([0\to T])=[0\to \pi_*^\Tori(T)].\]
        \item For any $G\in\SAV(K)$, the abelian part of $\pi_\sharp^\SAV(G)$ is isomorphic to the Chow image of the abelian part of $G$, and the torus part of $\pi_*^\SAV(G)$ is isomorphic to the Chow trace of the torus part of $G$.
    \end{enumerate}
\end{cor}
\begin{proof}
    Keep the notations in the proof of Theorem~\ref{M1adjoint}. Then for any Deligne $1$-motive $M$ over $K$, the unit morphism $\varphi \colon  M\to \pi^*\pi_\sharp(M)$ is a maximal element of $Q$. By definition, for $M=[L\to 0]$ with $L\in\M_0(K)$, the elements of $Q$ are the morphisms 
    \[(\varepsilon,g) \colon  [L\to 0]\longrightarrow[\pi_\sharp^{M_0}L\to G']_K,\]
    where $g$ is surjective. So $G'=0$ and $Q$ has a unique element $(\varepsilon,0) \colon  [L\to 0]\to [\pi_\sharp^{\M_0}(L)\to 0]_K$. So
    \[\pi_\sharp([L\to 0])=[\pi_\sharp^{\M_0}(L)\to 0].\]
    
    Similarly, for $M=[0\to G]$ with $G\in\SAV(K)$, the elements of $Q$ are the morphisms 
    \[(\varepsilon,g) \colon  [0\to G]\longrightarrow[0\to G']_K,\]
    where $g$ is surjective. Thus the left adjoint $\pi_\sharp(M)$ has the form $[0\to G_0]$ for some $G_0\in\SAV(k)$ and the canonical morphism $G\to \pi^*G_0$ is surjective. So $\pi^* \colon  \SAV(k)\to \SAV(K)$ has a left adjoint $\pi_\sharp^\SAV$ and we have 
    \[\pi_\sharp([0\to G])=[0\to \pi_\sharp^\SAV(G)].\]
    By Lemma~\ref{SubQuotSAV}, if $G$ is an abelian variety (resp. a torus), then $\pi_\sharp^\SAV(G)$ is an abelian variety (resp. a torus) because the canonical morphism $G\to \pi^*\pi_\sharp^\SAV(G)$ is surjective. Hence $\pi^* \colon  \AV(k)\to \AV(K)$ has a left adjoint $\pi_\sharp^\AV$, and for an abelian variety $A/K$, we have that
    \[\pi_\sharp([0\to A])=[0\to \pi_\sharp^\AV(A)].\]
    The same results hold for tori.
    
    By Cartier duality, we can obtain the assertions on the right adjoints.

    The last assertion follows from Remark \ref{SAV} (1) and the definition of Chow image and Chow trace.
\end{proof}
\begin{rmk}
    The functors $\pi_\sharp^\AV$ and $\pi_*^\AV$ are Chow's $K/k$-image and $K/k$-trace of abelian varieties. So our result recovers Chow's image and trace. This justifies the name of Chow image and Chow trace of Deligne $1$-motives.
\end{rmk}

In the remaining part of this section, we study the Chow image and Chow trace of a ``nontrivial'' Deligne $1$-motive $[L \to A]$ with $A$ an abelian variety. The case when $A$ has trivial Chow trace is quite simple.
\begin{cor}\label{ZtoAtrivial}
    Let $K/k$ be a primary field extension. Let $A$ be an abelian variety over $K$ with $\pi_\sharp(A)=0$ (equivalently $\pi_*(A)=0$ by \cite[Theorem 6.9]{Conrad06ChowLN}). Then 
    $$\pi_\sharp([L \to A])=[\pi_\sharp L \to 0] \text{\quad and \quad} \pi_*([L \stackrel{u}{\to} A])=[\pi_*\ker(u) \to 0].$$
\end{cor}
\begin{proof}
    Keep the notations in the proof of Theorem~\ref{M1adjoint}. For $[L \to A]$ over $K$, the elements of $Q$ are the morphisms 
    $$(\varepsilon,g) \colon  [L \to A] \longrightarrow [\pi_\sharp L \to G]_K,$$
    where $g$ is surjective. Since $g\colon A \to G_K$ factors through the unit $A \to (\pi_\sharp A)_K$, we get that $g=0$ by the assumption that $\pi_\sharp(A)=0$. Thus $G=0$ and $Q$ has a unique element $(\varepsilon,0) \colon  [L \to A]\to [\pi_\sharp L \to 0]_K$. So
    $$\pi_\sharp([L \to A])=[\pi_\sharp L \to 0].$$

    Now, we compute its Chow trace. Write $\pi_*([L \stackrel{u}{\to} A])=[L_0 \to G_0]$. For any semi-abelian variety $G'$ over $k$, we have that 
    \begin{align*}
        \Hom_{\SAV(k)}(G',G_0) &\simeq \Hom_{\M_1(k)}([0 \to G'],[L_0 \to G_0]) \\
                               &\simeq \Hom_{\M_1(K)}([0 \to G'_K],[\bZ \to A]) \\
                               &\simeq \Hom_{\SAV(K)}(G'_K,A).
    \end{align*}
    It implies that $G_0\simeq \pi_*(A)=0$. For any lattice $L'$ over $k$, we have that 
    \begin{align*}
        \Hom_{\M_0(k)}(L',L_0) &\simeq \Hom_{\M_1(k)}([L' \to 0],[L_0 \to G_0]) \\
                               &\simeq \Hom_{\M_1(K)}([L'_K \to 0],[L \stackrel{u}{\to} A]) \\
                               &\simeq \Hom_{\M_0(K)}(L'_K,\ker(u)).
    \end{align*}
    It follows that $L_0=\pi_*\ker(u)$. 
\end{proof}

A more interesting case is when $K$ is the function field of a smooth and geometrically connected variety $X/k$ and $A$ is defined over $k$. We shall use the Albanese scheme $\sA_{X/k}$ \cite[\S 1.3]{Ramachandran01AlbPic1Motive}. By a semi-abelian group scheme, we mean a group scheme locally of finite type over $k$ whose neutral component is a semi-abelian variety. The Albanese scheme $\sA_{X/k}$ is the universal semi-abelian group scheme with a morphism $\bZ(X) \to \sA_{X/k}$, where $\bZ(X)$ is the presheaf $U \mapsto \bZ[\Mor_k(U,X)]$. The neutral component $\sA_{X/k}^0$ is Serre's generalized Albanese semi-abelian variety of $X/k$. For a fixed rational point $x\in X(k)$, Serre's Albanese semi-abelian variety is the universal one  with a  morphism $X \to \sA_{X/k}^0$ mapping $x$ to $0$, which implies that the canonical map 
\[ \Hom_k(\sA_{X/k},A) \to \Hom_k(\sA_{X/k}^0,A) \]
is surjective.

\begin{lem}\label{AlbLN}
    Let $k$ be a field and let $K$ be the function field of a smooth and geometrically connected variety $X/k$ with a rational point $x\in X(k)$. Let $A$ be an abelian variety over $k$. Then we have isomorphisms
    \[ \Hom_k(\sA_{X/k},A) \simeq A(X) \simeq A(K), \]
    and
    \[ \Hom_k(\sA^0_{X/k},A) \simeq A(K)/A(k). \]
    In particular, 
    \begin{enumerate}[label={\rm(\arabic*)}]
        \item $A(K)/A(k)$ is a finitely generated free abelian group;
        \item a rational point in $A(K)$ belongs to $A(k)$ if and only if the corresponding homomorphism $\sA_{X/k}^0 \to A$ is trivial.
    \end{enumerate}
\end{lem}
\begin{proof}
    The isomorphism $\Hom_k(\sA_{X/k},A)\simeq A(X)$ is the definition of Albanese scheme. We have $A(X)\simeq A(K)$ by the valuative criterion of properness and Weil's extension theorem (\cite[\S 4.4, Theorem 1]{BLR90Neron}). Applying $\Hom_k(-,A)$ to the short exact sequence
    \[ 0 \to \sA_{X/k}^0 \to \sA_{X/k} \to \bZ \to 0, \]
    we get an exact sequence 
    \[ 0 \to \Hom_k(\bZ,A) \to \Hom_k(\sA_{X/k},A) \to \Hom_k(\sA_{X/k}^0,A) \to 0. \]
    The first term is isomorphic to $A(k)$ and the second term is isomorphic to $A(K)$. Thus the third term is isomorphic to $A(K)/A(k)$. By Lemma \ref{AffineAVtrivial}, the third term is also isomorphic to the Hom-group from the abelian part of $\sA_{X/k}^0$ to $A$. Recall \cite[Theorem 12.5]{Milne86AV} that the Hom-groups between abelian varieties are always finitely generated and free. We are done.
\end{proof}

\begin{prop}\label{ZtoAAlb}
    Let $k$ be a field and let $K$ be the function field of a smooth and geometrically connected variety $X/k$ with a rational point $x\in X(k)$. Let $A$ be an abelian variety over $k$. Then 
    \[ \pi_*([\bZ \stackrel{u}{\to} A_K])= \left\{
        \begin{array}{ll}
            {[\bZ \stackrel{u}{\to} A]}, & \hbox{if \( u(1)\in A(k) \);} \\
            {[0 \to A]}, & \hbox{if \( u(1) \notin A(k) \).} \\
        \end{array}\right.
    \]
    and 
    \[ \pi_\sharp([\bZ \stackrel{u}{\to} A_K])=[\bZ \stackrel{v}{\to} B], \]
    where $B$ is the cokernel of $u(1) \colon \sA_{X/k}^0 \to A$, and $v \colon \bZ \to B$ is the $k$-homomorphism corresponding to the trivial map $\sA_{X/k}^0 \to B$.
\end{prop}
\begin{proof}
    First, we compute the Chow trace. The case when $u(1)\in A(k)$ is a special case of Corollary \ref{M1adjIsom}. Suppose now $u(1) \notin A(k)$. We claim that morphisms $(f,g) \colon [L_K \stackrel{v_K}{\to} G_K] \to [\bZ \stackrel{u}{\to} A_K]$ all satisfy $f=0$. Otherwise, the image of $f \colon L_K \to \bZ$ will be $n\bZ$ for some positive integer $n$. By Theorem \ref{M1fullyfaithful}, the morphism $g \colon G_K \to A_K$ is defined over $k$, which implies that $nu(1)\in A(k)$. Thus $u(1)$ is a torsion element of the abelian group $A(K)/A(k)$. By Lemma \ref{AlbLN}, we get $u(1)\in A(k)$, a contradiction. Now, it is clear from the claim that $[0 \to A]$ is the Chow trace.

    Finally, we compute the Chow image. Keep the notations in the proof of Theorem \ref{M1adjoint}. The elements of $Q$ are the morphisms 
    $$(\id,g_K) \colon  [\bZ \stackrel{u}{\to} A_K] \longrightarrow [\bZ \stackrel{w_K}{\to} G_K].$$
    Here $g\colon A\to G$ is a surjective homomorphism; we implicitly used Theorem \ref{M1fullyfaithful} to say that $g$ is defined over $k$. As a quotient of an abelian variety, $G$ must also be an abelian variety. Since $g(u(1))=w_K(1)$ comes from $w(1)\in G(k)$, the composition
    \[ \xymatrix{\sA_{X/k}^0 \ar[r] &\sA_{X/k} \ar[r]^-{u(1)} &A \ar[r]^g &G} \]
    is trivial by Lemma \ref{AlbLN}. Thus $g$ factors through $B$, which completes the proof.
\end{proof}

\section{Direct and inverse images of \texorpdfstring{$n$}{n}-motivic sheaves}
In the remaining part of this chapter, we will study the Chow trace of Deligne $1$-motives in the framework of Voevodsky's triangulated categories of (\'etale) motives. In this section, we study Voevodsky's category of homotopy invariant sheaves (\cite{Voevodsky00DM}, \cite{MVW06Motive}) and some subcategories defined in \cite{ABV09MotShvLAlb}. We are mainly interested in the direct and inverse images of such sheaves.

Throughout this section, $k$ is a field of exponential characteristic $p$, i.e., $p=1$ if $\Char(k)$ is zero, and $p=\Char(k)$ otherwise. Let $\Lambda$ be the ring $\bZ[\frac1p]$. 
\subsection{Presheaves with transfers}
Let $\Sm/k$ be the category of smooth separated schemes of finite type over $k$. Recall Voevodsky's category of finite correspondences \cite[Lecture~1]{MVW06Motive}: Given $X,Y\in\Sm/k$, an elementary correspondence from $X$ to $Y$ is an integral closed subschemes $W$ of $X\times Y$ which is finite and surjective over a connected component of $X$. We denote by $\Corr_k(X,Y)$ the group of finite correspondences, i.e., the free abelian group generated by the elementary correspondences. Given elementary correspondences $V\in \Corr_k(X,Y)$ and $W\in\Corr_k(Y,Z)$, the composition $W\circ V$ is defined to be the pushforward of the intersection product $(V\times Z)\cdot(X\times W)$ of the corresponding cycles in $X\times Y\times Z$, along the projection $p \colon  X\times Y\times Z\to X\times Z$. Here, the intersection product and the pushforward of cycles are defined in \cite{Fulton98Intersection}. See \cite[p. 4]{MVW06Motive} for the verification that $W\circ V$ is a finite correspondence from $X$ to $Z$. Extending this composition linearly, we get the composition of arbitrary finite correspondences, which is associative and bilinear and has $\Delta_X$ as the identity of $\Corr_k(X,X)$. Let $\Cor(k)$ be the (additive) category whose objects are the same as $\Sm/k$ and whose morphisms from $X$ to $Y$ are elements of $\Corr_k(X,Y)$. The graph of a morphism yields a functor $\gamma_k \colon  \Sm/k\to \Cor(k)$. We consider the category $\PST(k,\Lambda)$ of presheaves with transfers of $\Lambda$-modules on $\Sm/k$, i.e., the category of additive contravariant functors from $\Cor(k)$ to the category of $\Lambda$-modules. For $X\in\Sm/k$, we denote by $\Lambda_\tr(X)$ the presheaf with transfers 
\[\Lambda_{\tr}(X)(U) \colonequals \Corr_k(U,X)\otimes_\bZ\Lambda.\]

Let $K/k$ be a field extension. Then we have an obvious extension of scalars functor $e \colon  \Cor(k)\to \Cor(K)$ taking $X$ to $X_K$ and $Z\in\Corr_k(X,Y)$ to $Z_K\in\Corr_K(X_K,Y_K)$. It induces a direct image functor 
\[e_*^\PST \colon  \PST(K)\to \PST(k),~\sF\mapsto \sF\circ e.\]
The functor $e_*^\PST$ is clearly exact.

\begin{prop}\label{e^*PST}
    \begin{enumerate}[label={\rm(\arabic*)},leftmargin=*]
        \item The functor $e_*^\PST$ admits a left adjoint $e^*_\PST$; \menum
        \item $e^*_\PST(\Lambda_\tr(X))=\Lambda_\tr(X_K)$;
        \item The functor $e^*_\PST$ is exact.
    \end{enumerate}
\end{prop}
\begin{proof}
    Everything is formal except the left exactness in (3). See \cite[Proposition~1.1 and Theorem~4.1]{Suslin17Nonperfect}.
\end{proof}

\subsection{\'Etale sheaves with transfers}\label{EST}
Recall that a presheaf with transfers $\sF$ is called an \'etale sheaf with transfers if its underlying presheaf $\sF\circ\gamma$ is an \'etale sheaf on $\Sm/k$. We denote by $\Shv_\et(\Sm/k,\Lambda)$ the category of \'etale sheaves of $\Lambda$-modules on $\Sm/k$, and denote by $\Shv_\et^\tr(k,\Lambda)$ the full subcategory of $\PST(k,\Lambda)$ whose objects are the \'etale sheaves with transfers. By \cite[Lemma~6.2]{MVW06Motive}, $\Lambda_\tr(X)$ is an \'etale sheaf with transfers.

\begin{prop}\label{Shv_et^tr}
    The category of \'etale sheaves with transfers has the following properties:
    \begin{enumerate}[label={\rm(\arabic*)}]
        \item The inclusion functor $\Shv_\et^\tr(k,\Lambda)\hookrightarrow \PST(k,\Lambda)$ has an exact left adjoint 
        \[a_\et \colon  \PST(k,\Lambda) \longrightarrow \Shv_\et^\tr(k,\Lambda).\] \menum
        \item The category $\Shv_\et^\tr(k,\Lambda)$ is a Grothendieck abelian category generated by the sheaves $\Lambda_\tr(X)$.
        \item The forgetful functor $\gamma_* \colon  \Shv_\et^\tr(k,\Lambda)\to \Shv_\et(\Sm/k,\Lambda)$ is conservative and commutes with all small limits and colimits.
        \item The functor $\gamma_*$ admits a left adjoint $\gamma^* \colon  \Shv_\et(\Sm/k,\Lambda) \to \Shv_\et^\tr(k,\Lambda)$.
    \end{enumerate}
\end{prop}
\begin{proof}
    See {\cite[6.18 and the proof of 6.19]{MVW06Motive}}; see also \cite[10.3.3, 10.3.9, 10.3.11]{CD19MixedMotive}.
\end{proof}

We consider some categories introduced in \cite{ABV09MotShvLAlb}. For $n\in\bN$, we denote by $(\Sm/k)_{\leq n}$ the full subcategory of $\Sm/k$ whose objects are the smooth schemes over $k$ of dimension less than or equal to $n$. Similarly, we denote by $\Cor(k_{\leq n})$ the full subcategory of $\Cor(k)$ having the same objects as $(\Sm/k)_{\leq n}$. We consider the $\Lambda$-additive dual $\PST(k_{\leq n},\Lambda)$ of $(\Sm/k)_{\leq n}$. As above, we have the notion of \'etale sheaves with transfers on $(\Sm/k)_{\leq n}$. We denote by $\Shv_\et^\tr(k_{\leq n},\Lambda)$ the full subcategory of $\PST(k_{\leq n},\Lambda)$ whose objects are the \'etale sheaves with transfers. For $X\in(\Sm/k)_{\leq n}$, we denote by $\Lambda_{\leq n}(X)$ the presheaf with transfers 
\[\Lambda_{\leq n}(X)(U) \colonequals \Corr_k(U,X)\otimes_\bZ\Lambda,\text{ where }U\in (\Sm/k)_{\leq n}.\]

Let $K/k$ be a field extension. Then we have the following commutative diagram
\begin{center}
    \begin{tikzcd}[scale cd=0.8]
        (\Sm/k)_{\leq n} \arrow[rr] \arrow[rd] \arrow[dd] &                                                                                  & \Sm/k \arrow[rd] \arrow[dd]   &                                                                                \\
                                                                  & \Cor(k_{\leq n}) \arrow[rr,crossing over]                                &                                                           & \Cor(k)                                            \\
        (\Sm/K)_{\leq n} \arrow[rd] \arrow[rr]            &                                                                                  & \Sm/K \arrow[rd]              &                                                                                \\
                                                                  & \Cor(K_{\leq n}) \arrow[rr] \arrow[from=uu, crossing over]               &                                                           & \Cor(K)  \arrow[from=uu, crossing over]              
    \end{tikzcd}
\end{center}
where the vertical arrows are base change functors induced by the morphism $\Spec K\to \Spec k$, the horizontal arrows are inclusions, and the arrows towards the lower right are graph functors. Note that the four functors on the back side are continuous functors for the \'etale topology in the sense of \cite[Expos\'e~III, D\'efinition 1.1]{SGA4I}, i.e.,  the corresponding direct images of sheaves are still sheaves. So the above diagram induces the following commutative diagram
\begin{center}
    \begin{tikzcd}[scale cd=0.8]
        \Shv_\et^\tr(K,\Lambda) \arrow[rd] \arrow[dd] \arrow[rr]  &                                                                                  & \Shv_\et(\Sm/K,\Lambda) \arrow[rd] \arrow[dd]  &                                     \\
                                                                  & \Shv_\et^\tr(K_{\leq n},\Lambda)  \arrow[rr,crossing over]                       &                                                & \Shv_\et((\Sm/K)_{\leq n},\Lambda) \arrow[dd] \\
        \Shv_\et^\tr(k,\Lambda) \arrow[rd] \arrow[rr]             &                                                                                  & \Shv_\et(\Sm/k,\Lambda) \arrow[rd]             &                                     \\
                                                                  & \Shv_\et^\tr(k_{\leq n},\Lambda) \arrow[rr]    \arrow[from=uu, crossing over]    &                                                & \Shv_\et((\Sm/k)_{\leq n},\Lambda)           
    \end{tikzcd}
\end{center}
where the vertical arrows $e_*$ are direct images of sheaves, the horizontal arrows $\gamma_*$ are forgetful functors, and the arrows towards the lower right $\sigma_{n*}$ are `restriction functors'. In particular, $\Lambda_{\leq n}(X)=\sigma_{n*}\Lambda_\tr(X)$ is an \'etale sheaf with transfers for $X\in(\Sm/k)_{\leq n}$. Noting that the inclusion functor $(\Sm/k)_{\leq n} \hookrightarrow \Sm/k$ is also co-continuous. Thus $\sigma_{n*}$ is exact.

\begin{lem}[{\cite[Lemma~1.1.12]{ABV09MotShvLAlb}}]\label{sigma^*}
    The functor $\sigma_{n*} \colon  \Shv_\et^\tr(k,\Lambda) \to \Shv_\et^\tr(k_{\leq n},\Lambda)$ has a left adjoint 
    \begin{align*}
        \sigma_n^* \colon  \Shv_\et^\tr(k_{\leq n},\Lambda) &\longrightarrow         \Shv_\et^\tr(k,\Lambda) \\
            \sF                                    &\longmapsto  \varinjlim_{\Lambda_{\leq n}(X)\to \sF} \Lambda_\tr(X),
    \end{align*}
    where the colimit is computed in $\Shv_\et^\tr(k,\Lambda)$.
\end{lem}

\begin{defn}[{\cite[Definition 1.1.13]{ABV09MotShvLAlb}}]\label{ngenerated}
    An \'etale sheaf with transfers $\sF\in\Shv_\et^\tr(k,\Lambda)$ is said to be strongly $n$-generated if the co-unit
    \[\sigma_n^*\sigma_{n*}\sF\to \sF\]
    is an isomorphism. Denote by $\Shv_{\leq n}^\tr(k,\Lambda)$ the category of strongly $n$-generated \'etale sheaves.
\end{defn}

\begin{lem}[{\cite[Lemma~1.1.17]{ABV09MotShvLAlb}}]\label{ShvnEquiv}
    The functor $\sigma_n^*$ is fully faithful and induces an equivalence between $\Shv_\et^\tr(k_{\leq n},\Lambda)$ and $\Shv_{\leq n}^\tr(k,\Lambda)$.
\end{lem}

\begin{lem}\label{Shve^*}
    \begin{enumerate}[label={\rm(\arabic*)},leftmargin=*]
        \item The direct image functor $e_*^\tr \colon  \Shv_\et^\tr(K,\Lambda) \to \Shv_\et^\tr(k,\Lambda)$ has a left adjoint
        \[e^*_\tr \colon  \Shv_\et^\tr(k,\Lambda) \longrightarrow        \Shv_\et^\tr(K,\Lambda)\]
        and $e^*_\tr\Lambda_\tr(X)\simeq\Lambda_\tr(X_K).$ \menum
        \item The direct image functor $e^{\leq n}_* \colon  \Shv_\et^\tr(K_{\leq n},\Lambda) \to \Shv_\et^\tr(k_{\leq n},\Lambda)$ has a left adjoint 
        \[e^*_{\leq n} \colon  \Shv_\et^\tr(k_{\leq n},\Lambda) \longrightarrow \Shv_\et^\tr(K_{\leq n},\Lambda)\] 
        and $e^*_{\leq n}(\Lambda_{\leq n}(X))=\Lambda_{\leq n}(X_K).$
        \item We have the following natural isomorphism
        \[\sigma_n^*\circ e_{\leq n}^*\simeq e^*_\tr\circ\sigma_n^*.\]
    \end{enumerate}
\end{lem}
\begin{proof}
    \begin{enumerate}[leftmargin=*]
        \item Clearly, $a_\et\circ e^*_\PST$ is left adjoint to $e_*^\tr$, where $a_\et$ is the functor in Proposition~\ref{Shv_et^tr}(1). We have
        \[a_\et (e^*_\PST(\Lambda_\tr(X)))=a_\et(\Lambda_\tr(X_K))=\Lambda_\tr(X_K),\]
        where the second equality holds because $\Lambda_\tr(X_K)$ is already an \'etale sheaf with transfers. \menum
        \item The proof is similar to (1).
        \item This is a direct corollary of $\sigma_{n*}\circ e_*^\tr = e^{\leq n}_*\circ\sigma_{n*}$. \qedhere
    \end{enumerate}
\end{proof}

\begin{prop}\label{Shve^*Exact}
    The inverse image functor $e^*_\tr \colon  \Shv_\et^\tr(k,\Lambda) \to \Shv_\et^\tr(K,\Lambda)$ is exact.
\end{prop}
\begin{proof}
    Note that $e^*_\tr=a_\et\circ e^*_\PST$. By Proposition~\ref{e^*PST} (3) and Proposition~\ref{Shv_et^tr}, the functors $e^*_\PST$ and $a_\et$ are both exact. Thus $e^*_\tr$ is also exact.
\end{proof}

\subsection{Homotopy invariant sheaves}\label{HI_et}
\begin{defn}[{\cite[Definitions 2.15 and 9.22]{MVW06Motive}}]\label{HI}
    Let $k$ be a field.
    \begin{enumerate}[label={\rm(\arabic*)}]
        \item  A presheaf with transfers $\sF$ is said to be homotopy invariant if the projection $X\times_k\bA^1_k\to X$ induces an isomorphism
        \[\sF(X)\stackrel{\sim}{\longrightarrow}\sF(X\times_k \bA^1_k).\] 
        \item An \'etale sheaf with transfers $\sF$ is said to be strictly homotopy invariant if the projection $X\times_k \bA^1_k\to X$ induces isomorphisms
        \[H^i_\et(X,\sF)  \stackrel{\sim}{\longrightarrow}  H^i_\et(X\times_k \bA^1_k,\sF)\text{~for all }i\geq 0.\]
    \end{enumerate}
    We denote by $\HI_\et(k,\Lambda)$ the full subcategory of $\Shv_\et^\tr(k,\Lambda)$ whose objects are homotopy invariant sheaves.
\end{defn}

\begin{thm}[Voevodsky, Suslin]\label{HIs=HI}
    Let $k$ be a field of exponential characteristic $p$ and let $\Lambda$ be the ring $\bZ[1/p]$.
    \begin{enumerate}[label={\rm(\arabic*)}]
        \item If $\sF$ is a homotopy invariant presheaf with transfers, then $a_\et(\sF)$ is strictly homotopy invariant.
        \item The category $\HI_\et(k,\Lambda)$ is a thick subcategory of $\Shv_\et^\tr(k,\Lambda)$. In particular, $\HI_\et(k,\Lambda)$ is abelian. 
    \end{enumerate}
\end{thm}
\begin{proof}
    The first assertion is essentially due to Voevodsky and Suslin. In fact, Voevodsky established this result for the Nisnevich topology and perfect fields $k$ (\cite[Theorem~24.1]{MVW06Motive}), and Suslin generalized Voevodsky's result to arbitrary fields (\cite[Theorem~3.4]{Suslin17Nonperfect}). Then one can deduce the result for the \'etale topology by using Suslin's rigidity theorem \cite[Theorem~7.20]{MVW06Motive}. See \cite[Proposition~1.7.5]{BVK16Derived1Motives} and \cite[Proposition~1.1.2]{ABV09MotShvLAlb}.

    The second assertion follows immediately from the first one and the five lemma.
\end{proof}

\begin{lem}[{\cite[Lemmas 1.1.1 and 1.1.2]{ABV09MotShvLAlb}}]
    The inclusion $\iota \colon  \HI_\et(k,\Lambda) \hookrightarrow \Shv_\et^\tr(k,\Lambda)$ admits a left adjoint 
    \[h_0^\et \colon   \Shv_\et^\tr(k,\Lambda) \to \HI_\et(k,\Lambda).\] 
    Here, $h_0^\et(\sF)$ is given by the \'etale sheaf with transfers associated with the $0$-th homology of the Suslin complex $C_*\sF$ (\cite[Lecture 2]{MVW06Motive}).
\end{lem}

For $X\in\Sm/k$, we let
\[h_0^\et(X) \colonequals h_0^\et(\Lambda_\tr(X)).\]

\begin{lem}\label{HIe_*}
    The direct image $e_*^\tr$ maps homotopy invariant sheaves to homotopy invariant sheaves. In other words, we have a functor $e_*^\HI \colon  \HI_\et(K,\Lambda)\to \HI_\et(k,\Lambda)$ such that the following diagram is commutative
    \[\xymatrix{
        \HI_\et(K,\Lambda) \ar@{.>}[d]_{e_*^\HI}   \ar@{^(->}[r]^{\iota_K} & \Shv_\et^\tr(K,\Lambda)  \ar[d]^{e_*^\tr}\\
        \HI_\et(k,\Lambda) \ar@{^(->}[r]_{\iota_k}                         & \Shv_\et^\tr(k,\Lambda).
    }\]
\end{lem}
\begin{proof}
    For $X\in\Sm/k$, we have the following commutative diagram
    \[\xymatrix{
        (K\times_k \bA^1_k)\times_K(K\times_k X)  \ar[r]^-\sim \ar[rd]_{g} & K\times_k(\bA^1_k\times_kX) \ar[d]^{f} \\
        & K\times_kX,
    }\]
    where $f$ is the base change of the projection $\bA^1_k\times_k X\to X$ and $g$ is the projection to $K\times_kX$. For $\sF\in\HI_\et(K,\Lambda)$, the morphism $\sF(g)$ is an isomorphism. Thus $\sF(f)$ is also an isomorphism, i.e., the morphism $(e_*^\tr\sF)(X)\to (e_*^\tr\sF)(\bA^1_k\times_k X)$ induced by the projection is an isomorphism, which means that $e_*^\tr\sF$ is homotopy invariant.
\end{proof}

\begin{lem}[{\cite[Proposition~4.9]{Suslin17Nonperfect}}]\label{InverseHIPST}
    The functor $e^*_\PST$ preserves homotopy invariant presheaves with transfers.
\end{lem}

\begin{prop}\label{HIe^*Exact}
    The inverse image $e^*_\tr$ maps homotopy invariant sheaves to homotopy invariant sheaves. In other words, we have a functor $e^*_\HI \colon  \HI_\et(k,\Lambda)\to \HI_\et(K,\Lambda)$ such that the following diagram is commutative
    \[\xymatrix{
        \HI_\et(k,\Lambda) \ar@{.>}[d]_{e^*_\HI} \ar@{^(->}[r]^{\iota_k} & \Shv_\et^\tr(k,\Lambda)  \ar[d]^{e^*_\tr}\\
        \HI_\et(K,\Lambda) \ar@{^(->}[r]_{\iota_K}                  & \Shv_\et^\tr(K,\Lambda).
    }\]
    Moreover, the functor $e_\HI^*$ is exact.
\end{prop}
\begin{proof}
    Note that $e^*_\tr=a_\et\circ e^*_\PST$. By Lemma~\ref{InverseHIPST} and Theorem~\ref{HIs=HI} (1), the functors $e^*_\PST$ and $a_\et$ preserve homotopy invariance. Thus $e^*_\tr$ maps homotopy invariant sheaves to homotopy invariant sheaves.
    
    By Proposition~\ref{Shve^*Exact}, the functor $e^*_\tr$ is exact. Note that $\iota_k$ and $\iota_K$ are both exact by Theorem~\ref{HIs=HI} (2). Thus the functor $e^*_\HI$ is exact.
\end{proof}

\begin{cor}\label{HIadj}
    \begin{enumerate}[leftmargin=*,label={\rm(\arabic*)}]
        \item The functor $e^*_\HI$ is left adjoint to $e_*^\HI$.\menum
        \item The unit $\id\to \iota_k h_0^\et$ induces a natural isomorphism 
        \[h_0^\et e_\tr^* \stackrel{\sim}{\longrightarrow} h_0^\et e_\tr^*\iota_k h_0^\et=e_\HI^*h_0^\et.\] 
        In particular, $e^*_\HI(h_0^\et(X))\simeq h_0^\et(X_K)$.
    \end{enumerate}
\end{cor}
\begin{proof}
    It is clear that $e^*_\HI$ is left adjoint to $e_*^\HI$ by the adjunction $(e_\tr^*,e_*^\tr)$. Taking left adjoint to $\iota_k\circ e_*^\HI=e_*^\tr\circ \iota_K$, we get the second assertion.
\end{proof}

The following result is established by Suslin \cite{Suslin17Nonperfect} for the Nisnevich topology and the Zariski topology. We deal with the \'etale topology.
\begin{prop}\label{PurelyInsepEquiv}
    Let $k$ be a field of characteristic $p>0$ and let $K/k$ be a purely inseparable extension. Then
    \begin{enumerate}[label={\rm(\arabic*)}]
        \item the functor $e_*^\PST \colon  \PST(K,\Lambda) \to \PST(k,\Lambda)$ is an equivalence of categories with quasi-inverse $e^*_\PST$;
        \item the functor $e_*^\tr \colon  \Shv_\et^\tr(K,\Lambda) \to \Shv_\et^\tr(k,\Lambda)$ is an equivalence of categories with quasi-inverse $e^*_\tr=e^*_\PST$;
        \item the functor $e_*^\HI \colon  \HI_\et^\tr(K,\Lambda) \to \HI_\et^\tr(k,\Lambda)$ is an equivalence of categories with quasi-inverse $e^*_\HI$.
    \end{enumerate} 
\end{prop}
\begin{proof}
    \begin{enumerate}[leftmargin=*,label={\rm(\arabic*)}]
        \item This is \cite[Corollary~1.14]{Suslin17Nonperfect}.
        \item First, assume that $K$ is a perfect closure of $k$. It suffices to show that if $\sG$ is an \'etale sheaf, then $e^*_\PST\sG$ is also an \'etale sheaf. Since $\sG \simeq e_*^\PST e^*_\PST\sG$ by (1), it suffices to check that if $\sF$ is a presheaf with transfers over $K$ such that $e_*^\PST\sF$ is an \'etale sheaf with transfers, then $\sF$ itself is also an \'etale sheaf. Suppose that $U\in \Sm/K$ and $\{V_i\to U\}$ is an \'etale covering. We check the sheaf condition for $\sF$. It suffices clearly to deal with the case when $U$ is irreducible. Then by \cite[Lemma~1.12]{Suslin17Nonperfect}, there exist an irreducible $X\in\Sm/k$ and a finite surjective purely inseparable\footnote{Purely inseparable morphisms are also called radicial or universally injective morphisms in other literatures.} morphism $U\to X_K$. By \cite[Expos\'e~VIII, Th\'eor\`eme~1.1]{SGA4II}, the \'etale morphisms $V_i\to U$ descend to \'etale morphisms $Y_i\to X_K$, and then descend to \'etale morphisms $Z_i\to X$, i.e., we have the following Cartesian squares: 
        \[\xymatrix{
            V_i \ar[r] \ar[d]  & Y_i \ar[r] \ar[d] & Z_i \ar[d] \\
            U   \ar[r]         & X_K \ar[r] \ar[d] & X   \ar[d] \\
                               & \Spec K    \ar[r] & \Spec k.
        }\]
        By our assumption, $e_*^\PST\sF$ is a sheaf on $(\Sm/k)_\et$. Thus we have the following exact sequence
        \[0\to \sF(X_K) \to \prod \sF(Y_i) \to \prod \sF(Y_i\times_{X_K} Y_j).\]
        Using \cite[Lemma~2.4]{Suslin17Nonperfect}, we obtain the following exact sequence from the above one
        \[0\to \sF(U) \to \prod \sF(V_i) \to \prod \sF(V_i \times_U V_j),\]
        which means that $\sF$ is a sheaf.

        For a general purely inseparable extension $K/k$, a perfect closure $k^\perf$ of $k$ is also a perfect closure of $K$. Thus the inverse image functor from $k$ to $k^\perf$ is an equivalence and so is the inverse image functor from $K$ to $k^\perf$. It follows that the inverse image functor from $k$ to $K$ is also an equivalence.
        \item The last assertion is a combination of Lemma~\ref{HIe_*}, Proposition~\ref{HIe^*Exact} and the second assertion. \qedhere
    \end{enumerate}
\end{proof}

\subsection{\texorpdfstring{$n$}{n}-motivic sheaves}\label{HI_leqn}
\begin{defn}[{\cite[Definition 1.1.20]{ABV09MotShvLAlb}}]
    A homotopy invariant sheaf $\sF\in\HI_\et(k,\Lambda)$ is said to be $n$-motivic if the natural morphism
    \[h_0^\et\sigma_n^*\sigma_{n*}\sF\longrightarrow h_0^\et(\sF)=\sF\]
    is an isomorphism. We denote by $\HI_{\leq n}(k,\Lambda)$ the full subcategory of $n$-motivic sheaves.
\end{defn}
\begin{rmk}
    In \cite{Ayoub11Motivic-t-structure}, $n$-motivic sheaves are called $n$-presented $\cH$-sheaves.
\end{rmk}

\begin{rmk}\label{HInShvn}
    As explained in \cite[Remark 1.1.21]{ABV09MotShvLAlb}, a homotopy invariant sheaf is $n$-motivic if and only if it is isomorphic to $h_0^\et\sigma_n^*\sF$ for some $\sF\in\Shv_\et^\tr(k_{\leq n},\Lambda)$.
\end{rmk}

\begin{lem}\label{HInEmbedShvn}
    We have a pair of adjoint functors
    \[h_0^\et\sigma_n^* \colon  \Shv_\et^\tr(k_{\leq n},\Lambda)\rightleftarrows \HI_{\leq n}(k,\Lambda) \colon   \sigma_{n*}\iota\iota_n,\]
    where $\iota_n \colon  \HI_{\leq n}(k,\Lambda)\hookrightarrow \HI_\et(k,\Lambda)$ and $\iota \colon  \HI_\et(k,\Lambda)\hookrightarrow \Shv_\et^\tr(k,\Lambda)$ are the inclusions. Moreover, the functor 
    \[\sigma_{n*}\iota\iota_n \colon  \HI_{\leq n}(k,\Lambda)\hookrightarrow \Shv_\et^\tr(k_{\leq n},\Lambda)\]
    is fully faithful.
\end{lem}
\begin{proof}
    The adjunction follows from the adjunctions $(h_0^\et,\iota)$ and $(\sigma_n^*,\sigma_{n*})$. By definition, the co-unit
    \[(h_0^\et\sigma_n^*)(\sigma_{n*}\iota\iota_n)\sF \longrightarrow \sF\]
    is an isomorphism for every $\sF\in\HI_{\leq n}(k,\Lambda)$, which implies that $\sigma_{n*}\iota\iota_n$ is fully faithful. 
\end{proof}

\begin{lem}\label{HIne^*}
    The functor $e^*_\HI$ maps $n$-motivic sheaves to $n$-motivic sheaves. More precisely, there exists a functor $e_n^* \colon  \HI_{\leq n}(k,\Lambda)\to \HI_{\leq n}(K,\Lambda)$ such that the following diagram is commutative
    \[\xymatrix{
        \Shv_\et^\tr(k_{\leq n},\Lambda) \ar[d]_{e_{\leq n}^*} \ar[rr]^{h_0^\et\sigma_n^*} & &   \HI_{\leq n}(k,\Lambda) \ar[r]^{\iota_n} \ar@{.>}[d]_{e_{n}^*} & \HI_\et(k,\Lambda) \ar[d]^{e^*_\HI} \\
        \Shv_\et^\tr(K_{\leq n},\Lambda) \ar[rr]^{h_0^\et\sigma_n^*}  && \HI_{\leq n}(K,\Lambda) \ar[r]^{\iota_n} & \HI_\et(K,\Lambda).
    }\]
    In particular, for $X\in (\Sm/k)_{\leq n}$, we have
    \[e^*_{n}(h_0^\et(X))\simeq h_0^\et(X_K).\]
\end{lem}
\begin{proof}
    Let $\sF\in\HI_{\leq n}(k,\Lambda)$. By Remark~\ref{HInShvn}, we write $\sF=h_0\sigma_n^*\sG$ for $\sG\in\Shv_\et^\tr(k_{\leq n},\Lambda)$. Then 
    \[e^*_\HI \iota_n\sF=e^*_\HI\iota_nh_0^\et\sigma_n^*\sG\simeq h_0^\et e^*_\tr\sigma_n^*\sG \simeq h_0^\et\sigma_n^* e^*_{\leq n}\sG,\]
    where the second isomorphism holds by Corollary~\ref{HIadj} (2), and the last isomorphism holds by Lemma~\ref{Shve^*} (3).
    Then by Remark~\ref{HInShvn} again, $e^*_\HI \iota_n\sF$ is $n$-motivic.
\end{proof}

\begin{defn}\label{HIne_*}
    We call $e^*_n$ the inverse image functor of $n$-motivic sheaves. We define the direct image functor $e_{n*}$ of $n$-motivic sheaves to be the composition
    \[\HI_{\leq n}(K,\Lambda) \stackrel{\iota_n}{\hookrightarrow} \HI_\et(K,\Lambda) \stackrel{e_*^\HI}{\longrightarrow} \HI_\et(k,\Lambda) \stackrel{h_0^\et\sigma_n^*\sigma_{n*}}{\longrightarrow} \HI_{\leq n}(k,\Lambda).\]
\end{defn}

\begin{lem}[{\cite[Lemma~1.1.23]{ABV09MotShvLAlb}}]\label{h0sigma}
    We have natural isomorphisms
    \[(\sigma_{n*}\iota) \stackrel{\sim}{\longrightarrow} (\sigma_{n*}\iota)(h_0^\et\sigma_n^*)(\sigma_{n*}\iota) \stackrel{\sim}{\longrightarrow} (\sigma_{n*}\iota).\]
\end{lem}

\begin{lem}
    \begin{enumerate}[leftmargin=*,label={\rm(\arabic*)}]
        \item The functor $h_0^\et\sigma_n^*\sigma_{n*}\iota$ is right adjoint to $\iota_n \colon  \HI_{\leq n}\hookrightarrow \HI_\et$.\menum
        \item We have a commutative diagram 
        \[\xymatrix{
            \HI_\et(K,\Lambda) \ar[d]_{e_*^\HI} \ar[rr]^{h_0^\et\sigma_n^*\sigma_{n*}\iota}  
            &&  \HI_{\leq n}(K,\Lambda) \ar[rr]^{\sigma_{n*}\iota\iota_n} \ar[d]_{e_{n*}} 
            &&  \Shv_\et^\tr(K_{\leq n},\Lambda) \ar[d]^{e_*^{\leq n}} \\
            \HI_\et(k,\Lambda) \ar[rr]^{h_0^\et\sigma_n^*\sigma_{n*}\iota} 
            && \HI_{\leq n}(k,\Lambda) \ar[rr]^{\sigma_{n*}\iota\iota_n} 
            && \Shv_\et^\tr(k_{\leq n},\Lambda).
        }\]
        \item The functor $e_{n*}$ is right adjoint to $e^*_n$. 
    \end{enumerate}
\end{lem}
\begin{proof}
    It is easy to check the first two assertions by using Lemma~\ref{h0sigma}. Then the third assertion can be checked easily by using the first assertion.
\end{proof}

\subsection{\texorpdfstring{$0$}{0}-motivic sheaves}
In this subsection, we consider sheaves on the small \'etale site $(\Et/k)_\et=(\Sm/k)_{\leq 0,\et}$.

The obvious inclusion of sites $\sigma \colon  (\Et/k)_\et \to (\Sm/k)_\et$ is continuous and co-continuous and preserves fiber products and the final object $\Spec k$. Thus it gives an adjunction of categories
\[\sigma^* \colon   \Shv_\et(\Et/k,\Lambda) \leftrightarrows \Shv_\et(\Sm/k,\Lambda) \colon  \sigma_{*},\]
where $\sigma_*\sF=\sF\circ\sigma$. These two functors are both exact and $\sigma^*$ is fully faithful. See Lemma~\ref{topoi} for more details. By Theorem~\ref{Shv_et^tr} (4), the forgetful functor $\gamma_* \colon  \Shv_\et^\tr(k,\Lambda) \to \Shv_\et(\Sm/k,\Lambda)$ has a left adjoint $\gamma^*$. We have the following result.
\begin{lem}\label{Shv0}
    \begin{enumerate}[leftmargin=*,label={\rm(\arabic*)}]
        \item The functor 
            \[\gamma^*\sigma^* \colon  \Shv_\et(\Et/k,\Lambda) \longrightarrow \Shv_\et^\tr(k,\Lambda)\]
            is exact and fully faithful. \menum
        \item We have an equivalence
            \[\gamma_* \colon  \Shv_\et^\tr(k_{\leq 0},\Lambda) \stackrel{\sim}{\longrightarrow} \Shv_\et(\Et/k,\Lambda)\]  
            with quasi-inverse $\sigma_{0*}\gamma^*\sigma^*$, where $\sigma_{0*} \colon  \Shv_\et^\tr(k,\Lambda) \to \Shv_\et^\tr(k_{\leq 0},\Lambda)$ is the restriction functor.
    \end{enumerate}
\end{lem}
\begin{proof}
    The first assertion is \cite[Proposition~3.1.4]{CD16EtaleMotive}. Thus $\gamma^*\sigma^*$ induces an equivalence between $\Shv_\et(\Et/k,\Lambda)$ and the category $\Shv_{\leq 0}^\tr(k,\Lambda)$ in Definition~\ref{ngenerated} with quasi-inverse $\sigma_*\gamma_*$. Then the second assertion is clear by Lemma~\ref{ShvnEquiv}.
\end{proof}

By \cite[Chapitre I, \S4, Proposition~6.5]{DG70AlgGrp}, for a scheme $X$ locally of finite type over a field $k$, there exists an \'etale $k$-scheme $\pi_0(X)$ and a morphism $q_X \colon  X\to \pi_0(X)$ satisfying the following universal property: for any morphism $f \colon  X\to Y$ from $X$ to an \'etale $k$-scheme $Y$, there exists a unique $g \colon  \pi_0(X)\to Y$ such that $f=g\circ q_X$. Moreover, the morphism $q_X$ is fully faithful and its fibers are the connected components of $X$. 

Thus we have a functor
\[\pi_0 \colon  \Sm/k \longrightarrow (\Sm/k)_{\leq 0},\]
which is left adjoint to the inclusion functor. As usual, for a presheaf $\sF$ on $(\Sm/k)_{\leq 0}$, the presheaf $\sF\circ \pi_0$ on $\Sm/k$ will be denoted by $\pi_{0*}(\sF)$. The following result is stated and used in \cite[1.2.1]{ABV09MotShvLAlb}. We give a proof here for completeness.
\begin{lem}\label{pi0cont}
    For a sheaf $\sF$ on $(\Sm/k)_{\leq 0,\et}$, the presheaf $\pi_{0*}(\sF)$ is a sheaf on $(\Sm/k)_\et$.
\end{lem}
\begin{proof}
    Let $U$ be a smooth variety and let $\{V_i\to U\}$ be an \'etale covering. Then $\{\pi_0(V_i)\to \pi_0(U)\}$ is an \'etale covering of $\pi_0(U)$. We claim that $\pi_0(V_i\times_U V_j)\to \pi_0(V_i)\times_{\pi_0(U)} \pi_0(V_j)$ is surjective. Consider the following commutative diagram
    \[\xymatrix{
        \sF(\pi_0(U)) \ar[r] \ar@{=}[d] 
        & \prod_i\sF(\pi_0(V_i)) \ar@<.5ex>[r] \ar@<-.5ex>[r] \ar@{=}[d] 
        & \prod_{i,j}\sF(\pi_0(V_i)\times_{\pi_0(U)}\pi_0(V_j)) \ar@{^(->}[d]
        \\
        \sF(\pi_0(U)) \ar[r]            
        & \prod_i\sF(\pi_0(V_i)) \ar@<.5ex>[r] \ar@<-.5ex>[r]
        & \prod_{i,j}\sF(\pi_0(V_i\times_U V_j)).
    }\]
    Since $\sF$ is a sheaf, the first row is exact and the last vertical arrow is an injection. It follows that the second row is also exact, which means that $\pi_{0*}(\sF)$ is a sheaf.

    Now, we prove the claim: for \'etale morphisms $V_1\to U$ and $V_2\to U$, the canonical morphism 
    \[\varphi \colon  \pi_0(V_1\times_U V_2) \longrightarrow \pi_0(V_1)\times_{\pi_0(U)} \pi_0(V_2)\]
    is surjective. Since $\pi_0$ commutes with field extensions (\cite[Chapitre I, \S 4, Proposition~6.7]{DG70AlgGrp}) and commutes with disjoint unions, we may assume that $k$ is separably closed, and $V_1$, $V_2$ and $U$ are connected. Then it suffices to show that $V_1\times_U V_2$ is nonempty. Since the morphism $f_i \colon  V_i \to U$ is \'etale, it is an open mapping. Because $U$ is connected and is smooth over $k$, it is irreducible. So the intersection of two open subsets $f_1(V_1)\cap f_2(V_2)$ is nonempty, i.e., $f_1(V_1)\times_{U} f_2(V_2)\neq \varnothing$. Since the canonical morphism 
    \[V_1\times_{U} V_2 \longrightarrow f_1(V_1)\times_{U} f_2(V_2)\]
    is surjective, we obtain that $V_1\times_U V_2$ is nonempty.
\end{proof}

\begin{rmk}
    This lemma means that $\pi_0 \colon  (\Sm/k)_\et \to (\Sm/k)_{\leq 0,\et}$ is a continuous functor in the sense of \cite[Expos\'e~III, D\'efinition 1.1]{SGA4I}. But it is not a continuous functor in the sense of \cite[\href{https://stacks.math.columbia.edu/tag/00WV}{Definition~00WV}]{stacks-project}, which is stronger. In general, the canonical morphism $\pi_0(X\times_U V) \to \pi_0(X)\times_{\pi_0(U)} \pi_0(V)$ is not an isomorphism even if $V\to U$ is an \'etale morphism. For example, we have the following Cartesian square
    \[\xymatrix{
        \mu_n \ar[r] \ar[d] & \Spec k \ar[d] \\
        \bG_m \ar[r]^n      & \bG_m,
    }\]    
    where the morphism $\Spec k \to \bG_m$ is the zero section. If $n\geq 2$ is prime to the characteristic of $k$, then $\mu_n \to\Spec k$ is \'etale and $\pi_0(\mu_n)\simeq \mu_n\not\simeq\Spec(k)$.
\end{rmk}

If $Z$ is an elementary correspondence from $X$ to $Y$, then $\pi_0(Z)$ is an elementary correspondence from $\pi_0(X)$ to $\pi_0(Y)$ by using the canonical isomorphism (\cite[Chapitre I, \S4, Corollaire 6.10]{DG70AlgGrp})
\[\pi_0(X\times_k Y)\stackrel{\sim}{\longrightarrow} \pi_0(X)\times_k \pi_0(Y)\]
and the fact that $q_X \colon  X\to \pi_0(X)$ is surjective. So $\pi_0$ induces a functor 
\[\pi_0^\tr \colon  \Cor(k)\to \Cor(k_\leq 0)\]
which is compatible with the graph functor: 
\[\xymatrix{
    \Sm/k     \ar[r]^-{\pi_0} \ar[d]_{\gamma}     & (\Sm/k)_{\leq 0} \ar[d]_\gamma \\
    \Cor(k)   \ar[r]^-{\pi_0^\tr}                 & \Cor(k_{\leq 0}). 
}\] 
The above commutative diagram induces the following one:  
\[\xymatrix{
    \Shv_\et^\tr(k_{\leq 0},\Lambda)    \ar[r]^{\pi_{0*}^\tr} \ar[d]_{\gamma_*}  & \Shv_\et^\tr(k,\Lambda) \ar[d]_{\gamma_*} \\
    \Shv_\et((\Sm/k)_{\leq 0},\Lambda)  \ar[r]^{\pi_{0*}}                        & \Shv_\et(\Sm/k,\Lambda),
}\]
where $\pi_{0*}$ and $\pi_{0*}^\tr$ are the direct image functors and $\gamma_*$ are the forgetful functors.

\begin{cor}
    We have $\sigma^*=\pi_{0*}$ and $\sigma_0^*\simeq\pi_{0*}^\tr$.
\end{cor}
\begin{proof}
    For a sheaf $\sF\in\Shv_\et(\Et/k,\Lambda)$, the inverse image $\sigma^*\sF$ is the sheaf associated with the presheaf
    \[U\longmapsto \varinjlim_{V}\sF(V),\]
    where $V$ are the \'etale schemes over $k$ such that $U\to \Spec k$ factors through them. By the universal property of $\pi_0$, the above colimit is in fact $\sF(\pi_0(U))$. So the above presheaf is in fact $\pi_{0*}\sF$, which is already a sheaf by Lemma~\ref{pi0cont}. Hence $\sigma^*\sF=\pi_{0*}\sF$. Using Lemma~\ref{Shv0}, we can check that $\gamma_*\sigma_{0}^*\simeq\sigma^*\gamma_*$, where the latter one is $\pi_{0*}\gamma_*=\gamma_*\pi_{0*}^\tr$. Since $\gamma_*$ is conservative by Proposition~\ref{Shv_et^tr} (3), we get $\sigma_0^*\simeq\pi_{0*}^\tr$.
\end{proof}

\begin{prop}\label{HI0}
    \begin{enumerate}[leftmargin=*,label={\rm(\arabic*)}]
        \item The functor $\gamma^*\sigma^* \colon  \Shv_\et(\Et/k,\Lambda)\to \Shv_\et^\tr(k,\Lambda)$ induces an equivalence of categories between       $\Shv_\et(\Et/k,\Lambda)$ and $\HI_{\leq 0}(k,\Lambda)$.   \menum
        \item The embedding $\HI_{\leq 0}(k,\Lambda)\hookrightarrow \Shv_\et^\tr(k,\Lambda)$ admits a left adjoint $\pi_0^*$ with 
            \[\pi_0^*(\Lambda_\tr(X))=\Lambda_\tr(\pi_0(X)).\]
        \item The category $\HI_{\leq 0}(k,\Lambda)$ is a Serre subcategory of $\Shv_\et^\tr(k,\Lambda)$.
    \end{enumerate}
\end{prop}
\begin{proof}
    See Lemma~\ref{Shv0} and \cite[1.2.2, 1.2.5, 1.2.7]{ABV09MotShvLAlb}.
\end{proof}

From now on, we shall write the functor $\pi_0^*$ simply as $\pi_0$.

\subsection{Derived direct image of \texorpdfstring{$1$}{1}-motivic sheaves}
To study the unbounded derived functors, we use the descent model structure on chain complexes developed by Cisinski and D\'eglise.

For readers' convenience, we recall the machinery: 
\begin{defn}[{\cite[Definition 2.2]{CD09ModelCategories}}]\label{descentstruc}
    Let $\cA$ be a Grothendieck category. Let $\cG$ be an essentially small set of objects of $\cA$ and $\cH$ be a subset of $C(\cA)$. For $E\in \cG$, let $D(E)$ be the complex $[E \stackrel{\id}{\to} E]$, concentrated in degrees $0$ and $1$, and let $f_E \colon   E[-1]\to D(E)$ be the map given by the identity in degree 1. 
    \begin{enumerate}[label={\rm(\arabic*)}]
        \item A morphism in $C(\cA)$ is defined to be a $\cG$-cofibration if it is contained in the smallest class of maps in $C(\cA)$ closed under pushouts, transfinite compositions and retracts, generated by the morphisms $f_E[n]$ for any integer $n$ and any $E$ in $\cG$. 
        \item A chain complex $C \in C(\cA)$ is said to be $\cG$-local if for any $E$ in $\cG$ and any integer $n$, there is a canonical isomorphism
        \[\Hom_{K(\cA)}(E[n],C) \stackrel{\sim}{\longrightarrow} \Hom_{D(\cA)}(E[n],C).\]
        \item An object $C$ of $C(\cA)$ is said to be $\cH$-flasque if for any integer $n$ and any $H$ in $\cH$,
        \[\Hom_{K(\cA)}(H,C[n])=0.\]
        \item The pair $(\cG,\cH)$ is called a descent structure on $\cA$ if 
        \begin{enumerate}[label={\rm(\alph*)}]
            \item elements in $\cH$ are $\cG$-cofibrant acyclic complexes;
            \item every $\cH$-flasque complex is $\cG$-local.
        \end{enumerate}
    \end{enumerate}
\end{defn}

\begin{thm}[{\cite[Theorem~2.5]{CD09ModelCategories}}]
    Let $\cA$ be a Grothendieck category endowed with a descent structure $(\cG,\cH)$. Then the category $C(\cA)$ is a proper cellular model category with quasi-isomorphisms as weak equivalences, and $\cG$-cofibrations as cofibrations. Furthermore, a complex $C\in C(\cA)$ is fibrant if and only if it is $\cH$-flasque, or equivalently, $\cG$-local.
\end{thm}

\begin{defn}[{\cite[p.~228]{CD09ModelCategories}}]
    Let $\cA$ and $\cA'$ be two Grothendieck categories. Suppose that $(\cG,\cH)$ (resp. $(\cG',\cH')$) is a descent structure on $\cA$ (resp. $\cA'$). A functor $f^* \colon  \cA'\to\cA$ is said to satisfy descent (with respect to the above descent structures) if it satisfies the following conditions:  
    \begin{enumerate}[label={\rm(\arabic*)}]
        \item the functor $f^*$ commutes with small colimits, or equivalently, it has a right adjoint $f_*$;
        \item $f^*(E')$ is a direct sum of elements of $\cG$ for any $E'$ in $\cG'$;
        \item $f^*(H')$ is in $\cH$ for any $H'$ in $\cH'$.
    \end{enumerate}
\end{defn}

\begin{thm}[{\cite[Theorem~2.14]{CD09ModelCategories}}]
    If $f^* \colon  \cA'\to \cA$ satisfies descent, then the pair of adjoint functors
    \[f^* \colon  C(\cA')\rightleftarrows C(\cA) \colon  f_*\]
    is a Quillen adjunction with respect to the descent model structure. In particular, the functors $f^*$ and $f_*$ have the functors
    \[Lf^* \colon  D(\cA')\to D(\cA) \text{\quad and \quad} Rf_* \colon  D(\cA)\to D(\cA')\]
    as left and right derived functors respectively, and $Lf^*$ is left adjoint to $Rf_*$.
\end{thm}

\begin{rmk}\label{CompositionDerived}
    Let $\cA$, $\cA'$ and $\cA''$ be three Grothendieck categories endowed with descent structures. Let $f'^* \colon  \cA''\to \cA'$ and $f^* \colon  \cA'\to \cA$ be two functors satisfying descent, and let $f'_*$ and $f_*$ be their right adjoints. Then it follows easily from general abstract nonsense about Quillen adjunctions and the preceding theorem that we have canonical isomorphisms of total derived functors
    \[Lf^*\circ Lf'^*\simeq L(f^*\circ f'^*) \text{\quad and \quad }  R(f'_*\circ f_*)\simeq Rf'_*\circ Rf_*.\]
    In fact, under the above assumptions, the functor $f_*$ preserves fibrant objects and fibrant resolutions compute right derived functors. The condition that $f^*$ satisfies descent can be viewed as an unbounded generalization of the condition that $f_* \colon  \cA \to \cA'$ preserves injective objects, or flasque sheaves in sheaf theory.
\end{rmk}

We come back to motivic sheaves. 

\begin{lem}\label{Le^*}
    Let $K/k$ be a field extension. Then we have the following commutative diagram
    \[\xymatrix{
        D(\Shv_\et((\Sm/k)_{\leq n},\Lambda))  \ar[r]^-{L\gamma^*}      \ar[d]_{Le_{\leq n}^*}
        & D(\Shv_\et^\tr(k_{\leq n},\Lambda))  \ar[r]^-{L\sigma_n^*}    \ar[d]_{Le^*_{\leq n}} 
        & D(\Shv_\et^\tr(k,\Lambda))           \ar[d]_{Le^*_\tr}  \\
        D(\Shv_\et((\Sm/K)_{\leq n},\Lambda))  \ar[r]^-{L\gamma^*} 
        & D(\Shv_\et^\tr(K_{\leq n},\Lambda))  \ar[r]^-{L\sigma_n^*}
        & D(\Shv_\et^\tr(K,\Lambda)).
    }\]
\end{lem}
\begin{proof}
    The non-derived version of this commutative diagram can be found in \S \ref{EST}. In \cite[Example 2.3]{CD09ModelCategories}, using Verdier's computation of hypercohomology \cite[Expos\'e~V, \S7]{SGA4II}, Cisinski and D\'eglise showed that there is a descent structure $(\cG,\cH)$ on $\Shv_\et((\Sm/k)_{\leq n},\Lambda)$, where $\cG$ is the essentially small family consisting of sheaves $\Lambda(X)$ with $X\in (\Sm/k)_{\leq n}$ and $\cH$ is the family of mapping cones of $\Lambda(Y_\bullet)\to \Lambda(X)$ for any \'etale hypercover $Y_\bullet \to X$ on the small \'etale site $X_\et$. By \cite[Proposition~2.2.3]{CD16EtaleMotive}, there are similar model structures on $\Shv_\et^\tr(k,\Lambda)$ and $\Shv_\et^\tr(k_{\leq n},\Lambda)$ by replacing $\Lambda(X)$ with $\Lambda_\tr(X)$. By definition, all the functors in the above diagram satisfy descent. Then the expected result follows from Remark~\ref{CompositionDerived}.
\end{proof}

Recall Voevodsky's triangulated category of effective \'etale motives over a field\footnote{In fact, Voevodsky defined and studied the triangulated subcategory $\DM^\eff_{-,\et}$ consisting of complexes that are bounded above over perfect fields (with finite cohomological dimension) in \cite{Voevodsky00DM}, \cite{MVW06Motive}. Unbounded motivic complexes (over a base scheme) were studied in some other places, for example, \cite{ABV09MotShvLAlb}, \cite{Ayoub11Motivic-t-structure}, and the six-functor formalism in the motivic world (e.g., \cite{Ayoub07SixOperationI}, \cite{Ayoub07SixOperationII}, \cite{Ayoub14EtaleRealization}, \cite{CD19MixedMotive}, \cite{CD16EtaleMotive}).}:   $\DM_\et^\eff(k,\Lambda)$ is the homotopy category of the Bousfield localization of $C(\Shv_\et^\tr(k,\Lambda))$ with respect to the class of arrows $\Lambda_\tr(X\times\bA^1)[n]\to \Lambda_\tr(X)[n]$ for $X\in\Sm/k$ and $n\in\bZ$. By the general theory of Bousfield localizations (\cite[4.3.1]{Hirschhorn03ModelCategories}), $\DM_\et^\eff(k,\Lambda)$ is the full subcategory of $D(\Shv_\et^\tr(k,\Lambda))$ whose objects are the $\bA^1$-local complexes (also called motivic complexes), i.e., the complexes $C$ such that 
\[\Hom_{D(\Shv_\et^\tr(k,\Lambda))}(\Lambda_\tr(X),C[m]) \simeq \Hom_{D(\Shv_\et^\tr(k,\Lambda))}(\Lambda_\tr(\bA_X^1),C[m]).\]
Denote by $L_{\bA^1}$ the $\bA^1$-localization functor, which is left adjoint to the obvious inclusion $\DM_\et^\eff(k,\Lambda)\hookrightarrow D(\Shv_\et^\tr(k,\Lambda))$.

\begin{defn}
    \begin{enumerate}[leftmargin=*,label={\rm(\arabic*)}]
        \item Denote $M(X)$ the object $L_{\bA^1}(\Lambda_\tr(X)[0])$ for $X\in\Sm/k$ and call it the homological motive of $X$.\menum
        \item Denote by $\DM_{\leq n}(k,\Lambda)$ the localizing subcategory of $\DM_\et^\eff(k,\Lambda)$ generated by $M(X)$ for $X\in (\Sm/k)_{\leq n}$. We will call it the triangulated category of $n$-motives.
    \end{enumerate}
\end{defn}

\begin{lem}\label{Le^*DM}
    Let $K/k$ be a field extension. Then we have the following commutative diagram
    \[\xymatrix{
        D(\Shv_\et((\Sm/k)_{\leq n},\Lambda))  \ar[r]^-{L\gamma^*}      \ar[d]_{Le_{\leq n}^*}
        & D(\Shv_\et^\tr(k_{\leq n},\Lambda))  \ar[r]^-{L_{\bA^1}\circ L\sigma_n^*}    \ar[d]_{Le^*_{\leq n}} 
        & \DM_{\leq n}(k,\Lambda)              \ar[d]_{e^*_\DM}  \\
        D(\Shv_\et((\Sm/K)_{\leq n},\Lambda))  \ar[r]^-{L\gamma^*} 
        & D(\Shv_\et^\tr(K_{\leq n},\Lambda))  \ar[r]^-{L_{\bA^1}\circ L\sigma_n^*}
        & \DM_{\leq n}(k,\Lambda),
    }\]
    where $e^*_{\DM}$ maps $M(X)$ to $M(X_K)$.
\end{lem}
\begin{proof}
    By \cite[2.2.4]{CD16EtaleMotive}, we have the following commutative diagram
    \[\xymatrix{
        D(\Shv_\et^\tr(k,\Lambda))  \ar[d]_{Le^*_\tr} \ar[r]^-{L_{\bA^1}}
        & \DM_\et^\eff(k,\Lambda)   \ar[d]_{e^*_\DM}  \\
        D(\Shv_\et^\tr(K,\Lambda))  \ar[r]^-{L_{\bA^1}}                    
        & \DM_\et^\eff(K,\Lambda),
    }\]
    where $e^*_\DM$ maps $M(X)$ to $M(X_K)$. By \cite[\S 2.2]{ABV09MotShvLAlb}, $L_{\bA^1}\circ L\sigma_n^* \colon  D(\Shv_\et^\tr(k_{\leq n},\Lambda)) \to \DM_\et^\eff(k,\Lambda)$ takes values in the subcategory $\DM_{\leq n}(k,\Lambda)\hookrightarrow \DM_\et^\eff(k,\Lambda)$. Then this assertion follows from Lemma~\ref{Le^*}.
\end{proof}

Now, we focus on $1$-motivic sheaves.
\begin{prop}[{\cite[Corollary~1.3.5]{ABV09MotShvLAlb}}]\label{HI1}
    The category $\HI_{\leq 1}(k,\Lambda)$ is a Serre subcategory of $\Shv_\et^\tr(k,\Lambda)$. In particular, the inclusion $\iota\iota_1 \colon  \HI_{\leq 1}(k,\Lambda)\hookrightarrow \Shv_\et^\tr(k,\Lambda)$ is exact.
\end{prop}

\begin{cor}\label{HI1exact}
    \begin{enumerate}[leftmargin=*,label={\rm(\arabic*)}]
        \item $\HI_{\leq 0}(k,\Lambda)$ is a Serre subcategory of $\HI_{\leq 1}(k,\Lambda)$. \menum
        \item The fully faithful functor $\sigma_{1*}\iota\iota_1 \colon  \HI_{\leq 1}(k,\Lambda) \hookrightarrow \Shv_\et^\tr(k_{\leq 1},\Lambda)$ is exact.
        \item The inverse image functor $e^*_1 \colon  \HI_{\leq 1}(k,\Lambda) \to \HI_{\leq 1}(K,\Lambda)$ (in Lemma~\ref{HIne^*}) is exact.
    \end{enumerate}
\end{cor}
\begin{proof}
    By Propositions~\ref{HI0} and \ref{HI1}, $\HI_{\leq 0}(k,\Lambda)$ and $\HI_{\leq 1}(k,\Lambda)$ are both Serre subcategories of $\Shv_\et^\tr(k,\Lambda)$, which implies the first assertion.

    The second assertion holds because the functors $\sigma_{1*}$ and $\iota\iota_1$ are both exact.

    By Propositions~\ref{HI1} and \ref{HIs=HI}, the inclusion functor $\iota_1 \colon  \HI_{\leq 1} \hookrightarrow \HI_\et$ is exact. By Proposition~\ref{HIe^*Exact}, the inverse image functor for $\HI_\et$ is exact. Then the exactness of $e^*_1$ follows from the natural isomorphism $\iota_1\circ e^*_1\simeq e^*_\HI\circ\iota_1$ (Lemma~\ref{HIne^*}). 
\end{proof}

Since the functors in the above lemma are exact, they can be derived trivially. 

\begin{thm}[{\cite[Theorem~2.4.1 and Corollary~2.4.9]{ABV09MotShvLAlb}}] 
    The derived functor 
    \[\iota\iota_1 \colon   D(\HI_{\leq 1}(k,\Lambda)) \to D(\Shv_\et^\tr(k,\Lambda))\]
    is fully faithful, and the essential image is the subcategory $\DM_{\leq 1}(k,\Lambda)$.
\end{thm}

The following result reduces the study of higher direct images of $1$-motivic sheaves to the study of higher direct images of sheaves on the site $(\Sm/k)_{\leq 1,\et}$.
\begin{prop}\label{HI1e_*curve}
    The following diagram is commutative
    \[\xymatrix{
        D(\HI_{\leq 1}(K,\Lambda))               \ar[d]_{Re_{1*}}        \ar[r]^-{\sigma_{1*}\iota\iota_1} 
        & D(\Shv_\et^\tr(K_{\leq 1},\Lambda))    \ar[d]_{Re^{\leq 1}_*}  \ar[r]^-{\gamma_*}                       
        & D(\Shv_\et((\Sm/K)_{\leq 1},\Lambda))  \ar[d]_{Re_{*}^{\leq 1}}\\
        D(\HI_{\leq 1}(k,\Lambda))               \ar[r]^-{\sigma_{1*}\iota\iota_1}  
        & D(\Shv_\et^\tr(k_{\leq 1},\Lambda))    \ar[r]^-{\gamma_*}
        & D(\Shv_\et((\Sm/k)_{\leq 1},\Lambda)).
    }\]
\end{prop}
\begin{proof}
    Since $e^*_1$ is exact by Corollary~\ref{HI1exact}~(3), the functor $e^*_\DM \colon  \DM_{\leq 1}(k,\Lambda) \to \DM_{\leq 1}(K,\Lambda)$ corresponds to the derived functor
    \[e^*_1 \colon  D(\HI_{\leq 1}(k,\Lambda)) \longrightarrow D(\HI_{\leq 1}(K,\Lambda)).\] 
    Then we get the assertion by taking right adjoint to the one in Lemma~\ref{Le^*DM} for $n=1$.
\end{proof}

Denote by $\delta$ the inclusion functor $\HI_{\leq 0}\hookrightarrow \HI_{\leq 1}$, and denote by $\theta$ the obvious inclusion of \'etale sites $\Et/k \hookrightarrow (\Sm/k)_{\leq 1}$. Then $\theta^* \colon  \Shv_\et(\Et/k,\Lambda) \hookrightarrow \Shv_\et((\Sm/k)_{\leq 1},\Lambda)$ is canonically isomorphic to the composition of the following functors
\[\Shv_\et(\Et/k,\Lambda) \stackrel{\gamma^*\sigma^*}{\simeq} \HI_{\leq 0}(k,\Lambda) \stackrel{\delta}{\hookrightarrow} \HI_{\leq 1}(k,\Lambda) \stackrel{\sigma_{1*}\iota\iota_1}{\hookrightarrow} \Shv_\et^\tr(k_{\leq 1},\Lambda) \stackrel{\gamma_*}{\to} \Shv_\et((\Sm/k)_{\leq 1},\Lambda).\]  

The higher direct images of $0$-motivic sheaves are compatible with the higher direct images of $1$-motivic sheaves in the following sense:  
\begin{thm}\label{HI0HI1Re*}
    For $\sF\in\HI_{\leq 0}(K,\Lambda)$, we have a canonical isomorphism
    \[\delta R^ie_{0*}\sF \stackrel{\sim}{\longrightarrow} R^ie_{1*}\delta\sF,\]
    where $e_{n*}$ are the direct images of $n$-motivic sheaves in Definition~\ref{HIne_*}. In particular, if $K/k$ is primary, then $R^ie_{1*}\delta\sF$ is the $0$-motivic sheaf associated with the $\Gal(k_s/k)$-module $H^i(\Gamma,\sF_{K_s})$, where $\Gamma=\Gal(K_s/Kk_s)$.
\end{thm}
\begin{proof}
    By Corollary~\ref{SmBCfield}, for $\sG\in\Shv_\et(\Et/K,\Lambda)$, the base change morphism
    \[\theta^* R^i e_*^{\leq 0}\sG \longrightarrow R^ie_*^{\leq 1}\theta^*\sG\]
    is an isomorphism for any $i$. Then by Proposition~\ref{HI1e_*curve}, we have
    \begin{align*}
        \gamma_*\sigma_{1*}\iota\iota_1\delta R^ie_{0*}\gamma^*\sigma^*\sG
        &\simeq \gamma_*\sigma_{1*}\iota\iota_1\delta \gamma^*\sigma^* R^ie_*^{\leq 0}\sG \\
        &\simeq \theta^* R^ie_*^{\leq 0}\sG\\
        &\simeq R^ie_*^{\leq 1}\theta^*\sG\\
        &\simeq R^ie_*^{\leq 1} \gamma_*\sigma_{1*}\iota\iota_1\delta \gamma^*\sigma^* \sG \\
        &\simeq \gamma_*\sigma_{1*}\iota\iota_1 R^ie_{1*} \delta \gamma^*\sigma^* \sG.
    \end{align*}
    By Proposition \ref{Shv_et^tr} and Lemma \ref{HInEmbedShvn}, the functor $\gamma_*\sigma_{1*}\iota\iota_1 \colon  \HI_{\leq 1}(k,\Lambda) \to \Shv_\et((\Sm/k)_{\leq 1},\Lambda)$ is conservative. Thus we get a canonical isomorphism
    \[\delta R^ie_{0*}\gamma^*\sigma^*\sG \stackrel{\sim}{\longrightarrow} R^ie_{1*}\delta\gamma^*\sigma^*\sG.\]
    By Proposition~\ref{HI0} (1), every $\sF\in\HI_{\leq 0}(K,\Lambda)$ is of the form $\gamma^*\sigma^*\sG$, which completes the proof.

    The last assertion follows from the fact that if $K/k$ is primary, then $e_{0*}$ corresponds to the functor $M\mapsto M^\Gamma$, which has been used in the proof of Proposition~\ref{M0FullFaithAdj}.
\end{proof}

\begin{cor}\label{Rie*torsion}
    \begin{enumerate}[leftmargin=*,label={\rm(\arabic*)}]
        \item If $\sF\in\HI_{\leq 0}(K,\Lambda)$, then $R^ie_{1*}\sF$ are $0$-motivic sheaves for all $i\geq 0$ and are torsion sheaves for all $i\geq 1$. \menum
        \item For $\sF\in\HI_{\leq 1}(K,\Lambda)$, we have $R^ie_{1*}\sF$ are torsion $0$-motivic sheaves for all $i\geq 2$.
    \end{enumerate}
\end{cor}
\begin{proof}
    If $\sF\in\HI_{\leq 0}(K,\Lambda)$, then $R^ie_{1*}\sF$ are $0$-motivic sheaves for all $i\geq 0$ by Theorem~\ref{HI0HI1Re*}. 
    
    For arbitrary $\sF\in\HI_{\leq 1}(K,\Lambda)$, consider the exact sequence of $1$-motivic sheaves
    \[0\to \sF' \to \sF \to \sF\otimes\bQ \to \sF'' \to 0.\]
    By Suslin's rigidity theorem (\cite[Theorem~7.20]{MVW06Motive}), the torsion sheaves $\sF'$ and $\sF''$ are $0$-motivic sheaves. Thus $R^ie_{1*}\sF'$ and $R^ie_{1*}\sF''$ are $0$-motivic sheaves for all $i$ and are in fact torsion sheaves by \cite[Expos\'e~IX, Proposition~1.2(v)]{SGA4III}; see also \cite[\href{https://stacks.math.columbia.edu/tag/0DDD}{Lemma~0DDD}]{stacks-project}. Note that $R^ie_*^{\leq 1}(\gamma_*\sigma_{1*}\iota\iota_1\sF\otimes\bQ)$ is the \'etale sheaf associated with the presheaf
    \begin{align*}
        (\Sm/k)_{\leq 1} & \longrightarrow         \Lambda\text{-}\Mod\\
        X                & \longmapsto  H^i_\et(X_K, \gamma_*\sigma_{1*}\iota\iota_1\sF\otimes\bQ).
    \end{align*}
    For $X\in(\Sm/k)_{\leq 1}$ and $i\geq 2$, 
    \[H^i_\et(X_K,\gamma_*\sigma_{1*}\iota\iota_1\sF\otimes\bQ)\simeq H^i_{\rm Nis}(X_K,\gamma_*\sigma_{1*}\iota\iota_1\sF\otimes\bQ)=0,\]
    where the first isomorphism holds by \cite[Proposition~14.23]{MVW06Motive}, and the second one holds because the Nisnevich cohomological dimension is bounded by the Krull dimension (\cite[\S 3.1, Proposition~1.8]{MV99A1homotopy}). By Proposition~\ref{HI1e_*curve}, for $i\geq 2$,
    \[\gamma_*\sigma_{1*}\iota\iota_1 R^ie_{1*}(\sF\otimes\bQ) \simeq R^ie_*^{\leq 1}(\gamma_*\sigma_{1*}\iota\iota_1\sF\otimes\bQ)=0.\]
    By Proposition~\ref{Shv_et^tr} and Lemma~\ref{HInEmbedShvn}, the functor $\gamma_*\sigma_{1*}\iota\iota_1 \colon  \HI_{\leq 1}(k,\Lambda) \to \Shv_\et((\Sm/k)_{\leq 1},\Lambda)$ is conservative, which implies that
    \[R^ie_{1*}(\sF\otimes\bQ)=0 \text{\quad for \quad} i\geq 2.\] 
    Then split the above exact sequence to two short exact sequences and consider the induced long exact sequences of cohomology. Noting that $\HI_{\leq 0}$ is a Serre subcategory of $\HI_{\leq 1}$ (Corollary~\ref{HI1exact} (1)), we obtain that $R^ie_{1*}\sF$ are torsion $0$-motivic sheaves for $i\geq 2$. 
    
    In particular, if $\sF$ is a $0$-motivic sheaf, then all the sheaves in the above exact sequence are in fact $0$-motivic sheaves and $R^ie_{1*}(\sF\otimes\bQ)=0$ for $i\geq 1$. Then the same argument as above shows that $R^ie_{1*}\sF$ are torsion $0$-motivic sheaves for $i\geq 1$. 
\end{proof}

\subsection{Inverse images of semi-abelian varieties}
We consider sheaves defined by commutative group schemes.
\begin{defn}
    Let $G$ be a commutative group scheme over $k$. Denote $\underline{G}$ the abelian sheaf on $(\Sm/k)_\et$ defined by $G$, i.e.,
    \[\underline{G}(U)=\Mor_{\Sm/k}(U,G) \text{\quad for\quad} U\in\Sm/k.\]
    Denote $\underline{G}_\Lambda=\underline{G}\otimes_\bZ\Lambda$ the presheaf tensor product, i.e.,
    \[\underline{G}_\Lambda(U)=\Mor_{\Sm/k}(U,G)\otimes_\bZ\Lambda \text{\quad for\quad} U\in\Sm/k.\]
    Then $\underline{G}_\Lambda$ is a sheaf of $\Lambda$-modules on $(\Sm/k)_\et$.
\end{defn}

There are transfer structures on such sheaves.
\begin{lem}[{\cite[Proof of Lemma~3.2]{SS03Albanese}, \cite[Lemmas 3.1.2 and 3.3.1]{Orgogozo041Motives}}]
    Let $G$ be a commutative group scheme over $k$. Then $\underline{G}_\Lambda$ has a canonical structure of \'etale sheaves with transfers, which is functorial. More precisely, there exists a unique \'etale sheaf with transfers $\underline G^\tr_\Lambda$ such that 
    \[\gamma_* \underline G^\tr_\Lambda \simeq \underline G_\Lambda,\]
    where $\gamma_* \colon  \Shv_\et^\tr(k,\Lambda)\to \Shv_\et(\Sm/k,\Lambda)$ is the forgetful functor. Moreover, if $G$ is a commutative \'etale group scheme or a semi-abelian variety, then $\underline G^\tr_\Lambda$ is homotopy invariant.
\end{lem}

The aim of this subsection is to show the following result.
\begin{prop}\label{SAVpullback}
    Let $K/k$ be a field extension and let $G$ be a commutative \'etale group scheme or a semi-abelian variety over $k$. Then 
    \[e^*_\HI(\underline{G}_\Lambda^\tr) \simeq \underline{G_K}_\Lambda^\tr.\]
\end{prop}

We start with the analogue for sheaves without transfers. Denote $e_*$ (resp. $e^*$) the direct image (resp. inverse image) functor of sheaves on the smooth-\'etale sites. The following result is standard and well-known.

\begin{lem}\label{e^*G=GK}
    Let $K/k$ be a field extension and $G$ be a commutative smooth group scheme over $k$. Then we have a canonical isomorphism
    \[e^*(\underline G_\Lambda)\simeq \underline{G_K}_\Lambda.\]
\end{lem}
\begin{proof}
    In this proof, by abuse of notation, we use $\Sm/k$ to mean the category of smooth separated schemes locally of finite type over $k$ rather than the full subcategory of smooth separated schemes of finite type over $k$, which is used in other places of this chapter. By \cite[Expos\'e~III, Th\'eor\`eme~4.1]{SGA4I}, these two categories give the same category of \'etale sheaves.  
    
    For $X\in\Sm/k$, denote $\Lambda(X)$ the \'etale sheaf associated with the presheaf mapping $U\in\Sm/k$ to the free $\Lambda$-module generated by $\Mor_k(U,X)$. By Yoneda's lemma, we have that
    \[e^*\Lambda(X) \simeq \Lambda(X_K).\]
    Recall the following exact sequence of \'etale sheaves
    \[\Lambda(G\times_k G) \longrightarrow \Lambda(G) \longrightarrow \underline G_\Lambda \longrightarrow 0,\]
    where the first map sends a generator $[(a_1,a_2)]$ to $[a_1]+[a_2]-[a_1 + a_2]$, and the second map sends a generator $[g]$ to $g$. Then the following commutative diagram with exact rows
    \[\xymatrix{ 
        e^*\Lambda(G \times_k G)     \ar[r] \ar[d]^\simeq  
        & e^*\Lambda(G)              \ar[r] \ar[d]^\simeq
        & e^*(\underline{G}_\Lambda) \ar[r] \ar[d]
        & 0 \\
        \Lambda(G_K \times_K G_K)    \ar[r]
        & \Lambda(G_K)               \ar[r] 
        & \underline{G_K}_\Lambda    \ar[r]
        & 0
    }\]
    gives us the desired isomorphism.
\end{proof}

We use the Frobenius and Verschiebung morphisms to deal with the direct images of commutative (flat) group schemes  in the case of purely inseparable field extensions. 
\begin{lem}\label{e_*GK=G}
    Let $k$ be a field of characteristic $p>0$ and let $K/k$ be a purely inseparable extension. Let $G$ be a commutative group scheme locally of finite type over $k$. Then we have a canonical isomorphism
    \[\underline{G}_\Lambda \simeq e_*(\underline{G_K}_\Lambda),\]
    In other words, for $X\in\Sm/k$, there is a canonical isomorphism 
    \[\Mor_k(X,G)\otimes_\bZ\Lambda \stackrel{\sim}{\longrightarrow} \Mor_K(X_K,G_K)\otimes_\bZ\Lambda.\]
\end{lem}
\begin{proof}
    We divide the proof into several steps: 
    \begin{enumerate}[leftmargin=*,label={\rm(\alph*)}]
        \item Let us start with the case $K=k^{1/p}$. Note that the map $k^{1/p}\to k,~x\mapsto x^p$ is an isomorphism with inverse $x\mapsto x^{1/p}$. So we reduce to show the isomorphism
        \[\Mor_k(X,G)\otimes_\bZ\Lambda \stackrel{\sim}{\longrightarrow} \Mor_k(X^{(p)},G^{(p)})\otimes_\bZ\Lambda,\]
        where $X^{(p)}$ (resp. $G^{(p)}$) is the base change of $X$ (resp. $G$) along the absolute Frobenius of $\Spec k$. By Lemma~\ref{fpqcdescent} (1), this map is injective. Now, we show that it is surjective. Let $f \colon  X^{(p)}\to G^{(p)}$ be a morphism of $k$-schemes. Denote $F_{X/k} \colon  X\to X^{(p)}$ (resp. $F_{G/k} \colon  G\to G^{(p)}$) the relative Frobenius of $X/k$ (resp. $G/k$). Note that $G$ is always flat over $k$. Thus by \cite[Expos\'e~VII, 4.3]{SGA3I}, there exists a Verschiebung morphism $V_{G/k} \colon  G^{(p)}\to G$ such that 
        \[V_{G/k}\circ F_{G/k}=p\id_G \text{\quad and \quad} F_{G/k}\circ V_{G/k}=p\id_{G^{(p)}}.\] 
        Let $f_0$ be the composition 
        \[\xymatrix{
            X \ar[r]^-{F_{X/k}} & X^{(p)} \ar[r]^f & G^{(p)} \ar[r]^-{V_{G/k}} & G,
        }\]
        and let $f_0^{(p)} \colon  X^{(p)} \to G^{(p)}$ be the base change of $f_0 \colon  X\to G$. Then 
        \[f_0^{(p)}\circ F_{X/k}=F_{G/k}\circ f_0=F_{G/k}\circ V_{G/k} \circ f \circ F_{X/k}=p\circ f\circ F_{X/k},\]
        Note that $X^{(p)}$ is reduced (because $X$ is smooth), $G^{(p)}$ is separated over $\Spec k$, and $F_{X/k}$ is surjective. By \cite[Propositions 11.10.4 and 11.10.1 (d)]{EGAIV3}, we obtain 
        \[f_0^{(p)}=p\circ f.\]
        \item We deal with the case $K=k^\perf=\bigcup_{n\in\bN} k^{1/p^n}$, a perfect closure of $k$. By (a), we have the expected isomorphisms for $K=k^{1/p^n}$. Note that $X \to \Spec k$ is quasi-compact and quasi-separated and that $G \to \Spec k$ is locally of finite presentation. Thus by \cite[Th\'eor\`eme~8.8.2]{EGAIV3}, there is a canonical isomorphism
        \[\varinjlim_n \Mor_{k^{1/p^n}}(X_{k^{1/p^n}},G_{k^{1/p^n}}) \stackrel{\sim}{\longrightarrow} \Mor_{k^\perf}(X_{k^\perf}, G_{k^{\perf}}),\]
        which implies the expected result.
        \item Now, we prove the assertion for general purely inseparable field extensions. Let $k^\perf$ be a perfect closure of $k$. Then $k^\perf$ is also a perfect closure of $K$ because $K/k$ is purely inseparable. By (b), we have the two isomorphisms in the following commutative diagram
        \[\xymatrix{
            \Mor_k(X,G)\otimes_\bZ\Lambda \ar[rr] \ar[dr]_{\sim} && \Mor_K(X_K,G_K)\otimes_\bZ\Lambda \ar[ld]^{\sim} \\
            &\Mor_{k^\perf}(X_{k^\perf}, G_{k^{\perf}})\otimes_\bZ\Lambda.
        }\]
        It follows that the horizontal arrow is also an isomorphism. \qedhere
    \end{enumerate}
\end{proof}

\begin{rmk}
    The above lemma is false before inverting $p$. For example, let $X=\Spec k$ and $G=\bG_a$. Then $\Mor_k(X,G)\simeq k$ and $\Mor_K(X_K,G_K)\simeq K$. The natural inclusion $k\hookrightarrow K$ is not an isomorphism in general.
\end{rmk}

\begin{lem}\label{e_HI^*h0G}
    Let $K/k$ be a field extension and $G$ be a commutative smooth group scheme over $k$. Then 
    \[e^*_\HI h_0^\et \gamma^* \underline{G}_\Lambda \simeq h_0^\et \gamma^* \underline{G_K}_\Lambda.\]
\end{lem}
\begin{proof}
    In fact, we have
    \[e^*_\HI h_0^\et\gamma^*\underline{G}_\Lambda \simeq h_0^\et e^*_\tr\gamma^*\underline{G}_\Lambda \simeq h_0^\et \gamma^* e^* \underline{G}_\Lambda \simeq h_0^\et \gamma^* \underline{G_K}_\Lambda,\]
    where the first and the last isomorphisms hold by Corollary~\ref{HIadj} (2) and Lemma~\ref{e^*G=GK} respectively, and the second isomorphism is obtained by taking left adjoint to $\gamma_*e_*^\tr \simeq e_*\gamma_*$ (part of the commutative diagram before Lemma~\ref{sigma^*}).
\end{proof}

\begin{lem}\label{h0gamma^*G=G^tr}
    Let $G$ be a commutative \'etale group scheme or a semi-abelian variety over $k$. Then 
    \[h_0^\et\gamma^* \underline{G}_\Lambda \simeq \underline{G}_\Lambda^\tr.\]
\end{lem}
\begin{proof}
    For a perfect field $k$, Barbieri-Viale and Kahn \cite[Lemma~3.9.2]{BVK16Derived1Motives} showed that the composition of the forgetful functors \[\HI_\et(k,\Lambda) \stackrel{\iota}{\longrightarrow} \Shv_\et^\tr(k,\Lambda) \stackrel{\gamma_*}{\longrightarrow} \Shv_\et(\Sm/k,\Lambda)\] 
    is fully faithful. Thus we have a natural isomorphism
    \[h_0^\et\gamma^*\gamma_*\iota (\underline{G}_\Lambda^\tr) \stackrel{\sim}{\longrightarrow} \underline{G}_\Lambda^\tr.\]
    In other words, we have the expected isomorphism over perfect fields.

    Now, we extend it to the general field $k$. Let $K$ be a perfect closure of $k$. Reformulating Lemma~\ref{e_*GK=G}, we have $\gamma_*\underline{G}_\Lambda^\tr \simeq \gamma_* e_*^\tr \underline{G_K}_\Lambda^\tr$. Since $\gamma_*$ is conservative, we have 
    \[\underline{G}_\Lambda^\tr \simeq e_*^\HI \underline{G_K}_\Lambda^\tr.\]
    Consider the following commutative diagram
    \[\xymatrix{
        h_0^\et \gamma^*\underline{G}_\Lambda                  \ar[rr]       \ar[d]_\simeq 
        && \underline{G}_\Lambda^\tr                           \ar[d]^\simeq \\
        e_*^\HI e^*_\HI h_0^\et \gamma^*\underline{G}_\Lambda  \ar[r]^-\simeq
        & e_*^\HI  h_0^\et \gamma^*\underline{G_K}_\Lambda     \ar[r]^-\simeq
        & e_*^\HI  \underline{G_K}_\Lambda^\tr,
    }\]
    Here, the left vertical arrow is an isomorphism by Proposition~\ref{PurelyInsepEquiv}, and the two horizontal arrows at the bottom are isomorphisms successively by Lemma~\ref{e_HI^*h0G} and the assertion over perfect fields. Thus the horizontal arrow on top is also an isomorphism.
\end{proof}

\begin{rmk}
    In \cite[Proposition~3.10]{AHPL16DecompMotiveGrp}, they proved using qfh topology that if $S$ is an excellent scheme and $G$ is a commutative smooth group scheme over $S$, then the co-unit
    \[\gamma^* \gamma_* \underline{G}_\bQ^\tr \longrightarrow \underline{G}_\bQ^\tr\]
    is an isomorphism of \'etale sheaves with transfers. 
\end{rmk}

We are now ready to prove the main result of this subsection:
\begin{proof}[Proof of Proposition~\ref{SAVpullback}]
    Consider the following commutative diagram
    \[\xymatrix{
        e^*_\HI h_0^\et \gamma^*\underline{G}_\Lambda  \ar[r]^-\simeq \ar[d]_\simeq
        & e^*_\HI(\underline{G}_\Lambda^\tr)           \ar[d] \\                 
        h_0^\et \gamma^*\underline{G_K}_\Lambda        \ar[r]^-\simeq
        & \underline{G_K}_\Lambda^\tr,
    }\]
    where the two horizontal isomorphisms hold by Lemma~\ref{h0gamma^*G=G^tr} and the left vertical isomorphism holds by Lemma~\ref{e_HI^*h0G}. Thus the right vertical arrow is also an isomorphism.
\end{proof}

From now on, we shall write the \'etale sheaves with transfers $\underline{G}_\Lambda^\tr$ as $\underline{G}_\Lambda$, or even as $\underline{G}$, $G$ for simplicity if it does not cause confusion.

\subsection{Direct images of semi-abelian varieties}
In this subsection, we show that the Chow trace of a semi-abelian variety is the ``connected component'' of its direct image.

Following \cite{ABV09MotShvLAlb}, we call a commutative group scheme $G$ over $k$ a semi-abelian group scheme if its connected component of the identity $G^0$ is a semi-abelian variety and $\pi_0(G)$ is a constructible group scheme. As explained in \cite[comments before 1.3.1 and Corollary~1.3.5]{ABV09MotShvLAlb}, semi-abelian group schemes are $1$-motivic sheaves.

\begin{defn}[{\cite[Definition 1.3.7]{ABV09MotShvLAlb}}]\label{fpHI1}
    A $1$-motivic sheaf $\sF$ is said to be finitely generated if there exist a semi-abelian group scheme $G$ and an epimorphism $q \colon \underline{G}_\Lambda \to \sF$. Moreover, if $\ker(q)$ is finitely generated, then $\sF$ is said to be finitely presented.  
\end{defn}

\begin{prop}[{\cite[1.3.8]{ABV09MotShvLAlb}}]\label{StructureHI1}
    \begin{enumerate}[leftmargin=*,label={\rm(\arabic*)}]
        \item Let $\sF$ be a finitely presented $1$-motivic sheaf. Then there is a unique and functorial exact sequence
            \[0 \to \underline{\sL}_\Lambda \to \underline{\sG}_\Lambda \to \sF \to 0\]
            where $\sG$ is a semi-abelian group scheme and $\sL$ is a lattice. \menum
        \item Let $\sF$ be a $1$-motivic sheaf. Then $\sF$ is a filtered colimit of finitely presented $1$-motivic sheaves.
    \end{enumerate}
\end{prop}

\begin{rmk}
    See also \cite[Chapter 3]{BVK16Derived1Motives} for some basic properties of finitely presented $1$-motivic sheaves.
\end{rmk}

\begin{cor}[cf. {\cite[Proposition~3.3.4]{BVK16Derived1Motives} and \cite[Theorem~1.3.10]{ABV09MotShvLAlb}}]\label{ExactfpHI1}
    Let $\sF$ be a finitely presented $1$-motivic sheaf. Then there exists an exact sequence in $\HI_{\leq 1}(k,\Lambda)$:
    \[0 \to \underline{L}_\Lambda \to \underline{G}_\Lambda \to \sF \to \underline{E}_\Lambda \to 0,\]
    where $L$ is a lattice, $G$ is a semi-abelian variety and $E$ is a constructible group scheme. Moreover, $\underline{E}_\Lambda=\pi_0(\sF)$, where $\pi_0 \colon \Shv_\et^\tr(k,\Lambda) \longrightarrow \HI_{\leq 0}(k,\Lambda)$ is a left adjoint to the inclusion functor (Proposition~\ref{HI0}).
\end{cor}
\begin{proof}
    By Proposition~\ref{StructureHI1}, there is a unique and functorial exact sequence
    \[0 \to \underline{\sL}_\Lambda \to \underline{\sG}_\Lambda \to \sF \to 0\]
    where $\sG$ is a semi-abelian group scheme and $\sL$ is a lattice. Let $G$ be the identity component of $\sG$ and let $\pi_0(\sG)$ be the quotient group scheme $\sG/\sG^0$. Let $L$ and $E$ be the kernel and cokernel of the induced morphism $\sL \to \pi_0(\sG)$ respectively. Applying the snake lemma to the following commutative diagram with exact rows
    \[\xymatrix{
        0 \ar[r] & \underline{\sL}_\Lambda          \ar[r] \ar[d] & \underline{\sG}_\Lambda        \ar[r] \ar[d] & \sF \ar[r] \ar[d] & 0\\
        0 \ar[r] & \underline{\pi_0(\sG)}_\Lambda   \ar[r]        & \underline{\pi_0(\sG)}_\Lambda \ar[r]        & 0    \ar[r]        & 0, 
    }\]
    we get the expected exact sequence. Applying $\pi_0$ to the top row of the above diagram, we have the following exact sequence
    \[\underline{\sL}_\Lambda \to \underline{\pi_0(\sG)}_\Lambda \to \pi_0(\sF) \to 0.\]
    Thus $\pi_0(\sF)$ is the cokernel $\underline{E}_\Lambda$.
\end{proof}

\begin{defn}
    \begin{enumerate}[leftmargin=*,label={\rm(\arabic*)}]
        \item For $\sF\in\Shv_\et^\tr(k,\Lambda)$, we denote $\sF^0 \colonequals \ker(\sF\to \pi_0(\sF))$ and call it the connected component of $\sF$. \menum
        \item A sheaf $\sF$ is called connected if $\pi_0(\sF)=0$. Denote by $\HI_{\leq 1}^0(k,\Lambda)$ the category of connected $1$-motivic sheaves.
    \end{enumerate}
\end{defn}

\begin{lem}\label{connected}
    \begin{enumerate}[leftmargin=*,label={\rm(\arabic*)}]
        \item Let $\sF\in\HI_{\leq 1}(k,\Lambda)$ be a $1$-motivic sheaf. Then $\sF^0$ is connected. In particular, there is no nontrivial morphism from $\sF^0$ to any $0$-motivic sheaves. \menum
        \item Every connected $1$-motivic sheaf $\sF$ is a filtered colimit of finitely presented connected $1$-motivic sheaves.
    \end{enumerate}
\end{lem}
\begin{proof}
    By Proposition~\ref{StructureHI1} (2), we write $\sF$ as a filtered colimit of finitely presented $1$-motivic sheaves $\sF=\varinjlim_i \sF_i$. Since $\pi_0$ commutes with colimits (as a left adjoint) and filtered colimits are exact, we have that $\sF^0=\varinjlim_i \sF_i^0$. By Corollary~\ref{ExactfpHI1}, $\pi_0(\sF_i^0)=0$. Thus $\pi_0(\sF^0)=\varinjlim_i \pi_0(\sF_i^0)=0$.

    If $\sF$ is connected, then $\sF^0\simeq \sF$. By the above argument, the sheaf $\sF^0$ is a filtered colimit of finitely presented connected $1$-motivic sheaves. Thus $\sF$ is also.
\end{proof}

Let $K/k$ be a field extension. Then the composition
\[\HI_{\leq 1}(K,\Lambda) \stackrel{e_{1*}}{\longrightarrow} \HI_{\leq 1}(k,\Lambda) \stackrel{(\cdot)^0}{\longrightarrow} \HI_{\leq 1}^0(k,\Lambda)\]
is right adjoint to the composition
\[\HI_{\leq 1}^0(k,\Lambda) \hookrightarrow \HI_{\leq 1}(k,\Lambda) \stackrel{e_1^*}{\longrightarrow} \HI_{\leq 1}(K,\Lambda).\]

\begin{thm}\label{ChowRevisited}
    Let $K/k$ be a primary field extension and $G/K$ be a semi-abelian variety. Then the connected $1$-motivic sheaf \( (e_{1*}(\underline{G}_\Lambda))^0 \) is represented by the Chow trace $\pi_*G$, and the $0$-motivic sheaf \(\pi_0(e_{1*}(\underline{G}_\Lambda))\) is the sheaf associated with the $\Gal(k_s/k)$-$\Lambda$-module 
    \[\LN(G,Kk_s/k_s)_\Lambda \colonequals G(Kk_s)/(\pi_*G)(k_s)\otimes_\bZ\Lambda.\]
    In other words, we have an exact sequence of \(1\)-motivic sheaves
    \[0\to \underline{\pi_*G}_\Lambda \to e_{1*}G \to \LN(G,Kk_s/k_s)_\Lambda\to 0.\]
\end{thm}
\begin{proof}
    By Proposition~\ref{SAVpullback} and Lemma~\ref{HIne^*}, we have
    \[e_1^*(\underline{\pi_*G}_\Lambda) \simeq \underline{(\pi_*G)_K}_\Lambda.\]
    Then the co-unit $\pi^*\pi_* \to \id$ induces a morphism $e_1^*(\underline{\pi_*G}_\Lambda) \to \underline{G}_\Lambda$. It suffices to show that for $\sF\in\HI_{\leq 1}^0(k,\Lambda)$, there exists a unique morphism $f \colon  \sF \to \underline{\pi_*G}_\Lambda$ such that the following diagram commutes
    \[\xymatrix{
        e_1^*(\sF)                       \ar[rd] \ar@{.>}[d]_{e_1^*(f)} \\
        e_1^*(\underline{\pi_*G}_\Lambda) \ar[r] & \underline{G}_\Lambda.
    }\] 
    By Lemma~\ref{connected}, we may and do assume that $\sF$ is a finitely presented connected $1$-motivic sheaf. Let 
    \[0\to \underline{L}_\Lambda \stackrel{a}{\to} \underline{G'}_\Lambda \stackrel{b}{\to }\sF \to 0\] 
    be the presentation of $\sF$ in Corollary~\ref{ExactfpHI1}. By Corollary~\ref{HI1exact} (3), the inverse image $e_1^*$ of $1$-motivic sheaves is exact. Thus we get another exact sequence
    \[0\to e_1^*(\underline{L}_\Lambda) \to e_1^*(\underline{G'}_\Lambda) \to e_1^*(\sF) \to 0.\]
    Denote by $[\underline{L}_\Lambda \to \underline{G'}_\Lambda]$ the complex of $1$-motivic sheaves concentrated in degrees $0$ and $1$. Other similar notations below have a similar meaning. Then
    \begin{align*}
        \Hom_{\HI_{\leq 1}^0(k,\Lambda)}(\sF,\underline{\pi_*G}_\Lambda) 
        &\simeq \Hom_{C(\HI_{\leq 1}(k,\Lambda))}([\underline{L}_\Lambda \to \underline{G'}_\Lambda],[0\to \underline{\pi_*G}_\Lambda])\\
        &\simeq \Hom_{\M_1(k)}([L \to G'],[0\to \pi_*G])\otimes_\bZ\Lambda\\
        &\simeq \Hom_{\M_1(K)}([L_K\to G'_K],[0\to G])\otimes_\bZ\Lambda\\
        &\simeq \Hom_{C(\HI_{\leq 1}(K,\Lambda))}([\underline{L_K}_\Lambda \to \underline{G'_K}_\Lambda],[0\to \underline{G}_\Lambda])\\
        &\simeq \Hom_{C(\HI_{\leq 1}(K,\Lambda))}([e_1^*(\underline{L}_\Lambda) \to e_1^*(\underline{G'}_\Lambda)],[0\to \underline{G}_\Lambda])\\
        &\simeq \Hom_{\HI_{\leq 1}(K,\Lambda)}(e_1^*\sF, \underline{G}_\Lambda),
    \end{align*}
    where the first and the last isomorphism hold by the above exact sequences, the second and the fourth isomorphisms hold by \cite[3.3.4 d)]{BVK16Derived1Motives}, the third isomorphism holds by construction of Chow trace, and the second to the last isomorphisms hold by Proposition~\ref{SAVpullback}. 

    By the above argument, we have $\underline{\pi_*G}_\Lambda \simeq (e_{1*}\underline{G}_\Lambda)^0$. Consider the exact sequence
    \[0\to \underline{\pi_*G}_\Lambda \to e_{1*}(\underline{G}_\Lambda) \to \pi_0(e_{1*}(\underline{G}_\Lambda)) \to 0.\]
    Taking the stalk at the geometric point $\Spec k_s\to \Spec k$, we get 
    \[0\to (\pi_*G)(k_s)_\Lambda \to G(Kk_s)_\Lambda \to \pi_0(e_{1*}(\underline{G}_\Lambda))_{k_s} \to 0.\]
    Thus $\pi_0(e_{1*}(\underline{G}_\Lambda))$ is the $0$-motivic sheaf associated with $\LN(A,Kk_s/k_s)_\Lambda$.
\end{proof}

\begin{cor}\label{LNe*A}
    Let $K/k$ be a finitely generated regular extension and let $A$ be an abelian variety over $K$. Then the $1$-motivic sheaf $e_{1*}(\underline{A}_\Lambda)$ is finitely presented.
\end{cor}
\begin{proof}
    The Lang-N\'eron theorem (\cite[Theorem~7.1]{Conrad06ChowLN}) says that $A(Kk_s)/(\pi_*A)(k_s)$ is a finitely generated abelian group. This means that $\pi_0(e_{1*}(\underline{A}_\Lambda))$ is a finitely presented $1$-motivic sheaf. Since every extension of two finitely presented $1$-motivic sheaves is still finitely presented, the $1$-motivic sheaf $e_{1*}(\underline{A}_\Lambda)$ is finitely presented as well.
\end{proof}

It is well-known that an \'etale sheaf which is an extension of two separated group schemes is represented by a separated group algebraic space and that separated group algebraic spaces over a field are represented by group schemes. In the special case of the exact sequence in Theorem~\ref{ChowRevisited}, we can prove the representability without using algebraic spaces. See Proposition~\ref{extRep}. Thus we have the following result.
\begin{cor}
    Let $K/k$ be a primary extension of fields of characteristic $0$ and let $G$ be a semi-abelian variety over $K$. Then the $1$-motivic sheaf $e_{1*}G$ is represented by a semi-abelian group scheme over $k$.
\end{cor}

For an abelian variety $A$ over $K$, we have gotten some information about $e_{1*}(A)$ and $R^ie_{1*}(A)$ for $i\geq 2$ by Theorem~\ref{ChowRevisited} and Corollary~\ref{Rie*torsion}. The following result is about the first direct image. It uses Raynaud's results on torsors under abelian schemes.
\begin{thm}\label{Rie*A}
    Let $K/k$ be a field extension, and let $A$ be an abelian variety over $K$. Then $R^ie_{1*}(A)$ is a torsion $0$-motivic sheaf for $i\geq 1$.
\end{thm}
\begin{proof}
    We use the same argument as in the proof of Corollary~\ref{Rie*torsion}. It suffices to show that $R^1e_*^{\leq 1}(A\otimes\bQ)=0$, where $e_*^{\leq 1}$ is the direct image functor for \'etale sheaves on $(\Sm)_{\leq 1}$. Note that $R^1e_{*}^{\leq 1}(A\otimes\bQ)$ is the \'etale sheafification of the presheaf 
    \begin{align*}
        (\Sm/k)_{\leq 1} &\longrightarrow \Lambda\text{-}\Mod,\\ 
        X                &\longmapsto H^1_\et(X_K,A)\otimes\bQ.
    \end{align*} 
    Since $X$ is noetherian and regular, torsors under the abelian scheme $A_X$ are torsion, i.e., $H^1_\et(X_K,A)$ is a torsion group by \cite[Proposition~XIII 2.6.(ii) and Proposition~XIII 2.3.(ii)]{Raynaud70AmpleShv}. Thus $H^1_\et(X_K,A)\otimes\bQ=0$, which implies that $R^1e_*^{\leq 1}(A\otimes\bQ)=0$.
\end{proof}

We can also study $R^1e_{1*}\bG_m$ in some interesting cases.  
\begin{thm}\label{R1e*Gm}
    Let $f\colon X\to \Spec k$ be a smooth projective and geometrically connected variety, and let $K$ be the function field of $X$. Then we have an exact sequence 
    \[ \Div^0(X_{k_s}) \to \Pic^0_{X/k} \to R^1e_{1*}\bG_m \to 0, \]
    where $\Div^0(X_{k_s})$ is $\Gal(k_s/k)$-module of divisors on $X_{k_s}$ algebraically equivalent to zero, viewed as a locally constant \'etale sheaf. In particular, we have $R^1e_*\bG_m$ is connected.
\end{thm}
\begin{proof}
    For an affine open subscheme $u \colon U \hookrightarrow X$, we have the following exact sequence 
    \[ \bigoplus_{x\in X^{(1)}\cap {(X\setminus U)}}\bZ \to \Pic_{X/k} \to \Pic_{U/k} \to 0, \]
    where $\Pic_{X/k}$ (resp. $\Pic_{U/k}$) is the Picard functor $R^1f_*\bG_{m,X}$ (resp. $R^1(fu)_*\bG_{m,U}$). Taking colimit, we obtain
    \[ \Div_{X/k} \to \Pic_{X/k} \to R^1e_*\bG_m \to 0, \]
    where $\Div_{X/k}$ is the locally constant \'etale sheaf associated with the $\Gal(k_s/k)$-module $\Div(X_{k_s})$. Recall that $\Pic_{X/k}$ is represented by a group scheme, which is an extension of the discrete sheaf $\NS_{X/k}$ by the abelian variety $\Pic_{X/k}^0$ (Picard variety). Applying the snake lemma to the following commutative diagram with exact rows
    \[\xymatrix{
                & \Div_{X/k}  \ar[r] \ar[d] & \Pic_{X/k}  \ar[r] \ar[d] & R^1e_*\bG_m \ar[r] \ar[d] & 0\\
        0\ar[r] & \NS_{X/k}   \ar@{=}[r]    & \NS_{X/k}   \ar[r]        & 0 
    }\]
    we obtain the desired result.
\end{proof}

\section{The \texorpdfstring{$1$}{1}-motivic \texorpdfstring{$t$}{t}-structure}\label{t-structure}
In this section, we use Ayoub's way of perverting $t$-structures to get a new $t$-structure from the standard one on $D(\HI_{\leq 1})$. This new $t$-structure will be called the $1$-motivic $t$-structure, and the objects in its heart will be called $1$-motives. We shall translate the results on higher direct images of $1$-motivic sheaves to results on $1$-motives.  
\subsection{The abelian category of \texorpdfstring{$1$}{1}-motives with torsion}
Before we deal with the $1$-motivic $t$-structure, we review in this subsection the construction of the abelian category of $1$-motives with torsion. This category was introduced in \cite{BRS03DeligneConj-1-motive} (in characteristic $0$) and studied in details in \cite{BVK16Derived1Motives} over perfect fields. It contains the category of Deligne $1$-motives. We shall see in the next subsections that this abelian category can be embedded into the heart of the $1$-motivic $t$-structure.

Let $k$ be a field of exponential characteristic $p$, i.e., $p=1$ if $\Char(k)$ is zero, and $p=\Char(k)$ otherwise.
\begin{defn}[{\cite[C.1]{BVK16Derived1Motives}}]
    \begin{enumerate}[leftmargin=*,label={\rm(\arabic*)}]
        \item An effective $1$-motive with torsion over $k$ is a complex of group schemes
            \[ M=[L\stackrel{u}{\to} G],\]
            where $L\in{}^t\M_0(k)$ is a constructible group scheme and $G\in\SAV(k)$ is a semi-abelian variety.\menum
        \item An effective morphism of $1$-motives with torsion from $M=[L\stackrel{u}{\to} G]$ to $M'=[L' \stackrel{u'}{\to} G']$ is a commutative square 
            \[\xymatrix{
                L  \ar[r]^u \ar[d]_f & G \ar[d]^g \\
                L' \ar[r]^{u'} & G'
            }\]
            in the category of group schemes. Denote by $(f,g)\colon M\to M'$ such a morphism. We will denote by $\Hom_\eff(M,M')$ the abelian group of effective morphisms.
        \item We denote the category of effective $1$-motives with torsion over $k$ by ${}^t\M_1^\eff(k)$.
        \item If $[F\to 0]$ is an effective $1$-motive with $F$ a finite \'etale group scheme, then it is called a torsion $1$-motive. The full subcategory of ${}^t\M_1^\eff(k)$ consisting of torsion $1$-motives is denoted by ${}^t\M_1^\tor(k)$.
    \end{enumerate}
\end{defn}

\begin{rmk}
    If $L\in\M_0(k)$ is a lattice, then $[L\to G]$ is a Deligne $1$-motive.
\end{rmk}

\begin{defn}[{\cite[C.2.1]{BVK16Derived1Motives}}]
    An effective morphism of $1$-motives with torsion from $M=[L\stackrel{u}{\to} G]$ to $M'=[L' \stackrel{u'}{\to} G']$ is called a quasi-isomorphism of $1$-motives with torsion if it yields a pullback diagram 
    \[
    \xymatrix{
        0 \ar[r] & F \ar[r]\ar@{=}[d] & L \ar[r]\ar[d]_u  & L' \ar[d]_{u'} \ar[r] &0\\
        0 \ar[r] & F \ar[r]           & G \ar[r]          & G' \ar[r]             &0,
    }
    \]
    where $F$ is a finite \'etale group.
\end{defn}

By \cite[C.2.4]{BVK16Derived1Motives}, the class of quasi-isomorphisms is a left multiplicative system in the sense of \cite[7.1.7]{KS06CategorySheaf}. 

\begin{defn}[{\cite[C.3.1]{BVK16Derived1Motives}}]
    The category ${}^t\M_1(k)$ of $1$-motives with torsion is the localization of ${}^t\M_1^\eff(k)$ with respect to the multiplicative class of quasi-isomorphisms. In other words, the objects of ${}^t\M_1(k)$ are the same as the objects in ${}^t\M_1^\eff(k)$, and the Hom-sets are given by the formula
    \[ \Hom_{{}^t\M_1(k)}(M,M')=\varinjlim_{(\widetilde{M}\to M) \text{ q.i.}}\Hom_\eff(\widetilde M,M'),\]
    where the colimit is taken over the co-filtrant category of all quasi-isomorphisms $\widetilde M\to M$.
\end{defn}

\begin{rmk}
    By definition, there exist no nontrivial quasi-isomorphisms to $[L\to 0]$ with $L\in{}^t\M_0(k)$. Thus the natural functor
    \begin{align*}
        {}^t\M_0(k) &\longrightarrow        {}^t\M_1(k), \\
        L           &\longmapsto [L\to 0]
    \end{align*} 
    is fully faithful.
\end{rmk}

\begin{defn}
    Let $\cA$ be an additive category. Then we denote by $\cA[1/p]$ the category with the same objects as $\cA$ but 
    \[ \Hom_{\cA[1/p]}(X,Y) \colonequals \Hom_\cA(X,Y)\otimes_\bZ\bZ[1/p].\]
\end{defn}

Now, we introduce a special kind of quasi-isomorphisms and use it to reformulate the category of $1$-motives with torsion.
\begin{eg}\label{snqi}
    Let $M=[L\to G]$ be an effective $1$-motive with torsion over $k$ and let $n$ be an integer with $\Char(k)\nmid n$. We consider the multiplication by $n$ on $G$. By \cite[7.3, Lemma 1 and 2]{BLR90Neron}, the group scheme ${}_nG$ is finite \'etale over $k$. Let $L^{(n)}$ be the pullback of $L$ along $n \colon G\to G$. Thus we have the following commutative diagram with exact rows:
    \[
    \xymatrix{
        0 \ar[r] & {}_nG \ar[r]\ar@{=}[d] & L^{(n)} \ar[r]\ar[d]   & L \ar[d] \ar[r]  &0\\
        0 \ar[r] & {}_nG \ar[r]           & G \ar[r]^n             & G \ar[r]         &0.
    }
    \]
    It is clear that $L^{(n)}$ is a constructible group scheme. Thus we obtain a quasi-isomorphism from $[L^{(n)}\to G]$ to $[L\to G]$. We shall denote the effective $1$-motive $[L^{(n)}\to G]$ by $M^{(n)}$ and denote this quasi-isomorphism by $s_n:M^{(n)}\to M$.
\end{eg}

\begin{lem}\label{QuasiisomIsogeny}
    Let $(f,g)\colon [L'\to G']\to [L\to G]$ be a quasi-isomorphism of effective $1$-motives with torsion. Let $n$ be the degree of the isogeny $g\colon G'\to G$. If $\Char(k)\nmid n$, then there exist morphisms $(f',g')\colon [L^{(n)}\to G]\to [L'\to G']$ and $(f^{(n)},g)\colon [L'^{(n)}\to G']\to [L^{(n)}\to G]$ such that 
    \[ s_n=(f,g)\circ(f',g') \text{\quad and \quad } s_n=(f',g')\circ(f^{(n)},g).\]
\end{lem}
\begin{proof}
    By \cite[7.3, Lemma 5]{BLR90Neron}, there exists an isogeny of semi-abelian varieties $g'\colon G\to G'$ such that $g'\circ g=n_{G'}$. Thus $g\circ g'\circ g=g\circ n_{G'}=n_G\circ g$. Since $g$ is an epimorphism, we obtain $g\circ g'=n_G$. By the definition of quasi-isomorphisms, $L'$ is the pullback of $L$ along $g:G'\to G$. Then by the universal property of pullback, there exists a homomorphism $f'\colon L^{(n)}\to L'$ such that the following diagram is commutative:
    \[\xymatrix{
        L^{(n)} \ar@/_2pc/[dd] \ar[r] \ar@{.>}[d]^{f'}
        & G \ar[d]_{g'} \ar@/^2pc/[dd]^{n_G}
        \\
        L'  \ar[r] \ar[d]^f  
        & G'  \ar[d]_g
        \\
        L   \ar[r] 
        & G.
    }\]
    Since the lower square and the composition square are pullbacks, the upper square is also a pullback. Then by the universal property of pullback, there exists a homomorphism $f^{(n)}\colon L'^{(n)}\to L^{(n)}$ such that the following diagram is commutative:
    \[\xymatrix{
        L'^{(n)} \ar@/_2pc/[dd] \ar[r] \ar@{.>}[d]^{f^{(n)}}
        & G' \ar[d]_{g} \ar@/^2pc/[dd]^{n_{G'}}
        \\
        L^{(n)}  \ar[r] \ar[d]^{f'}  
        & G \ar[d]_{g'}
        \\
        L'  \ar[r] 
        & G',
    }\]
    which completes the proof.
\end{proof}

\begin{prop}\label{tM1n}
    Let $M$ and $M'$ be $1$-motives with torsion over $k$. Then 
    \[\varinjlim_{(\widetilde{M}\to M) \text{ q.i.}}\Hom_\eff(\widetilde M,M')\otimes\bZ[1/p]\simeq \varinjlim_{(n,p)=1}\Hom_\eff(M^{(n)},M')\otimes\bZ[1/p].\]
    In other words, ${}^t\M_1(k)[1/p]$ is the localization of ${}^t\M_1^\eff(k)[1/p]$ with respect to the quasi-isomorphisms of the form \(s_n\) for $n$ prime to $p$.
\end{prop}
\begin{proof}
    By Lemma \ref{QuasiisomIsogeny}, the embedding functor from the category of quasi-isomorphisms of the form $s_n\colon M^{(n)}\to M$ to the category of quasi-isomorphisms $\widetilde M\to M$ is cofinal in the sense of \cite[Definition 2.5.1]{KS06CategorySheaf}. Then the assertion follows from \cite[Proposition 2.5.2]{KS06CategorySheaf}.
\end{proof}

In \cite[C.5.3]{BVK16Derived1Motives}, Barbieri-Viale and Kahn showed that ${}^t\M_1(k)[1/p]$ is an abelian category if $k$ is a perfect field. Using Proposition \ref{SAVsubRed}, we can check that their proof also works over arbitrary fields. We can also use the following result to extend their theorem from perfect fields to arbitrary fields.
\begin{lem}\label{tM1PurelyInsep}
    Let $k$ be a field of characteristic $p>0$ and let $K/k$ be a purely inseparable extension. Then the extension of scalars
    \begin{align*}
        \pi^* \colon {}^t\M_1(k)[1/p] &\longrightarrow        {}^t\M_1(K)[1/p], \\
        [L \to G]   &\longmapsto [L_K \to G_K]
    \end{align*} 
    is an equivalence of categories.
\end{lem}
\begin{proof}
    First, we have to explain that this functor is well-defined. In other words, the extension of scalars of effective $1$-motives
    \[ \pi^*_\eff \colon {}^t\M_1^\eff(k) \longrightarrow {}^t\M_1^\eff(K) \]
    sends a quasi-isomorphism to a quasi-isomorphism. Since the extension of scalars of group schemes preserves kernels and cokernels, the functor $\pi^*_\eff$ sends a pullback diagram to a pullback diagram.

    Next, we show that $\pi^*$ is fully faithful. Note that $(L_K)^{(n)}=(L^{(n)})_K$, $(M_K)^{(n)}=(M^{(n)})_K$ and $s_n:M_K^{(n)}\to M_K$ is the base change of $s_n:M^{(n)}\to M$. By Proposition \ref{tM1n}, it suffices to show the isomorphisms
    \[ \Hom_\eff(M^{(n)},M') \simeq \Hom_\eff(M_K^{(n)},M_K') \]
    for any $1$-motives with torsion $M,M'$ over $k$ and for any positive integer $n$ prime to $p$. We can prove it by copying the proof of Theorem \ref{M1fullyfaithful} (purely inseparable extensions are primary).

    Finally, we show that $\pi^*$ is essentially surjective. Let $[L \stackrel{u}{\to} G]$ be a $1$-motive with torsion over $K$. Since $K/k$ is purely inseparable, there exists a constructible group scheme $L_0$ over $k$ such that $(L_0)_K\simeq L$. By \cite[Lemma 3.10]{Brion17AlgGrpIsogenyI}, there exists a semi-abelian variety $G_0$ over $k$ and an epimorphism $f\colon G \to (G_0)_K$ such that $\ker(f)$ is infinitesimal. Using Lemma \ref{e_*GK=G}, we have a $p$-power $q$ such that $qfu$ is the base change of a $k$-morphism $v \colon L_0 \to G_0$. Clearly, in ${}^t\M_1(K)[1/p]$, the $1$-motives with torsion $[L \stackrel{u}{\to} G]$ and $[(L_0)_K \stackrel{v_K}{\to} (G_0)_K]$ are isomorphic to each other.
\end{proof}

\begin{thm}
    The category ${}^t\M_1(k)[1/p]$ is abelian.
\end{thm}
\begin{proof}
    In \cite[C.5.3]{BVK16Derived1Motives}, Barbieri-Viale and Kahn proved it over perfect fields. For a general field $k$, the category ${}^t\M_1(k)[1/p]$ is equivalent to ${}^t\M_1(k^\perf)[1/p]$ by Lemma \ref{tM1PurelyInsep}, where $k^\perf$ is a perfect closure of $k$. Since ${}^t\M_1(k^\perf)[1/p]$ is abelian, so is ${}^t\M_1(k)[1/p]$.
\end{proof}

\begin{prop}[{\cite[C.7]{BVK16Derived1Motives}}]
    \begin{enumerate}[leftmargin=*,label={\rm(\arabic*)}]
        \item The natural functor 
            $$\M_1(k)[1/p]\longrightarrow {}^t\M_1(k)[1/p]$$
            is fully faithful and has a left adjoint $M\mapsto M_\free$.\menum
        \item The natural functor 
            $${}^t\M_1^\tor(k)[1/p]\longrightarrow {}^t\M_1(k)[1/p]$$
            is fully faithful and has a right adjoint $M\mapsto M_\tor$.
    \end{enumerate}
\end{prop}

In the remainder of this subsection, we show that ${}^t\M_1(k)[1/p]$ is Noetherian. This result will not be used later in this paper. We record it here for interested readers and for potential future reference.
\begin{lem}\label{0tM1}
    Let $M=[L\to G]$ be a $1$-motive with torsion. The following are equivalent:
    \begin{enumerate}[label={\rm(\arabic*)}]
        \item $M=0$ in ${}^t\M_1(k)[1/p]$;
        \item $G=0$ in $\Grp(k)$, and $\#L(k_s)$ is a power of $p$.
    \end{enumerate} 
\end{lem}
\begin{proof}
    Note that in an abelian category, $X=0$ if and only if $\id_X=0$. Thus $M=0$ in ${}^t\M_1(k)[1/p]$ if and only if $\id_M=0$, i.e., $p^n\id_M=0$ in ${}^t\M_1(k)$ for some $n\in\bN$. This means that $p^nL=0$ and $p^nG=0$ as group schemes. Equivalently, $G=0$ and $\#L(k_s)$ is a power of $p$.
\end{proof}

\begin{lem}\label{LNoether}
    Let $L$ be a constructible group scheme over $k$. 
    \begin{enumerate}[label={\rm(\arabic*)}]
        \item Every monomorphism to $[L\to 0]$ in ${}^t\M_1(k)[1/p]$ can be represented by a monomorphism $f:L'\to L$ in ${}^t\M_0(k)$.
        \item The $1$-motive $[L\to 0]$ is Noetherian in ${}^t\M_1(k)[1/p]$.
    \end{enumerate}
\end{lem}
\begin{proof}
    \begin{enumerate}[leftmargin=*,label={\rm(\arabic*)}]
        \item Every monomorphism to $[L\to 0]$ in ${}^t\M_1(k)[1/p]$ can be represented by an effective morphism 
        $$(f,0):[L'\to G']\to [L\to 0].$$
        Thus $\ker(f,0)=[\ker(f)\to G']$ is zero in ${}^t\M_1(k)[1/p]$. By Lemma \ref{0tM1}, $G'=0$ and $\#\ker(f)(k_s)$ is a power of $p$. Thus the canonical morphism $[L'\to 0]\to [L'/\ker(f)\to 0]$ is an isomorphism in ${}^t\M_1(k)[1/p]$ and $(f,0):[L'\to 0]\to[L\to 0]$ factors through it as effective morphisms. It means that $(f,0)$ can be represented by the induced morphism 
        $$[L'/\ker(f)\to 0]\to [L\to 0].$$
        \item Let $M_1\hookrightarrow M_2\hookrightarrow \cdots$ be an ascending chain of subobjects of $[L\to 0]$ in ${}^t\M_1(k)[1/p]$. By (1), each morphism $M_n\hookrightarrow [L\to 0]$ can be represented by an effective morphism 
        $$(f_n,0):[L_n\to 0]\hookrightarrow[L\to 0]$$
        with $f_n:L_n\to L$ a monomorphism in ${}^t\M_0(k)$. Since there are no nontrivial quasi-isomorphisms to $[L_n\to 0]$, the monomorphism $[L_n\to 0]\to [L_{n+1}\to 0]$ is of the form 
        $$\left(\frac{i_n}{p^{\alpha_n}},0\right):[L_n\to 0]\to [L_{n+1}\to 0],$$ 
        where $i_n:L_n\to L_{n+1}$ is a morphism in ${}^t\M_0(k)$. Because $f_{n+1}\circ i_n=p^{\alpha_n}f_n$ is a monomorphism in ${}^t\M_0(k)$, the morphism $i_n$ is also a monomorphism in ${}^t\M_0(k)$. Thus 
        $$\rank L_1(k_s)\leq \rank L_2(k_s)\leq \cdots\leq \rank L(k_s).$$
        So there exists a positive integer $N$ such that 
        $$\rank L_n(k_s)=\rank L_N(k_s) \text{\quad and \quad} \#L_n(k_s)/L_N(k_s)<+\infty,\quad \forall~n\geq N.$$
        We get an ascending chain
        $$[L_{N+1}/L_N\to 0]\hookrightarrow [L_{N+2}/L_N\to 0] \hookrightarrow [L_{N+3}/L_N\to 0] \hookrightarrow \cdots $$ 
        of subobjects of $[L/L_N \to 0]$. Thus
        $$\#(L_{N+1}/L_N)(k_s)\leq \#(L_{N+2}/L_N)(k_s)\leq\cdots\leq \#(L/L_N)_\tor(k_s).$$
        It follows that there exists an integer $N'\geq N$ such that 
        $$\#(L_n/L_N)(k_s)=\#(L_{N'}/L_N)(k_s),~\forall~n\geq N'.$$
        It means that the monomorphisms $[L_n/L_N\to 0]\hookrightarrow [L_{n+1}/L_N\to 0]$ are isomorphisms for $n\geq N'$. So the morphisms $[L_n\to 0]\to [L_{n+1}\to 0]$ are isomorphisms for $n\geq N'$. Hence $[L\to 0]$ is Noetherian in ${}^t\M_1(k)[1/p]$.
        \qedhere
    \end{enumerate}
\end{proof}

\begin{lem}\label{GNoether}
    Let $G$ be a semi-abelian variety over $k$.
    \begin{enumerate}[label={\rm(\arabic*)}]
        \item Every monomorphism to $[0\to G]$ in ${}^t\M_1(k)[1/p]$ can be represented by a morphism $g:G'\to G$ of semi-abelian varieties with $\ker(g)$ finite over $k$.
        \item The $1$-motive $[0\to G]$ is Noetherian in ${}^t\M_1(k)[1/p]$.
    \end{enumerate}
\end{lem}
\begin{proof}
    \begin{enumerate}[leftmargin=*,label={\rm(\arabic*)}]
        \item Every monomorphism to $[0\to G]$ in ${}^t\M_1(k)[1/p]$ can be represented by an effective morphism 
        $$(0,g):[L'\to G']\to [0\to G].$$
        Thus $\ker(0,g)$ is zero in ${}^t\M_1(k)[1/p]$. Recall that $\ker(0,g)=[L'_0\to \ker(g)_\red^0]$, where $L_0'$ is the pullback of $L'$ along the closed immersion $\ker(g)_\red^0\to \ker(g)$. By Lemma \ref{0tM1}, $\ker(g)_\red^0=0$ and $\#L_0'(k_s)$ is a power of $p$. Thus the canonical morphism $[L'\to G']\to [L'/L_0'\to G']$ is an isomorphism in ${}^t\M_1(k)[1/p]$ and the morphism $(0,g):[L'\to G']\to [0\to G]$ factors through it as effective morphisms. Since $\ker(g)_\red^0=0$, the group scheme $L_0'$ is the kernel of the induced morphism $u:L'\to \ker(g)$ and the group scheme $\ker(g)$ is finite over $k$. Thus $L'/L_0'=L'/\ker(u)$ is a subgroup of $\ker(g)$, which implies that $L'/L_0'$ is finite \'etale over $k$. So the canonical morphism $[L'/L_0'\to G']\to [0\to G'/(L'/L_0')]$ is a quasi-isomorphism and the morphism $[L'/L_0'\to G']\to [0\to G]$ factors through it. In conclusion, in ${}^t\M_1(k)[1/p]$, the morphism $(0,g)$ can be represented by the induced morphism
        $$[0\to G'/(L'/L_0')]\longrightarrow [0\to G],$$
        which satisfies the desired properties.
        \item Let $M_1\hookrightarrow M_2\hookrightarrow \cdots$ be an ascending chain of subobjects of $[0\to G]$ in ${}^t\M_1(k)[1/p]$. By (1), each morphism $M_n\hookrightarrow [0\to G]$ can be represented by an effective morphism 
        $$(0,g_n):[0\to G_n]\hookrightarrow[0\to G]$$
        with $\ker(g_n)$ finite over $k$. Since all morphisms between Deligne $1$-motives are in fact effective, the monomorphism $[0\to G_n]\to [0\to G_{n+1}]$ is of the form 
        $$\left(0,\frac{i_n}{p^{\alpha_n}}\right):[0\to G_n]\to [0\to G_{n+1}],$$ 
        where $i_n:G_n\to G_{n+1}$ is a morphism in $\SAV(k)$. Because $p^{\alpha_n}$ is an isomorphism in ${}^t\M_1(k)[1/p]$, the morphism $(0,i_n):[0\to G_n]\to [0\to G_{n+1}]$ is also a monomorphism in ${}^t\M_1(k)[1/p]$, which implies that $\ker(i_n)$ is a finite group scheme over $k$. Thus
        $$\dim G_1\leq \dim G_2\leq\cdots\leq \dim G.$$
        So there exists a positive integer $N$ such that 
        $$\dim G_n=\dim G_N,~\forall~n\geq N.$$
        We get an ascending chain of subobjects of $\ker(g_N)$ in $\Grp(k)[1/p]$:
        $$\ker(i_N)\hookrightarrow \ker(i_{N+1}\circ i_N)\hookrightarrow \ker(i_{N+2}\circ i_{N+1}\circ i_N)\hookrightarrow\cdots$$
        Taking Cartier duality of finite (flat) group schemes over $k$, we get a chain of quotients of $\ker(g_N)^\vee$:
        $$\xymatrix{
            && \ker(g_N)^\vee \ar@{->>}[ld]_{q_{N+2}} \ar@{->>}[d]^{q_{N+1}} \ar@{->>}[rd]^{q_N} & \\
            \cdots  \ar@{->>}[r] &\ker(i_{N+2}\circ i_{N+1}\circ i_N)^\vee \ar@{->>}[r] & \ker(i_{N+1}\circ i_N)^\vee \ar@{->>}[r] & \ker(i_N)^\vee.
        }$$
        Since $\ker(g_N)^\vee$ is a Noetherian scheme, the descending chain of its subobjects
        $$\ker(q_N)\supseteq \ker(q_{N+1})\supseteq \ker(q_{N+2})\supseteq\cdots$$
        is stationary. Thus there exists an integer $N'\geq N$ such that
        $$\ker(i_n\circ\cdots \circ i_N)=\ker(i_{n+1}\circ \cdots \circ i_N),~\forall~n\geq N',$$
        It follows that $i_n:G_n\to G_{n+1}$ are isomorphisms in $\Grp(k)[1/p]$ for $n\geq N'$.
        Hence $[0\to G]$ is Noetherian in ${}^t\M_1(k)[1/p]$.
        \qedhere
    \end{enumerate}
\end{proof}

\begin{thm}\label{tM1Noether}
    Let $k$ be a field. Then the category ${}^t\M_1(k)[1/p]$ is Noetherian.
\end{thm}
\begin{proof}
    Let $[L\to G]$ be a $1$-motive with torsion over $k$. Consider the canonical short exact sequence
    $$0\to [0\to G] \to [L\to G] \to [L\to 0] \to0.$$
    By Lemma \ref{LNoether} and \ref{GNoether}, the $1$-motives $[L\to 0]$ and $[0\to G]$ are Noetherian in ${}^t\M_1(k)[1/p]$. Thus the $1$-motive $[L\to G]$ is Noetherian in ${}^t\M_1(k)[1/p]$. Hence the category ${}^t\M_1(k)[1/p]$ is Noetherian.
\end{proof}

\subsection{Perverting \texorpdfstring{$t$}{t}-structures}
Recall that a $t$-structure on a triangulated category $\cD$ is a pair of full subcategories satisfying three simple axioms (\cite[D\'efinition 1.3.1]{BBD82PerverseSheaves}). Let $\cD$ be a triangulated category endowed with a $t$-structure $(\cD^{\leq 0},\cD^{\geq 0})$. We denote by $\cD^\heartsuit$ its heart $\cD^{\leq 0}\cap \cD^{\geq 0}$. For $n\in\bZ$, we denote by $\tau^{\leq n}$ and $\tau^{\geq n}$ the truncation functors with respect to this $t$-structure. We also write $H^nX=\tau^{\leq n}\tau^{\geq n}X[n]$, which is an object of $\cD^\heartsuit$. 

We consider full subcategories of $\cD^\heartsuit$ with the following properties.
\begin{hyp}[{\cite[Hypothesis 2.1]{Ayoub11Motivic-t-structure}}]\label{PervertHypothesis}
    \begin{enumerate}[label={\rm(\arabic*)},leftmargin=*]
        \item $\cA$ is a Serre subcategory of $\cD^\heartsuit$, i.e., stable under subobjects, quotients and extensions. \menum
        \item The inclusion $\cA\hookrightarrow \cD^\heartsuit$ admits a left adjoint $F \colon  \cD^\heartsuit\to \cA$.
        \item If $0\to X'\to X\to X''\to 0$ is an exact sequence in $\cD^\heartsuit$ with $X''\in\cA$, then $F(X')\to F(X)$ is a monomorphism.
    \end{enumerate}
\end{hyp}

\begin{defn}[{\cite[Definition 2.3]{Ayoub11Motivic-t-structure}}]
    An object $X\in\cD^\heartsuit$ is said to be $F$-connected if $F(X)=0$.
\end{defn}
\begin{rmk}\label{ConnecedEtale}
    There is no nontrivial morphisms from $F$-connected objects to objects in $\cA$. In fact, if $X'$ is $F$-connected and $X''$ is in $\cA$, then
    \[\Hom_{\cD^\heartsuit}(X',X'')\simeq \Hom_\cA(FX',X'')=0.\]
    As a result, if $0\to X'\to X\to X''\to 0$ is an exact sequence in $\cD^\heartsuit$ where $X'$ is $F$-connected and $X''\in\cA$, then $X''\simeq FX$.
\end{rmk}

\begin{defn}
    We define two full subcategories of $\cD$ as follows:  
    \begin{itemize}
        \item ${}^t\cD^{\leq 0}$ is the full subcategory consisting of $P\in\cT$ such that $H^i(P)=0$ for $i>1$ and $H^1(P)$ is $F$-connected.
        \item ${}^t\cD^{\geq 0}$ is the full subcategory consisting of $N\in\cT$ such that $H^i(N)=0$ for $i<0$ and $H^0(N)\in\cA$.
    \end{itemize}
\end{defn}

\begin{prop}[{\cite[Proposition~2.4]{Ayoub11Motivic-t-structure}}]
    The pair $({}^t\cD^{\leq 0},{}^t\cD^{\geq 0})$ is a $t$-structure on $\cD$.
\end{prop}

\begin{defn}
    The $t$-structure $({}^t\cD^{\leq 0},{}^t\cD^{\geq 0})$ defined above is called the $\cA$-perverted $t$-structure. We denote by ${}^t\cD^\heartsuit$ its heart ${}^t\cD^{\leq 0}\cap {}^t\cD^{\geq 0}$. For $n\in\bZ$, we denote by ${}^t\tau^{\leq n}$ and ${}^t\tau^{\geq n}$ the truncation functors with respect to this $t$-structure. We also write the cohomology with respect to this $t$-structure as ${}^t\!H^n$.
\end{defn}

\begin{rmk}[{\cite[Remark 2.6]{Ayoub11Motivic-t-structure}}]\label{heart}
    By definition, an object $X$ of $\cD$ is in ${}^t\cD^\heartsuit$ if and only if it satisfies the following properties:  
    \begin{enumerate}[label={\rm(\arabic*)}]
        \item $H^i(X)=0$ for $i\notin \{0,1\}$;
        \item $H^0(X)\in\cA$;
        \item $H^1(X)$ is $F$-connected.
    \end{enumerate}
\end{rmk}

\begin{rmk}
    By definition, if the old $t$-structure on $\cD$ is non-degenerate, i.e.,
    \[\bigcap_{n\in\bZ} \cD^{\leq n}=\bigcap_{n\in\bZ} \cD^{\geq n}=0,\]
    then so is the $\cA$-perverted $t$-structure. 
\end{rmk}

The following result is a generalization of \cite[Proposition~3.11.2]{BVK16Derived1Motives}.
\begin{prop}\label{tstrucExact}
    Keep the above notations and assumptions. Let $X$ be an object of \(\cD\). For any $n\in\bZ$,
    \begin{enumerate}[label={\rm(\arabic*)}]
        \item $H^m({}^t\tau^{\leq n-1}X)=0$ for $m\geq n+1$; and $H^m({}^t\tau^{\geq n}X)=0$ for $m\leq n-1$;
        \item for $m\geq n+1$, \[H^m(X)\simeq H^m({}^t\tau^{\geq n}X);\]
        \item for $m\leq n-1$, \[H^m({}^t\tau^{\leq n-1}X)\simeq H^m(X);\]
        \item we have 
            \[H^0({}^t\!H^nX)\simeq H^n({}^t\tau^{\geq n}X), \quad H^{n}({}^t\tau^{\leq n-1}X)\simeq H^1({}^t\!H^{n-1}X).\]
            and a short exact sequence in $\cD^\heartsuit$
            \[0\to H^1({}^t\!H^{n}X) \to H^{n+1}(X) \to H^0({}^t\!H^{n+1}X) \to 0;\]
        \item $H^0({}^t\!H^{n+1}X)\simeq F(H^{n+1}(X))$.
    \end{enumerate}
\end{prop}
\begin{proof}
    The first assertion is clear by definition.

    Consider the following distinguished triangle in $\cD$
    \[{}^t\tau^{\leq n-1}X  \longrightarrow  X  \longrightarrow {}^t\tau^{\geq n}X \longrightarrow {}^t\tau^{\leq n-1}X[1].\]
    It induces a long exact sequence in $\cD^\heartsuit$
    \[\cdots \to H^m({}^t\tau^{\leq n-1}X) \to H^m(X) \to H^m({}^t\tau^{\geq n}X) \to H^{m+1}({}^t\tau^{\leq n-1}X)\to \cdots.\]
    Combining it with (1), we get (2) and (3).

    Now, consider the following distinguished triangle in $\cD$
    \[{}^t\!H^{n}(X)[-n] \longrightarrow {}^t\tau^{\geq n}X \longrightarrow {}^t\tau^{\geq n+1}X \longrightarrow {}^t\!H^{n}(X)[-n+1].\]
    It induces the following exact sequence in $\cD^\heartsuit$
    \[\cdots \to H^{n-1}({}^t\tau^{\geq n+1}X) \to H^{n}({}^t\!H^{n}(X)[-n]) \to H^{n}({}^t\tau^{\geq n}X) \to H^{n}({}^t\tau^{\geq n+1}X)\]
    \[\to H^{n+1}({}^t\!H^{n}X[-n]) \to H^{n+1}({}^t\tau^{\geq n}X) \to H^{n+1}({}^t\tau^{\geq n+1}X)\to H^{n+2}({}^t\!H^{n}X[-n]) \to \cdots.\]
    By (1), $H^{n-1}({}^t\tau^{\geq n+1}X)$ and $H^{n}({}^t\tau^{\geq n+1}X)$ both vanish. Note also that \[H^{n+2}({}^t\!H^{n}X[-n])=0.\] Thus we get the isomorphism
    \[H^0({}^t\!H^{n}X)\simeq H^{n}({}^t\tau^{\geq n}X)\] 
    and the exact sequence in (4). 

    Using the above argument for the distinguished triangle 
    \[{}^t\tau^{\leq n-2}X\longrightarrow {}^t\tau^{\leq n-1}X\longrightarrow {}^t\!H^{n-1}X[-n+1]\longrightarrow {}^t\tau^{\leq n-2}X[1],\]
    we get $H^{n}({}^t\tau^{\leq n-1}X)\simeq H^1({}^t\!H^{n-1}X)$.

    The last assertion follows from Remark~\ref{ConnecedEtale}.
\end{proof}

\subsection{\texorpdfstring{$1$}{1}-motives}
We consider the unbounded derived category $D(\HI_{\leq 1}(k,\Lambda))$. The standard $t$-structure on it is called the homotopy $t$-structure. The homotopy $t$-structure on $\DM_\et^\eff(k,\Lambda)$ with heart $\HI_\et(k,\Lambda)$ restricts to the above $t$-structure on $D(\HI_{\leq 1}(k,\Lambda))$, which justifies the name.

The following result is proved in \cite{Ayoub11Motivic-t-structure} for $\bQ$-coefficients, and we refine it to integral coefficients here (at least inverting the exponential characteristics).
\begin{prop}
    The subcategory $\HI_{\leq 0}(k,\Lambda)\hookrightarrow \HI_{\leq 1}(k,\Lambda)$ satisfies Hypothesis~\ref{PervertHypothesis}. More precisely,
    \begin{enumerate}[label={\rm(\arabic*)}]
        \item $\HI_{\leq 0}(k,\Lambda)$ is a Serre subcategory of \(\HI_{\leq 1}(k,\Lambda)\);
        \item the inclusion $\HI_{\leq 0}(k,\Lambda)\hookrightarrow \HI_{\leq 1}(k,\Lambda)$ admits a left adjoint 
            \[\pi_0 \colon  \HI_{\leq 1}(k,\Lambda)\to \HI_{\leq 0}(k,\Lambda);\]
        \item if \[\xymatrix{ 0\ar[r] &\sF' \ar[r]^i &\sF \ar[r]^q & \sF'' \ar[r] &0}\] is an exact sequence in $\HI_{\leq 1}(k,\Lambda)$ with $\sF''\in\HI_{\leq 0}(k,\Lambda)$, then $\pi_0(i) \colon  \pi_0(\sF')\to \pi_0(\sF)$ is a monomorphism.
    \end{enumerate}
\end{prop}
\begin{proof}
    The first assertion has been proved in Corollary~\ref{HI1exact}.

    Clearly, the restriction of $\pi_0$ to $\HI_{\leq 1}(k,\Lambda)$ is left adjoint to $\HI_{\leq 0}(k,\Lambda)\hookrightarrow \HI_{\leq 1}(k,\Lambda)$.

    Now, we prove the last assertion. By Lemma~\ref{connected}, the composition $\sF^0\to \sF\to \sF''$ is trivial. Thus the morphism $\sF^0\to \sF$ factors through $\sF'$. Again, the induced morphism $\sF^0 \to \pi_0(\sF')$ is trivial and thus the morphism $\sF^0\to \sF'$ factors through $\sF'^0$.
    \[\xymatrix{
                & \sF'^0 \ar@<0.5ex>[r]^f \ar[d] & \sF^0 \ar[d] \ar@<0.5ex>[l]^g \ar[dl] &&\\
        0\ar[r] & \sF'   \ar[r]   \ar[d]         & \sF   \ar[r] \ar[d]                   & \sF'' \ar[r] & 0\\
                & \pi_0(\sF')     \ar[r]         & \pi_0(\sF)
    }\]
    It is easy to check that $f$ and $g$ are inverses of each other. Applying the snake lemma to the top two rows of the above diagram (with zero being the third term of the first row), we get that $\pi_0(\sF')\to \pi_0(\sF)$ is a monomorphism.
\end{proof}

We apply the construction from the previous subsection.
\begin{defn}
    The $1$-motivic $t$-structure on $D(\HI_{\leq 1}(k,\Lambda))$ is the $\HI_{\leq 0}(k,\Lambda)$-perverted $t$-structure associated with the homotopy $t$-structure. The heart of this $1$-motivic $t$-structure will be denoted by $\MM_1(k,\Lambda)$. We call objects in $\MM_1(k,\Lambda)$ $1$-motives.

    As explained in Remark~\ref{heart}, an object $X$ in $D(\HI_{\leq 1}(k,\Lambda))$ is a $1$-motive if and only if it satisfies the following properties:  
    \begin{enumerate}[label={\rm(\arabic*)}]
        \item $H^i(X)=0$ for $i\notin \{0,1\}$;
        \item $H^0(X)$ is a $0$-motivic sheaf;
        \item $H^1(X)$ is a connected $1$-motivic sheaf.
    \end{enumerate}
    By truncation, in fact, we may and do represent $X$ by a two-term complex concentrated in degrees $0,1$. We write it as $[L\to G]$ with Deligne $1$-motives as main examples. We call it a $0$-motive if it is quasi-isomorphic to $[L\to 0]$ with $L$ a $0$-motivic sheaf. We will call it constructible if $H^0(X)$ and $H^1(X)$ are finitely presented $1$-motivic sheaves in the sense of Definition~\ref{fpHI1}.
\end{defn}

There is a functor
\[ T \colon {}^t\M_1(k)[1/p] \to \MM_1(k) \]
which sends a $1$-motive with torsion $[L \to G]$ to the $1$-motive $[L^\tr_\Lambda \to G^\tr_\Lambda]$.

\begin{prop}
    The functor $T$ induces an exact full embedding of \({}^t\M_1(k)[1/p]\) into $\MM_1(k)$. Moreover, with $\bQ$-coefficients, the category of $1$-motives with torsion is equivalent to the category of constructible $1$-motives.
\end{prop}
\begin{proof}
    This is a refinement of \cite[Proposition 3.11]{Ayoub11Motivic-t-structure} from rational coefficients to $\bZ[1/p]$-coefficients. By d\'evissage, we have to check the isomorphisms between Homs and Extensions of $1$-motives of the form $[L\to 0]$ and $[0\to G]$. Such isomorphisms can be found in \cite[C.8]{BVK16Derived1Motives}. With $\bQ$-coefficients, Ayoub and Barbieri-Viale \cite[Proposition 2.4.10]{ABV09MotShvLAlb} showed that the cohomological dimension of $\HI_{\leq 1}$ is $1$, and then Ayoub \cite[Lemma 3.10]{Ayoub11Motivic-t-structure} used it to show that a $1$-motive $M$ decomposes into a direct sum 
    \[ M \simeq H^0M[0] \oplus H^1M[-1]. \]
    When $M$ is a constructible $1$-motive, this direct sum is a $1$-motive with torsion.
\end{proof}

\subsection{Higher direct images of \texorpdfstring{$1$}{1}-motives}
Let $K/k$ be a field extension. By abuse of notation, we denote the direct image functor of $1$-motivic sheaves as $e_*$, which is the functor $e_{1*}$ in Definition \ref{HIne_*}. For a complex $X\in D(\HI_{\leq 1}(K,\Lambda))$, we denote $Re_*X$ the total derived direct image of $X$, which is a complex of $1$-motivic sheaves over $k$. We denote by $R^ie_*X$ (resp. ${}^m\!R^ie_*X=[L^i\to G^i]$) the $i$-th cohomology of $Re_*X$ relative to the homotopy (resp. $1$-motivic) $t$-structure on $D(\HI_{\leq 1}(k,\Lambda))$.

\begin{thm}\label{mRe*0}
    Let $K/k$ be a field extension and let $L$ be a $0$-motivic sheaf over $K$. Then for all $i\geq 0$,
    \[{}^m\!R^ie_*[L\to 0]=[R^ie_*L\to 0].\]
    In particular, ${}^m\!R^ie_*[L\to 0]$ is a torsion $0$-motive for all $i\geq 1$. Moreover, if $K/k$ is primary, then 
    \[ {}^m\!R^ie_*[L \to 0] =[ H^i(\Gamma,L(K_s)) \to 0] \] 
    where $\Gamma=\Gal(K_s/Kk_s)$ and $H^i(\Gamma,L(K_s))$ is the $0$-motivic sheaf associated with the $\Gal(k_s/k)$-module $H^i(\Gamma,L(K_s))$.
\end{thm}
\begin{proof}
    By Corollary~\ref{Rie*torsion} (1), $R^ie_*L$ is a $0$-motivic sheaf for any $i\geq 0$ and is torsion for any $i\geq 1$. Applying Proposition~\ref{tstrucExact} (5) to $Re_*L$, we get for any $i\geq 0$,
    \[\ker(L^i\to G^i)\simeq \pi_0(R^ie_*L)\simeq R^ie_*L\]
    and
    \[\coker(L^i\to G^i)\simeq (R^{i+1}e_*L)^0=0.\]
    This means that $[L^i\to G^i]$ is quasi-isomorphic to $[R^ie_*L\to 0]$.
\end{proof}

When $K/k$ is primary, we can use results on cohomology of profinite groups to deduce some properties of the higher direct images of $0$-motivic sheaves. 
\begin{lem}\label{H1GalFinite}
    Let $\Gamma$ be a profinite group and $L$ be a discrete $\Gamma$-module which is a finitely generated free abelian group. Then $H^1(\Gamma,L)$ is finite.
\end{lem}
\begin{proof}
    This result is well-known to experts. For example, it is \cite[Exercise~4.10]{Harari20GaloisCFT}. For readers' convenience, we give a proof. By definition of discrete $\Gamma$-modules, for any $x \in L$, the stabilizer $\Gamma_x=\{g\in \Gamma\mid gx = x\}$ is open in $\Gamma$. Since $L$ is finitely generated, the group 
    \[U  \colonequals \left\{g\in\Gamma \mid gx=x,\text{ for all } x\in L\right\}\]
    is an open normal subgroup of $\Gamma$. The Hochschild--Serre spectral sequence \cite[Chapter I, 2.6~b)]{Serre02GaloisCohomology} gives us the exact sequence
    \[0\to H^1(\Gamma/U,L^U) \to H^1(\Gamma,L) \to H^1(U,L)^{\Gamma/U}.\]
    We claim that $H^1(U,L)$ vanishes. Then it suffices to show the finiteness of $H^1(\Gamma/U,L^U)$. It is finitely generated because $L^U$ is finitely generated. Then it is finite as the higher cohomology groups are torsion groups.
    
    Now, we prove the claim. By \cite[Chapter I, 2.2, Corollary~1]{Serre02GaloisCohomology}, we have 
    \[H^1(U,L)\simeq \varinjlim H^1(U/V,L^V),\]
    where $V$ runs over all open normal subgroups of $U$. Thus we reduce to show $H^1(U/V,L^V)=0$. By construction, the group $U$ acts trivially on $L$. So $U/V$ acts trivially on $L^V=L$. It follows that $H^1(U/V,L^V)$ is the group of homomorphisms from $U/V$ to $L$. Since $U/V$ is finite and $L$ is free, this group vanishes.
\end{proof}
\begin{rmk}
    It is necessary to assume that $L$ is free in Lemma~\ref{H1GalFinite}. A counterexample is that $\Gamma=\Gal(k_s/k)$ and $L=\mu_n(k_s)=\{x\in k_s \mid x^n=1\}$ for $n$ prime to $\Char(k)$. Taking the long exact sequence of cohomology associated with the Kummer exact sequence
    \[0 \to \mu_n(k_s) \to k_s^\times \to k_s^\times \to 0,\]
    and applying Hilbert's theorem 90, we obtain $H^1(\Gal(k_s/k),\mu_n(k_s))\simeq k^\times / k^{\times n}$. It is not finite in general.
\end{rmk}

So we have the following result. 
\begin{lem}\label{H1latticeFinite}
    Let $K/k$ be a primary field extension and let $L$ be a lattice over $K$. Then ${}^m\!R^1e_*[L\to 0]$ is a constructible $0$-motive.
\end{lem}
\begin{proof}
    This is a direct consequence of Theorems~\ref{mRe*0} and Lemma~\ref{H1GalFinite}.
\end{proof}

\begin{lem}\label{mRe*HI1^0}
    Let $K/k$ be a field extension and let $G$ be a connected $1$-motivic sheaf over $K$. Then 
    \[{}^m\!R^ie_*[0\to G]=\left\{
        \begin{array}{ll}
            {[0\to (e_*G)^0]}, & \hbox{if \(i=0\);} \\
            {[\pi_0(R^1e_*G)\to 0]}, & \hbox{if \(i=2\);} \\
            {[R^{i-1}e_*G\to 0]}, & \hbox{if \(i\geq 3\)}
        \end{array}\right.\]
    and we have an exact sequence 
    \[0\to \pi_0(e_*G) \to L^1 \to G^1 \to (R^1e_*G)^0 \to 0.\]
    In particular, ${}^m\!R^ie_*[0\to G]$ are $0$-motives for $i\geq 2$ and are torsion for $i\geq 3$.
\end{lem}
\begin{proof}
    By Theorem~\ref{Rie*torsion}, $R^ie_*G$ are torsion $0$-motivic sheaves for all $i\geq 2$. Applying Proposition~\ref{tstrucExact} (5) to $Re_*[0\to G]$, we get for all $i\geq 3$,
    \[\ker(L^i\to G^i)\simeq \pi_0(R^ie_*[0\to G]) \simeq R^{i-1}e_*G\]
    and
    \[\coker(L^{i-1}\to G^{i-1}) \simeq (R^ie_*[0\to G])^0=0.\]
    This means that $[L^i\to G^i]$ is quasi-isomorphic to $[R^{i-1}e_*G\to 0]$ for $i\geq 3$, and $\coker(L^2\to G^2)=0$. Note that 
    \[\ker(L^2\to G^2) \simeq \pi_0(R^2e_*[0\to G])\simeq \pi_0(R^1e_*G).\]
    Thus $[L^2\to G^2]$ is quasi-isomorphic to $[\pi_0(R^1e_*G)\to 0]$. We also have 
    \[\ker(L^0\to G^0) \simeq \pi_0(R^{-1}e_*G)=0,\]
    and \[\coker(L^0\to G^0)\simeq (e_*G)^0.\]
    Thus $[L^0\to G^0]$ is quasi-isomorphic to $[0\to (e_*G)^0]$.
\end{proof}

\begin{thm}\label{mRe*AV}
    Let $K/k$ be a primary field extension and let $A$ be an abelian variety over $K$. Then 
    \[{}^m\!R^ie_*[0\to A]=\left\{
        \begin{array}{ll}
            {[0\to \pi_*A]}, & \hbox{if \(i=0\);} \\
            {[\LN(A,Kk_s/k_s)\to 0]}, & \hbox{if \(i=1\);} \\
            {[R^{i-1}e_*A\to 0]}, & \hbox{if \(i\geq 2\).}
        \end{array}\right.
    \]
    In particular, ${}^m\!R^0e_*[0\to A]$ is a constructible $1$-motive, and ${}^m\!R^ie_*[0\to A]$ are torsion $0$-motives for $i\geq 2$. Moreover, if $K/k$ is a finitely generated regular extension, then ${}^m\!R^1e_*[0\to A]$ is a constructible $0$-motive. 
\end{thm}
\begin{proof}
    By Lemma~\ref{mRe*HI1^0}, it suffices to compute ${}^m\!R^0e_*$ and ${}^m\!R^1e_*$. This involves the information of $e_*A$ and $R^1e_*A$. By Theorem~\ref{ChowRevisited}, 
    \[(e_*A)^0=\pi_*A \text{\quad and \quad} \pi_0(e_*A)=\LN(A,Kk_s/k_s).\]
    Thus ${}^m\!R^0e_*[0\to A]=[0\to \pi_*A]$ and
    \[\ker(L^1\to G^1)\simeq \pi_0(R^1e_*[0\to A])\simeq \pi_0(e_*A)\simeq \LN(A,Kk_s/k_s).\]
    By Theorem~\ref{Rie*A}, $R^1e_*A$ is a torsion $0$-motivic sheaf. So 
    \[\coker(L^1\to G^1)=(R^1e_*A)^0=0.\]
    Hence $[L^1\to G^1]$ is quasi-isomorphic to $[\LN(A,Kk_s/k_s)\to 0]$.
\end{proof}

\begin{thm}\label{mRe*Gm}
    Let $X$ be a smooth projective and geometrically connected variety over $k$ and let $K$ be the function field of $X$. Then we have
    \[{}^m\!R^ie_*[0\to \bG_m]=\left\{
        \begin{array}{ll}
            {[0\to \bG_m]}, & \hbox{if \(i=0\);} \\
            0             , & \hbox{if \(i=2\);} \\
            {[R^{i-1}e_*\bG_m \to 0]}, & \hbox{if \(i\geq 3\),}
        \end{array}\right.
    \]
    and an exact sequence 
    \[0\to (Kk_s)^\times/k_s^\times \to L^1 \to G^1 \to R^1e_*\bG_m \to 0,\]
    where $R^1e_*\bG_m$ is the cokernel of \(\Div^0(X_{k_s}) \to \Pic^0_{X/k}\).
  
    Moreover, with $\bQ$-coefficients, we have 
    \[ {}^m\!R^1e_*[0\to \bG_m] = [\Div^0(X_{k_s}) \to \Pic^0_{X/k}]. \] 
\end{thm}
\begin{proof}
    By Lemma~\ref{mRe*HI1^0}, it suffices to compute ${}^m\!R^ie_*$ for $i=0,1,2$. By Theorem~\ref{R1e*Gm}, $R^1e_*\bG_m$ is a connected $1$-motivic sheaf. Thus 
    \[ {}^m\!R^2e_*[0\to \bG_m]=0, \text{\quad and \quad} \coker(L^1 \to G^1) = R^1e_*\bG_m. \] 
    By Theorem~\ref{ChowRevisited}, 
    \[(e_*\bG_m)^0=\pi_*\bG_m=\bG_m \text{\quad and \quad} \pi_0(e_*\bG_m)=(Kk_s)^\times /k_s^\times.\]
    Thus ${}^m\!R^0e_*[0\to \bG_m]=[0\to \bG_m]$ and
    \[\ker(L^1\to G^1)\simeq \pi_0(R^1e_*[0\to \bG_m])\simeq \pi_0(e_*\bG_m)\simeq (Kk_s)^\times/k_s^\times.\]
    Note that we have the following exact sequence
    \[ 0 \to k_s^\times \to (Kk_s)^\times \to \Div^0(X_{k_s}) \to \Pic^0_{X/k} \to R^1e_*\bG_m \to 0, \]
    where the exactness at the last three terms is given by Theorem~\ref{R1e*Gm}. So the two-term complexes $[L^1 \to G^1]$ and $[\Div^0(X_{k_s}) \to \Pic^0_{X/k}]$ have the same cohomology. If we work with $\bQ$-coefficients, then these two complexes are equal in the derived category of $1$-motivic sheaves because the cohomological dimension of $\HI_{\leq 1}$ is $1$ by \cite[Proposition 2.4.10]{ABV09MotShvLAlb}.
\end{proof}

We can also describe the higher direct images of more general Deligne $1$-motives.
\begin{lem}\label{Rie*1motives}
    Let $K/k$ be a field extension let $M=[L\to G]$ be a Deligne $1$-motive over $K$.
    \begin{enumerate}[label={\rm(\arabic*)}]
        \item $R^ie_*M$ are $0$-motivic sheaves for $i=0$ and $i\geq 3$, and they are torsion for $i\geq 3$.
        \item If $G$ is an abelian variety, then $R^2e_*M$ is also a $0$-motivic sheaves.
        \item If $G$ is an abelian variety and $K/k$ is a finitely generated regular field extension, then $R^1e_*M$ is a finitely presented $1$-motivic sheaf.
    \end{enumerate}     
\end{lem}
\begin{proof}
    The distinguished triangle 
    \[G[-1] \to M \to L \to G\]
    induces the following long exact sequence 
    \[0\to R^0e_*M \to R^0e_*L \to R^0e_*G \to R^1e_*M \to \cdots.\]
    By Theorem~\ref{Rie*torsion}, $R^ie_*(L)$ are torsion $0$-motivic sheaves for $i\geq 1$ and $R^ie_*G$ are torsion $0$-motivic sheaves for $i\geq 2$. By Theorem~\ref{Rie*A}, $R^1e_*G$ is also a torsion $0$-motivic sheaf if $G$ is an abelian variety. By Corollary~\ref{LNe*A} and Lemma~\ref{H1latticeFinite}, if $G$ is an abelian variety and $K/k$ is a finitely generated regular field extension, then $e_*G$ and $R^1e_*L$ are finitely presented $1$-motivic sheaves. Recall from Corollary~\ref{HI1exact} that $\HI_{\leq 0}$ is a Serre subcategory of $\HI_{\leq 1}$. Then the expected result follows from the above long exact sequence.
\end{proof}

\begin{lem}
    Let $K/k$ be a field extension and let $M=[L\to G]$ be a Deligne $1$-motive over $K$. 
    \begin{enumerate}[label={\rm(\arabic*)}]
        \item We have \[{}^m\!R^ie_*[L\to G]=\left\{
            \begin{array}{ll}
                {[\pi_0(R^2e_*M)\to 0]}, & \hbox{if \(i=2\);} \\
                {[R^ie_*M\to 0]}, & \hbox{if \(i\geq 3\).}
            \end{array}\right.\]
        \item If $G=A$ is an abelian variety, then we have 
        \[{}^m\!R^ie_*[L\to A]=\left\{
            \begin{array}{ll}
                {[\pi_0(R^1e_*M)\to 0]}, & \hbox{if \(i=1\);} \\
                {[R^ie_*M\to 0]}, & \hbox{if \(i\geq 2\).}
            \end{array}\right.\]
    \end{enumerate}
\end{lem}
\begin{proof}
    These are direct consequences of Proposition \ref{tstrucExact} and Lemma \ref{Rie*1motives}.
\end{proof}

The following result can be viewed as a generalization of the Lang-N\'eron theorem for Deligne $1$-motives of the form $[L \to A]$ with $A$ an abelian variety.
\begin{thm}\label{LNLA}
    Let $K/k$ be a finitely generated regular field extension and let $M=[L \to A]$ be a Deligne $1$-motive over $K$ where $A$ is an abelian variety. Write $\Gamma=\Gal(K_s/Kk_s)$. Then we have an exact sequence of \(0\)-motivic sheaves 
    \[ 0 \to X \to \pi_0(R^1e_*M) \to Y \to 0,\]
    where 
    \[ X =\coker(L(K_s)^\Gamma \to \LN(A,Kk_s/k_s)) \] and 
    \[ Y =\ker(H^1(\Gamma,L(K_s)) \to R^1e_*A). \]
    In particular, 
    \begin{enumerate}[label={\rm(\arabic*)}]
        \item $\pi_0(R^1e_*M)(k_s)$ is a finitely generated $\Gal(k_s/k)$-module;
        \item ${}^m\!R^1e_*M=[\pi_0(R^1e_*M) \to 0]$ is a constructible $0$-motive.
    \end{enumerate}
\end{thm}
\begin{proof} 
    The distinguished triangle in $D(\HI_{\leq 1}(K,\Lambda))$
    \[ G[-1] \to M \to L \to G \]
    induces the following exact sequence 
    \[ \cdots \to e_*L \to e_*G \to R^1e_*M \to R^1e_*L \to R^1e_*G \to \cdots. \]
    By Theorem \ref{HI0HI1Re*}, $e_*L=L(K_s)^\Gamma$ and $R^1e_*L=H^1(\Gamma,L(K_s))$. Taking $\pi_0$ of the above exact sequence and using Theorem \ref{ChowRevisited}, we obtain the desired exact sequence. By the Lang-N\'eron theorem (resp. Lemma \ref{H1latticeFinite}), we see that $X$ (resp. $Y$) is a $0$-motivic sheaf associated with a finitely generated (resp. finite) $\Gal(k_s/k)$-module. So the same is true for the extension $\pi_0(R^1e_*M)$.
\end{proof}

\appendix
\section{Some results on \'etale cohomology}\label{SmBC}
In this appendix, we prove a smooth base change theorem for non-torsion sheaves and use it to compare the small-\'etale and smooth-\'etale topoi.
\subsection{A smooth base change theorem for non-torsion \'etale sheaves}
Recall the classical smooth base change theorem: 
\begin{thm}[{\cite[Expos\'e~XVI, Corollaire 1.2]{SGA4III}}]\label{SmBCtor}
    Consider the Cartesian diagram of schemes 
    \[\xymatrix{
        Y \ar[r]^g \ar[d]_h & X \ar[d]^f \\
        T \ar[r]^e          & S,
    }\]
    where $f$ is smooth and $e$ is quasi-compact and quasi-separated. Let $\Lambda$ be the localization of \(\bZ\) by inverting the exponential characteristics of all local residue fields of $S$. If $\sF$ is a sheaf of sets (resp. of torsion $\Lambda$-modules) on $T_\et$, then the base change morphism 
    \[\alpha^i_\sF \colon   f^*R^ie_*\sF \longrightarrow R^ig_*h^*\sF\]
    is an isomorphism for $i=0$ (resp. for every $i$). 
\end{thm}

In \cite{Deninger88ProperBaseChangeNontorsion}, Deninger proved a proper base change theorem for non-torsion sheaves. Using the same strategy, we prove the following version of smooth base change theorem.
\begin{thm}\label{SmBCnontor}
    Consider the Cartesian diagram of noetherian schemes 
    \[\xymatrix{
        Y \ar[r]^g \ar[d]_h & X \ar[d]^f \\
        T \ar[r]^e          & S,
    }\]
    and assume that $f$ is smooth and $T$ is excellent. Let $\Lambda$ be the localization of $\bZ$ by inverting the exponential characteristics of all local residue fields of $S$. If $\sF$ is a sheaf of $\Lambda$-modules on $T_\et$, then the base change morphism 
    \[\alpha^i_\sF \colon   f^*R^ie_*\sF \longrightarrow R^ig_*h^*\sF\]
    is an isomorphism for every $i$. 
\end{thm}

The argument in \cite{Deninger88ProperBaseChangeNontorsion} works almost word by word in our case, except that we shall use the smooth base change theorem for torsion sheaves (Theorem~\ref{SmBCtor}) instead of the proper base change theorem, and that we need some additional conditions on the torsion order of the sheaves, i.e., we use sheaves of $\Lambda$-modules rather than abelian sheaves.

For completeness and for readers' convenience, we give a proof of Theorem~\ref{SmBCnontor} in this subsection. A key ingredient is the following result.
\begin{lem}[{\cite[2.2]{Deninger88ProperBaseChangeNontorsion}}]\label{Re*Q}
    Let $T$ be a normal scheme and $e \colon  T\to S$ be a morphism of noetherian schemes. Then $R^ie_*(\bQ)=0$ for $i\geq 1$.
\end{lem}

We use it to prove the smooth base change theorem for the constant sheaves defined by finitely generated $\Lambda$-modules.
\begin{lem}[cf. {\cite[2.3]{Deninger88ProperBaseChangeNontorsion}}]\label{BCfgC}
    Consider the Cartesian diagram of noetherian schemes 
    \[\xymatrix{
        Y \ar[r]^g \ar[d]_h & X \ar[d]^f \\
        T \ar[r]^e          & S,
    }\]
    ans assume that $f$ is smooth and $T$ is normal. Let $\Lambda$ be the localization of $\bZ$ by inverting the exponential characteristics of all local residue fields of $S$, and let $C$ be a finitely generated $\Lambda$-module. Then the base change morphism 
    \[\alpha^i_C \colon   f^*R^ie_*C \longrightarrow R^ig_*h^*C\]
    is an isomorphism for every $i\geq 0$. 
\end{lem}
\begin{proof}
    Note that $e$ is quasi-compact and quasi-separated since $T$ is noetherian. In view of the smooth base change theorem for torsion sheaves (Theorem~\ref{SmBCtor}), we may assume that $C=\Lambda$. Consider the exact sequence
    \[0\to \Lambda \to \bQ \to \bQ/\Lambda \to 0.\]
    By our assumptions, $h$ is smooth and $T$ is normal. Thus $Y$ is also normal. It follows from Lemma~\ref{Re*Q} that
    \[R^ie_*\bQ=0 \text{\quad and \quad}R^ig_*\bQ=0 \text{\quad for \quad}i\geq 1.\]
    Thus the above short exact sequence gives us a commutative diagram with exact rows
    \[\xymatrix{
        0 \ar[r] 
        & f^*e_*\Lambda   \ar[r] \ar[d]_{\alpha^0_\Lambda} 
        & f^*e_*\bQ   \ar[r] \ar[d]_{\alpha^0_\bQ} 
        & f^*e_*(\bQ/\Lambda)   \ar[r] \ar[d]_{\alpha^0_{\bQ/\Lambda}} 
        & f^*R^1e_*\Lambda   \ar[r] \ar[d]_{\alpha^1_\Lambda} 
        & 0\\
        0 \ar[r] 
        & g_*h^*\Lambda \ar[r]                           
        & g_*h^*\bQ \ar[r]                       
        & g_*h^*(\bQ/\Lambda) \ar[r]                                 
        & R^1g_*h^*\Lambda \ar[r]                           
        & 0,
    }\]
    and a commutative diagram for every $i\geq 2$
    \[\xymatrix{
        f^*R^{i-1}e_*(\bQ/\Lambda) \ar[r] \ar[d]_{\alpha^{i-1}_{\bQ/\Lambda}} 
        & f^*R^{i}e_*\Lambda       \ar[d]_{\alpha^{i}_{\Lambda}}  \\
        R^{i-1}g_*h^*(\bQ/\Lambda) \ar[r]
        & R^ig_*(h^*\Lambda).
    }\]
    By the smooth base change theorem for torsion sheaves (Theorem~\ref{SmBCtor}), the $\alpha^i_{\bQ/\Lambda}$ are isomorphisms for all $i$. Hence $\alpha^i_\Lambda$ is an isomorphism for $i\geq 2$. By Theorem~\ref{SmBCtor} again, for every sheaf of sets $\sF$, the base change morphism $f^*e_*\sF \to g_*h^*\sF$ is an isomorphism. Note that the inverse images and the direct images for sheaves of $\Lambda$-modules are compatible with taking the underlying sheaves of sets (\cite[\href{https://stacks.math.columbia.edu/tag/00YV}{Proposition~00YV}]{stacks-project}). Thus $\alpha^0_\Lambda$, $\alpha^0_\bQ$ and $\alpha^0_{\bQ/\Lambda}$ are isomorphisms. Hence $\alpha^1_{\Lambda}$ is also an isomorphism.
\end{proof}

\begin{lem}[cf. {\cite[2.4]{Deninger88ProperBaseChangeNontorsion}}]\label{BCtC}
    Consider the Cartesian diagram of noetherian schemes 
    \[\xymatrix{
        Y \ar[r]^g \ar[d]_h & X \ar[d]^f \\
        T \ar[r]^e          & S,
    }\]
    ans assume that $f$ is smooth. Let $\Lambda$ be the localization of $\bZ$ by inverting the exponential characteristics of all local residue fields of $S$, and let $\sF$ be a sheaf of $\Lambda$-modules on $T_\et$. If $\sF$ is of the form $\tau_*C$ where $\tau \colon  U\to T$ is a finite morphism with $U$ normal and $C$ is a finitely generated $\Lambda$-module, then the base change morphism 
    \[\alpha^i_\sF \colon   f^*R^ie_*\sF \longrightarrow R^ig_*h^*\sF\]
    is an isomorphism for every $i\geq 0$. 
\end{lem}
\begin{proof}
    Consider the following commutative diagram of Cartesian squares
    \[\xymatrix{
        Z \ar[r]^{\tau'} \ar[d]_{h'}  &  Y \ar[r]^g \ar[d]_h & X \ar[d]^f \\
        U \ar[r]^\tau                 &  T \ar[r]^e          & S.
    }\]
    By \cite[Expos\'e~XII, Proposition~4.4(ii)]{SGA4III}, we have a morphism of spectral sequences
    \[\xymatrix{
        E_2^{p,q}=f^*R^pe_*R^q\tau_*C  \ar@{=>}[r] \ar[d] & f^* R^{p+q}(e\tau)_*C \ar[d]^{\alpha_C^{p+q}} \\
        R^pg_*R^q\tau'_*h'^*C         \ar@{=>}[r]         & R^{p+q}(g\tau')_*h'^*C.
    }\]
    Since $\tau$ and $\tau'$ are finite morphisms, the direct images $\tau_*$ and $\tau'_*$ are exact. Thus we have a commutative diagram
    \[\xymatrix{
        f^*R^ie_*(\tau_*C) \ar[r]^-{\simeq} \ar[d]^{\alpha^i_{\tau_*C}} & f^*R^i(e\tau)_*C \ar[dd]^{\alpha_C^i} \\
        R^ig_* h^*(\tau_*C) \ar[d]^{\alpha^0_C} \\
        R^ig_*(\tau'_*h'^*C) \ar[r]^-{\simeq}  & R^i(g\tau')_*h'^*C.
    }\]
    By Lemma~\ref{BCfgC}, the base change morphisms $\alpha_C^0$ and $\alpha^i_C$ in the diagram are isomorphisms. Hence the base change morphism $\alpha_{\tau_*C}^i$ is an isomorphism, i.e., $\alpha_\sF^i$ is an isomorphism.
\end{proof}

Now we prove the main theorem of this subsection.
\begin{proof}[Proof of Theorem~\ref{SmBCnontor}](cf. {\cite[2.5]{Deninger88ProperBaseChangeNontorsion}}).
    By \cite[Expos\'e~IX, Corollaire 2.7.2]{SGA4III}, every sheaf of $\Lambda$-modules $\sF$ on $T_\et$ is a filtered colimit of constructible sheaves of $\Lambda$-modules. Since $e$ and $g$ are quasi-compact and quasi-separated by our assumptions, the higher direct images $R^ie_*$ and $R^ig_*$ commute with filtered colimits (\cite[Expos\'e~VII, Corollaire 5.11]{SGA4II}). So we may assume that $\sF$ is a constructible sheaf of $\Lambda$-modules. Because $T$ is excellent, it is a universally Japanese scheme by \cite[Scholie 7.8.3(vi)]{EGAIV2}. Then according to \cite[Expos\'e~IX, Remarques 2.14.2]{SGA4III}, there exists a monomorphism
    \[\sF \hookrightarrow \bigoplus_{i=1}^n \tau_{i*}C_i,\]
    where $C_i$ is a finitely generated $\Lambda$-module and $\tau_i \colon  U_i\to T$ is a finite morphism with $U_i$ normal. Denote $\sG$ the constructible sheaf of $\Lambda$-modules $\bigoplus_{i=1}^n \tau_{i*}C_i$, and denote $\sH$ the cokernel of above inclusion $\sF\hookrightarrow \sG$. Then $\sH$ is also a constructible sheaf of $\Lambda$-modules. The short exact sequence 
    \[0 \to \sF \to \sG \to \sH \to 0\]
    induces the following commutative diagram with exact rows:
    \[\xymatrix{
        f^*R^ie_*\sG \ar[r] \ar[d]_{\alpha_{\sG}^i}^\simeq 
        & f^*R^ie_*\sH \ar[r] \ar[d]_{\alpha_{\sH}^i} 
        & f^*R^{i+1}e_*\sF \ar[r] \ar[d]_{\alpha_{\sF}^{i+1}} 
        & f^*R^{i+1}e_*\sG \ar[r] \ar[d]_{\alpha_{\sG}^{i+1}}^\simeq 
        & f^*R^{i+1}e_*\sH \ar[d]_{\alpha_{\sH}^{i+1}} 
        \\
        R^ig_*h^*\sG \ar[r]
        & R^ig_*h^*\sH \ar[r]
        & R^{i+1}g_*h^*\sF \ar[r]
        & R^{i+1}g_*h^*\sG \ar[r]
        & R^{i+1}g_*h^*\sH.
    }\]
    We prove by induction on $i$ that $\alpha^i$ is an isomorphism for every constructible sheaf of $\Lambda$-modules. For $i<0$, this is trivial. Assume that the assertion holds for a fixed $i$. Then $\alpha_\sH^i$ is an isomorphism. By Lemma~\ref{BCtC}, the base change morphisms $\alpha^i_\sG$ and $\alpha_\sG^{i+1}$ are isomorphisms. Thus $\alpha_\sF^{i+1}$ is a monomorphism. Since $\sF$ is an arbitrary constructible sheaf of $\Lambda$-modules, $\alpha_\sH^{i+1}$ is a monomorphism as well. It follows from the five lemma that $\alpha_\sF^{i+1}$ is an isomorphism, which completes the proof.
\end{proof}

\subsection{Comparing small and smooth \'etale topoi}
We use the smooth base change theorem for non-torsion sheaves (Theorem~\ref{SmBCnontor}) to compare the small and smooth \'etale topoi. The reader may want to compare this subsection with \cite[\href{https://stacks.math.columbia.edu/tag/0757}{Section~0757}]{stacks-project}.

Let $S$ be a noetherian scheme. Let $\Sm/S$ be the category of smooth separated schemes of finite type over $S$. For $n\in\bN$, we denote by $(\Sm/S)_{\leq n}$ the full subcategory of $\Sm/S$ whose objects are the smooth schemes over $S$ of relative dimension less than or equal to $n$. We  sometimes write $(\Sm/S)_{\leq 0}$ as $\Et/S$\footnote{In fact, we consider here the \'etale schemes separated and of finite type over $S$ rather than all the \'etale schemes over $S$. However, by \cite[Expos\'e~VII, 3.1 and 3.2]{SGA4II}, they give the same category of \'etale sheaves. Because we are mainly interested in \'etale sheaves here, we do not distinguish between these two categories.}.

Let $\cC/S$ be the category $\Sm/S$ or $(\Sm/S)_{\leq n}$. Let $\Lambda$ be the localization of $\bZ$ by inverting the exponential characteristics of all local residue fields of $S$. Denote $\Shv_\et(\cC/S,\Lambda)$ the category of \'etale sheaves of $\Lambda$-modules on $\cC/S$.

For $f \colon  X\to S$ in $\cC$, the natural inclusion $\sigma_f \colon  (\Et/X)_\et \hookrightarrow (\cC/S)_\et$ is a continuous functor, i.e., we have a functor\footnote{The  notations used here are in the same spirit as in \cite[Expos\'e~VII, \S4]{SGA4II}, but are different from the ones in \cite[Expos\'e~III]{SGA4I}. See also \cite[\href{https://stacks.math.columbia.edu/tag/0CMZ}{Section~0CMZ}]{stacks-project} for a comparison of notations.} 
\begin{align*}
    \sigma_{f*} \colon  \Shv_\et(\cC/S,\Lambda) &\longrightarrow        \Shv_\et(\Et/X,\Lambda),\\
                \sF                    &\longmapsto \sF\circ \sigma_f.
\end{align*}
By \cite[Expos\'e~III, Proposition~1.2]{SGA4I}, $\sigma_{f*}$ admits a left adjoint $\sigma_f^*$. We sometimes denote $\sigma_{\id_S}$ by $\sigma_S$ (or $\sigma$ if there is no risk of confusion).

Let $e \colon  T\to S$ be a morphism of noetherian schemes. Then the base change functor $\cC/S \to \cC/T$ induces a continuous functor of \'etale sites\footnote{Warning: For $\cC=\Sm$, this continuous functor does not induce a morphism of sites in general. See \cite[2.2.30]{Olsson16Stacks} or \cite[\href{https://stacks.math.columbia.edu/tag/07BF}{Section~07BF}]{stacks-project} for some examples that $e^*_\Sm$ is not exact.}. Thus we have a pair of adjunctions
\[e_\cC^* \colon  \Shv_\et(\cC/S,\Lambda) \leftrightarrows \Shv_\et(\cC/T,\Lambda) \colon  e^\cC_*,\]
where $e^\cC_*\sF=\sF\circ e$. When $\cC=\Et$, we write $e_\cC^*$ (resp. $e^\cC_*$) as $e^*$ (resp. $e_*$).

The following Cartesian square 
\[\xymatrix{
        Y \ar[r]^g \ar[d]_h & X \ar[d]^f \\
        T \ar[r]^e          & S
}\]
induces a commutative diagram
\[\xymatrix{
    \Et/X \ar[r]^g \ar[d]_{\sigma_f}   & \Et/Y \ar[d]_{\sigma_h} \\
    \cC/S \ar[r]^{e_\cC}               & \cC/T.
}\]
By definition of the direct images, the above diagram induces the following commutative diagram
\[\xymatrix{
    \Shv_\et(\cC/T,\Lambda)   \ar[r]^{e^\cC_*} \ar[d]_{\sigma_{h*}} & \Shv_\et(\cC/S,\Lambda) \ar[d]_{\sigma_{f*}} \\
    \Shv_\et(\Et/Y,\Lambda)   \ar[r]^{g_*}                          & \Shv_\et(\Et/X,\Lambda).
}\]
In particular, if $f=\id_S$, then $\sigma_{S*} e^\cC_* =e_* \sigma_{T*}$ and $\sigma_T^* e^* \simeq e_\cC^* \sigma_S^*$.

\begin{lem}\label{topoi}
    \begin{enumerate}[leftmargin=*,label={\rm(\arabic*)}]
        \item The functors $\sigma_{f*}$ and $\sigma_f^*$ are exact. \menum
        \item The functor $\sigma_f^*$ is fully faithful.
        \item $\sigma_{f*}\sigma^*_S\simeq f^*$.
    \end{enumerate}
\end{lem}
\begin{proof}
    This result is well-known. In fact, the inclusion $\sigma_f \colon  (\Et/X)_\et \hookrightarrow \cC_\et$ is not only continuous but also co-continuous in the sense of \cite[Expos\'e~III, D\'efinition 2.1]{SGA4I}. It follows from \cite[\href{https://stacks.math.columbia.edu/tag/04BH}{Lemma~04BH}]{stacks-project} and \cite[\href{https://stacks.math.columbia.edu/tag/077I}{Lemma~077I}]{stacks-project} that $\sigma_{f*}$ and $\sigma_f^*$ are exact and that $\sigma_f^*$ is fully faithful ($\sigma_{f*}$ is the $g^{-1}$, and $\sigma_f^*$ is the $g_!$ in loc. cit.). 

    Now, we prove the last assertion. Note that $\sigma_f$ can be factored as
    \[\Et/X \stackrel{\sigma_{X}}{\longrightarrow} \cC/X \stackrel{\iota}{\longrightarrow} \cC/S.\]
    Thus $\sigma_{f*}=\sigma_{X*}\iota_*$, where $\iota_*\sF(U/X)=\sF(U/S)$. Since $f \colon  X\to S$ in a morphism in $\cC/S$, the functor $\iota_*$ is in fact $f_\cC^*$. It follows that 
    \[\sigma_{f*}\sigma_S^*  = \sigma_{X*} f_\cC^* \sigma_S^* \simeq \sigma_{X*} \sigma_X^* f^* \simeq f^*,\]
    where the last isomorphism holds by (2).
\end{proof}

Now, we study the derived functors. First, we derive the diagram before Lemma~\ref{topoi}. By Lemma~\ref{topoi} (1), the functors $\sigma_{f*}$ and $\sigma_{h*}$ are exact.
\begin{lem}\label{RestTopoi}
    For $K\in D(\Shv_\et(\cC/T,\Lambda))$, we have a canonical isomorphism
    \[\sigma_{f*}Re^\cC_*K\simeq Rg_*\sigma_{h*}K.\]
\end{lem}
\begin{proof}
    By Lemma~\ref{topoi}, the functor $\sigma_{h*}$ admits an exact left adjoint $\sigma_h^*$. Thus by some formal reason, the functor $\sigma_{h*}$ preserves $K$-injective complexes in the sense of \cite{Spaltenstein88Unbounded}. Then this lemma follows because $K$-injective resolutions compute unbounded right derived functors.
\end{proof}

\begin{thm}\label{SmEtSite}
    Let $e \colon  T\to S$ be a morphism of noetherian schemes with $T$ excellent. Then for $\sF\in\Shv_\et(\Et/T,\Lambda)$, the base change morphism
    \[\alpha_\sF \colon   \sigma_S^*Re_{*}\sF \longrightarrow Re^\cC_*\sigma_T^*\sF\]
    is an isomorphism in $D(\Shv_\et(\cC/S,\Lambda))$.
\end{thm}
\begin{proof}
    Let $\Lambda(X)$ the \'etale sheaf associated with the presheaf mapping $U\in\cC/S$ to the free $\Lambda$-module generated by $\Mor_S(U,X)$.
    Consider the following commutative diagram
    \begin{center}
        \begin{tikzcd}[scale cd=0.95]
            \Hom_{D(\Shv_\et(\cC/S,\Lambda))}(\Lambda(X)[n],\sigma_S^*Re_{*}\sF) \arrow{r}{\alpha_1} \arrow{d}{\beta_1}[swap]{\simeq} 
            & \Hom_{D(\Shv_\et(\cC/S,\Lambda))}(\Lambda(X)[n],Re^\cC_*\sigma_T^*\sF) \arrow{d}{\beta_2}[swap]{\simeq} \\
            \Hom_{D(\Shv_\et(\Et/X,\Lambda))}(\Lambda(X)[n],\sigma_{f*}\sigma_S^*Re_{*}\sF) \arrow{r}{\alpha_2}
            & \Hom_{D(\Shv_\et(\Et/X,\Lambda))}(\Lambda(X)[n],\sigma_{f*}Re^\cC_*\sigma_T^*\sF).
        \end{tikzcd}
    \end{center}
    Here, $\alpha_1$ and $\alpha_2$ are induced by the base change morphisms; $\beta_1$ and $\beta_2$ are adjunction isomorphisms. Consider the pullback of the base change morphism
    \begin{align*}
        f^*Re_{*}\sF &\stackrel{\sim}{\to}    \sigma_{f*}\sigma_S^*Re_{*}\sF\\
                    &\to                     \sigma_{f*}Re^\cC_*\sigma_T^*\sF\\
                    &\stackrel{\sim}{\to}    Rg_*\sigma_{h_*}\sigma_T^*\sF \\
                    &\stackrel{\sim}{\to}    Rg_*h^*\sF,
    \end{align*}
    where the first and the last arrows are isomorphisms by Lemma~\ref{topoi} (3) and the third arrow is an isomorphism by Lemma~\ref{RestTopoi}. According to the smooth base change theorem (Theorem~\ref{SmBCnontor}), $f^*Re_{*}\sF \to Rg_*h^*\sF$ is an isomorphism, which implies that $\alpha_2$ and then $\alpha_1$ are isomorphisms. Since $\{\Lambda(X)[n]\}$ is a system of generators in $D(\Shv_\et(\cC/S,\Lambda))$, we obtain the expected isomorphism in the derived category.
\end{proof}
\begin{rmk}
    Using a spectral sequence argument like \cite[\href{https://stacks.math.columbia.edu/tag/0F09}{Lemma~0F09}]{stacks-project}, we can also establish Theorem~\ref{SmBCnontor} and Theorem~\ref{SmEtSite} for bounded below complexes of sheaves of $\Lambda$-modules on $T_\et$.
\end{rmk}

Note that the spectrum of a field is an excellent scheme. So we have the following result:
\begin{cor}\label{SmBCfield}
    Let $k$ be a field of exponential characteristic $p$, and let $\Lambda=\bZ[\frac1p]$. Let $K/k$ be a field extension. Write $e \colon  \Spec K\to \Spec k$ the induced morphism. Let $\sF$ be a sheaf of $\Lambda$-modules on $(\Spec K)_\et$.
    \begin{enumerate}[label={\rm(\arabic*)}]
        \item For a Cartesian diagram 
        \[\xymatrix{
            X_K \ar[r]^g \ar[d]_h & X \ar[d]^f \\
            \Spec K \ar[r]^e          & \Spec k
        }\]
        with $f$ smooth and of finite type, the base change morphism 
        \[\alpha^i_\sF \colon   f^*R^ie_*\sF \longrightarrow R^ig_*h^*\sF\]
        is an isomorphism for every $i$. 
        \item The base change morphism
        \[\alpha_\sF \colon   \sigma^*_kRe_{*}\sF \longrightarrow Re^\cC_*\sigma^*_K\sF\]
        is an isomorphism in $D(\Shv_\et(\cC/k,\Lambda))$, where $\cC$ is $\Sm$ or $\Sm_{\leq n}$ for some $n\in\bN$.
    \end{enumerate}
\end{cor}

\section{A representability result}
\smallskip
\begin{center}by \textsc{Bruno Kahn}\end{center}
\medskip 

For a field $k$, let $\Sm_l(k)$ be the category of smooth separated $k$-schemes locally of finite type, and $\Sm(k)$ the full subcategory of those which are of finite type. We provide them with the étale topology. and write $\Shv_l(k)$ and $\Shv(k)$ for the corresponding categories of sheaves of sets.

\begin{lem}\label{l1r} 
    The restriction functor $\Shv_l(k)\to \Shv(k)$ is an isomorphism of categories.
\end{lem}
\begin{proof} 
    The inverse functor sends a sheaf $\sF$ to $U\mapsto \prod_{i\in I} \sF(U_i)$, where the $U_i$ are the connected components of $U$.
\end{proof}

\begin{lem}\label{l2r} 
    Let $\sF\in \Shv(k)$. If $\sF\ne \emptyset$, then $\sF(\Spec E)\ne \emptyset$ for some finite separable extension $E/k$.
\end{lem}
\begin{proof} 
    The assumption means that there exists $U\in \Sm(k)$ such that $\sF(U)\ne \emptyset$. But $U$ has a closed point $u$ with separable residue field $E$ (this follows from the characterisation of smoothness in \cite[II, Def. 1.1]{SGA1}), hence $\sF(\Spec E)=\sF(u)\ne \emptyset$.
\end{proof}

\begin{lem}\label{l3r}
    Let $E$ be a finite separable extension of $k$, and let $\underline{L}\in \Shv(k)$ be the étale sheaf represented by $L=\Spec E$. Then there is an isomorphism of categories
    \[\Shv(E)\simeq \Shv(k)/\underline{L}.\]
This isomorphism transports a representable sheaf $\underline{F}\in \Shv(E)$ to $\underline{F}\to \underline{L}\in \Shv(k)/\underline{L}$.
\end{lem}
\begin{proof} 
    Let $\sF\xrightarrow{p} \underline{L}\in \Shv(k)/\underline{L}$. For $U\in \Sm(k)$ and $\pi\in \underline{L}(U)=\Mor_k(U,L)$, let $\sF_\pi(U)=p^{-1}(\pi)$ so that $\sF(U)=\coprod_{\pi\in \underline{L}(U)} \sF_\pi(U)$. The isomorphism of categories is now clear: writing $U\mapsto U_{(k)}$ for the forgetful functor $\Sm(E)\to \Sm(k)$,
    \begin{description}
    \item[In one direction] Let $\sF\xrightarrow{p} \underline{L}\in \Shv(k)/\underline{L}$. For $U\in \Sm(E)$, let $\sF'(U)= \sF_{\pi_U}(U_{(k)})$
    where $\pi_U:U\to \Spec E$ is the structural morphism. 
    \item[In the other direction] Let $\sF'\in \Shv(E)$. For $U\in \Sm(k)$, let $\sF(U)=\coprod_{\pi\in \underline{L}(U)} \sF'(U,\pi)$, and let $p(U):\sF(U)\to \underline{L}(U)$ be the obvious projection.
    \end{description}

    If $F\in \Sm(E)$ and $U\in \Sm(k)$, a $k$-morphism $f:U\to F$ induces a unique $E$-structure on $U$ through which $f$ factors; hence the claim for representable sheaves.
\end{proof}

\begin{prop}\label{extRep}
    Let $\sF\in \Shv(\Sm(k))$ be a sheaf of groups. Suppose that there is an exact sequence
    \[1\to \underline{G}\to \sF\xrightarrow{p} \sL\to 1\]
    where $\underline{G}$ is representable by a smooth connected algebraic $k$-group $G$ and $\sL$ is locally constant. Then $\sF$ is representable by a $k$-group scheme in $\Sm_l(k)$, whose identity component is $G$.
\end{prop}
(The assertion makes sense thanks to Lemma~\ref{l1r}.)
\begin{proof} 
    It suffices to show that $\sF$ is representable as a sheaf of sets, the group structure taking care of itself by Yoneda's lemma as well as the claim on the identity component. Since $\sL$ is locally constant, it is representable by an étale $k$-scheme locally of finite type $L=\coprod_{i\in I} L_i$ where $L_i=\Spec E_i$ for a finite separable extension $E_i/k$ \cite[VIII, p. 54, Rem. 1.12]{Milne80Etale}. Then $\sF=\coprod_{i\in I} \sF_i$ where $\sF_i=p^{-1}(\underline{L}_i)$ with $\underline{L}_i$ the sheaf represented by $L_i$; since a coproduct of representable sheaves in $\Shv_l(k)$ is representable, it suffices to show that each $\sF_i$ is representable. 

    Recall (e.g. \cite[Expos\'e~VII, \S 1]{SGA7I}) that the action of $\underline{G}$ on $\sF$ by left translations defines a $\underline{G}$-torsor over $\sL$, whence a $\underline{G}$-torsor structure on $\sF_i=\sF\times_\sL \underline{L}_i$ over $\underline{L}_i$. By transport of structure, this makes the sheaf $\sF'_i\in \Shv(E_i)$ associated to $\sF_i$ by Lemma~\ref{l3r} a $\underline{G}_E$-torsor over the point.

    By Lemma~\ref{l2r}, $\sF'_i$ is trivial over a finite separable extension of $E$, therefore it  is representable by descent \cite[VIII, Cor. 7.6]{SGA1} since $G$ is quasi-projective \cite{Chow57}. If $F_i$ is the corresponding $E_i$-scheme, $\sF_i$ is then represented by $(F_i)_{(k)}$ by applying Lemma~\ref{l3r} again.
\end{proof}


\begin{thebibliography}{SGA 4$\rm_{III}$}

    \bibitem[ABV09]{ABV09MotShvLAlb}
    Joseph Ayoub and Luca Barbieri-Viale.
    \newblock 1-motivic sheaves and the {A}lbanese functor.
    \newblock {\em Journal of Pure and Applied Algebra}, 213(5):809--839, 2009.
    
    \bibitem[AHPL16]{AHPL16DecompMotiveGrp}
    Giuseppe Ancona, Annette Huber, and Simon Pepin~Lehalleur.
    \newblock On the relative motive of a commutative group scheme.
    \newblock {\em Algebraic Geometry}, 3(2):150--178, 2016.
    
    \bibitem[{Ayo}07a]{Ayoub07SixOperationI}
    Joseph {Ayoub}.
    \newblock {\em {Les six op\'erations de Grothendieck et le formalisme des cycles \'evanescents dans le monde motivique. I}}, volume 314 of {\em {Ast\'erisque}}.
    \newblock Soci\'et\'e Math\'ematique de France, Paris, 2007.
    
    \bibitem[{Ayo}07b]{Ayoub07SixOperationII}
    Joseph {Ayoub}.
    \newblock {\em {Les six op\'erations de Grothendieck et le formalisme des cycles \'evanescents dans le monde motivique. II}}, volume 315 of {\em {Ast\'erisque}}.
    \newblock Soci\'et\'e Math\'ematique de France, Paris, 2007.
    
    \bibitem[Ayo11]{Ayoub11Motivic-t-structure}
    Joseph Ayoub.
    \newblock The {$n$}-motivic {$t$}-structures for {$n=0$}, {$1$} and {$2$}.
    \newblock {\em Advances in Mathematics}, 226(1):111--138, 2011.
    
    \bibitem[Ayo14]{Ayoub14EtaleRealization}
    Joseph Ayoub.
    \newblock La r\'{e}alisation \'{e}tale et les op\'{e}rations de {G}rothendieck.
    \newblock {\em Annales Scientifiques de l'\'{E}cole Normale Sup\'{e}rieure. Quatri\`eme S\'{e}rie}, 47(1):1--145, 2014.
    
    \bibitem[BBD82]{BBD82PerverseSheaves}
    A.~A. Be\u{\i}linson, J.~Bernstein, and P.~Deligne.
    \newblock Faisceaux pervers.
    \newblock In {\em Analysis and topology on singular spaces, {I} ({L}uminy, 1981)}, volume 100 of {\em Ast\'{e}risque}, pages 5--171. Soc. Math. France, Paris, 1982.
    
    \bibitem[BLR90]{BLR90Neron}
    Siegfried {Bosch}, Werner {L\"utkebohmert}, and Michel {Raynaud}.
    \newblock {\em {N\'eron models}}, volume~21 of {\em {Ergebnisse der Mathematik und ihrer Grenzgebiete. 3. Folge}}.
    \newblock Berlin etc.: Springer-Verlag, 1990.
    
    \bibitem[Bri17]{Brion17AlgGrpIsogenyI}
    Michel Brion.
    \newblock Commutative algebraic groups up to isogeny.
    \newblock {\em Documenta Mathematica}, 22:679--725, 2017.
    
    \bibitem[BVK16]{BVK16Derived1Motives}
    Luca Barbieri-Viale and Bruno {Kahn}.
    \newblock {\em {On the derived category of 1-motives}}, volume 381 of {\em {Ast\'erisque}}.
    \newblock Soci\'et\'e Math\'ematique de France (SMF), Paris, 2016.
    
    \bibitem[BVRS03]{BRS03DeligneConj-1-motive}
    L.~Barbieri-Viale, A.~Rosenschon, and M.~Saito.
    \newblock {Deligne's conjecture on 1-motives}.
    \newblock {\em Annals of Mathematics. Second Series}, 158(2):593--633, 2003.
    
    \bibitem[BVS01]{BS01AlbPic-1-motive}
    Luca Barbieri-Viale and Vasudevan {Srinivas}.
    \newblock {Albanese and Picard \(1\)-motives.}
    \newblock {\em {M\'emoires de la Soci\'et\'e Math\'ematique de France. Nouvelle S\'erie}}, 87:vi + 104, 2001.
    
    \bibitem[CD09]{CD09ModelCategories}
    Denis-Charles Cisinski and Fr\'{e}d\'{e}ric D\'{e}glise.
    \newblock Local and stable homological algebra in {G}rothendieck abelian categories.
    \newblock {\em Homology, Homotopy and Applications}, 11(1):219--260, 2009.
    
    \bibitem[CD16]{CD16EtaleMotive}
    Denis-Charles Cisinski and Fr\'{e}d\'{e}ric D\'{e}glise.
    \newblock \'{E}tale motives.
    \newblock {\em Compositio Mathematica}, 152(3):556--666, 2016.
    
    \bibitem[CD19]{CD19MixedMotive}
    Denis-Charles Cisinski and Fr\'{e}d\'{e}ric D\'{e}glise.
    \newblock {\em Triangulated categories of mixed motives}.
    \newblock Springer Monographs in Mathematics. Springer, Cham, 2019.
    
    \bibitem[{Cho}55]{Chow55AVfunctionfield}
    Wei-Liang {Chow}.
    \newblock {Abelian varieties over function fields}.
    \newblock {\em {Transactions of the American Mathematical Society}}, 78:253--275, 1955.
    
    \bibitem[Cho57]{Chow57}
    Wei-Liang Chow.
    \newblock On the projective embedding of homogeneous varieties.
    \newblock {\em Princeton Mathematical Series}, 12:122--128, 1957.
    
    \bibitem[Con02]{Conrad02ChevalleyThmAlgGrp}
    Brian Conrad.
    \newblock A modern proof of {Chevalley}'s theorem on algebraic groups.
    \newblock {\em Journal of the Ramanujan Mathematical Society}, 17(1):1--18, 2002.
    
    \bibitem[{Con}06]{Conrad06ChowLN}
    Brian {Conrad}.
    \newblock {Chow's \(K/k\)-image and \(K/k\)-trace, and the Lang-N\'eron theorem}.
    \newblock {\em {Enseign. Math. (2)}}, 52(1-2):37--108, 2006.
    
    \bibitem[{Del}74]{Deligne74HodgeIII}
    Pierre {Deligne}.
    \newblock {Th\'eorie de Hodge. III}.
    \newblock {\em {Publications Math\'ematiques de l'IH\'ES}}, 44:5--77, 1974.
    
    \bibitem[Den88]{Deninger88ProperBaseChangeNontorsion}
    Ch. Deninger.
    \newblock A proper base change theorem for nontorsion sheaves in \'{e}tale cohomology.
    \newblock {\em Journal of Pure and Applied Algebra}, 50(3):231--235, 1988.
    
    \bibitem[DG70]{DG70AlgGrp}
    Michel Demazure and Pierre Gabriel.
    \newblock {\em Groupes alg\'{e}briques. {T}ome {I}: {G}\'{e}om\'{e}trie alg\'{e}brique, g\'{e}n\'{e}ralit\'{e}s, groupes commutatifs}.
    \newblock Masson \& Cie, \'{E}diteurs, Paris; North-Holland Publishing Co., Amsterdam, 1970.
    \newblock Avec un appendice {{\i}t Corps de classes local} par Michiel Hazewinkel.
    
    \bibitem[EGA IV$_2$]{EGAIV2}
    Alexander {Grothendieck}.
    \newblock {\'El\'ements de g\'eom\'etrie alg\'ebrique : IV. \'Etude locale des sch\'emas et des morphismes de sch\'emas, Seconde partie}.
    \newblock {\em Publications Math\'ematiques de l'IH\'ES}, 24:5--231, 1965.
    
    \bibitem[EGA IV$_3$]{EGAIV3}
    Alexander {Grothendieck}.
    \newblock {\'El\'ements de g\'eom\'etrie alg\'ebrique : {IV.} {\'Etude} locale des sch\'emas et des morphismes de sch\'emas, {Troisi\`eme} partie}.
    \newblock {\em Publications Math\'ematiques de l'IH\'ES}, 28:5--255, 1966.
    
    \bibitem[{Ful}98]{Fulton98Intersection}
    William {Fulton}.
    \newblock {\em {Intersection theory}}, volume~2 of {\em {Ergebnisse der Mathematik und ihrer Grenzgebiete. 3. Folge}}.
    \newblock Springer, Berlin, 2nd edition, 1998.
    
    \bibitem[Har20]{Harari20GaloisCFT}
    David Harari.
    \newblock {\em Galois cohomology and class field theory}.
    \newblock Universitext. Springer, Cham, 2020.
    \newblock Translated from the 2017 French original by Andrei Yafaev.
    
    \bibitem[Hir03]{Hirschhorn03ModelCategories}
    Philip~S. Hirschhorn.
    \newblock {\em Model categories and their localizations}, volume~99 of {\em Mathematical Surveys and Monographs}.
    \newblock American Mathematical Society, Providence, RI, 2003.
    
    \bibitem[{Kah}06]{Kahn06ClassGroup}
    Bruno {Kahn}.
    \newblock {Sur le groupe des classes d'un sch\'ema arithm\'etique. (Avec un appendice de Marc Hindry)}.
    \newblock {\em {Bulletin de la Soci\'et\'e Math\'ematique de France}}, 134(3):395--415, 2006.
    
    \bibitem[{Kah}18]{Kahn18MotifsAdjoint}
    Bruno {Kahn}.
    \newblock {Motifs et adjoints}.
    \newblock {\em {Rendiconti del Seminario Matematico della Universit\`a di Padova}}, 139:77--128, 2018.
    
    \bibitem[KS06]{KS06CategorySheaf}
    Masaki {Kashiwara} and Pierre {Schapira}.
    \newblock {\em {Categories and Sheaves}}, volume 332 of {\em {Grundlehren der mathematischen Wissenschaften, A Series of Comprehensive Studies in Mathematics}}.
    \newblock Springer-Verlag, Berlin Heidelberg, 2006.
    
    \bibitem[LN59]{LangNeron59LNfg}
    Serge Lang and Andr\'e N\'{e}ron.
    \newblock Rational points of abelian varieties over function fields.
    \newblock {\em American Journal of Mathematics}, 81:95--118, 1959.
    
    \bibitem[{Mil}80]{Milne80Etale}
    James~S. {Milne}.
    \newblock {\em {\'Etale cohomology}}, volume~33.
    \newblock Princeton University Press, Princeton, NJ, 1980.
    
    \bibitem[{Mil}86]{Milne86AV}
    James~S. {Milne}.
    \newblock Abelian varieties.
    \newblock In Gary {Cornell} and Joseph~H. {Silverman}, editors, {\em Arithmetic geometry}, pages 103--150. Springer, New York, 1986.
    
    \bibitem[{Mum}14]{Mumford14AV}
    David {Mumford}.
    \newblock {\em {Abelian Varieties. With Appendices by C. P. Ramanujam and Yuri Manin}}.
    \newblock Hindustan Book Agency (India), New Delhi, 2014.
    
    \bibitem[MV99]{MV99A1homotopy}
    Fabien Morel and Vladimir Voevodsky.
    \newblock {$\mathbb{A}^1$}-homotopy theory of schemes.
    \newblock {\em {Publications Math\'ematiques de l'IH\'ES}}, pages 45--143 (2001), 1999.
    
    \bibitem[MVW06]{MVW06Motive}
    Carlo {Mazza}, Vladimir {Voevodsky}, and Charles {Weibel}.
    \newblock {\em {Lecture notes on motivic cohomology}}, volume~2 of {\em {Clay Mathematics Monographs}}.
    \newblock Providence, RI: American Mathematical Society (AMS); Cambridge, MA: Clay Mathematics Institute, 2006.
    
    \bibitem[Ols16]{Olsson16Stacks}
    Martin Olsson.
    \newblock {\em Algebraic spaces and stacks}, volume~62 of {\em Colloquium Publications. American Mathematical Society}.
    \newblock Providence, RI: American Mathematical Society (AMS), 2016.
    
    \bibitem[Org04]{Orgogozo041Motives}
    Fabrice Orgogozo.
    \newblock {Isomotifs de dimension inf\'erieure ou \'egale \`a un}.
    \newblock {\em Manuscripta Mathematica}, 115(3):339--360, 2004.
    
    \bibitem[PL19]{PL19DA1}
    Simon Pepin~Lehalleur.
    \newblock Triangulated categories of relative 1-motives.
    \newblock {\em Advances in Mathematics}, 347:473--596, 2019.
    
    \bibitem[Ram01]{Ramachandran01AlbPic1Motive}
    Niranjan Ramachandran.
    \newblock Duality of {A}lbanese and {P}icard 1-motives.
    \newblock {\em $K$-Theory}, 22(3):271--301, 2001.
    
    \bibitem[Ray70]{Raynaud70AmpleShv}
    Michel Raynaud.
    \newblock {\em Faisceaux amples sur les sch{\'e}mas en groupes et les espaces homog{\`e}nes}, volume 119 of {\em Lecture Notes in Mathematics}.
    \newblock Springer, Cham, 1970.
    
    \bibitem[Ser02]{Serre02GaloisCohomology}
    Jean-Pierre Serre.
    \newblock {\em Galois cohomology}.
    \newblock Springer Monographs in Mathematics. Springer-Verlag, Berlin, english edition, 2002.
    \newblock Translated from the French by Patrick Ion and revised by the author.
    
    \bibitem[SGA 1]{SGA1}
    A.~{Grothendieck}, editor.
    \newblock {\em {S\'eminaire de g\'eom\'etrie alg\'ebrique du Bois Marie 1960-61. Rev\^etements \'etales et groupe fondamental (SGA 1).}}, volume~3 of {\em {Documents Math\'ematiques}}.
    \newblock Paris: Soci\'et\'e Math\'ematique de France, 2003.
    
    \bibitem[SGA 3$\rm _I$]{SGA3I}
    Michel {Demazure} and Alexander {Grothendieck}, editors.
    \newblock {\em {S\'eminaire de g\'eom\'etrie alg\'ebrique du Bois Marie 1962-64. Sch\'emas en groupes (SGA 3). Tome I: Propri\'et\'es g\'en\'erales des sch\'emas en groupes.}}, volume~7 of {\em {Documents Math\'ematiques}}.
    \newblock Paris: Soci\'et\'e Math\'ematique de France, 2011.
    
    \bibitem[SGA 4$\rm _I$]{SGA4I}
    M.~{Artin}, A.~{Grothendieck}, and J.~L. {Verdier}, editors.
    \newblock {\em {S\'{e}minaire de G\'{e}om\'{e}trie Alg\'{e}brique du Bois-Marie 1963-1964. Th\'{e}orie des topos et cohomologie \'{e}tale des sch\'{e}mas (SGA 4). {T}ome 1: {T}h\'{e}orie des topos.}}, volume 269 of {\em Lecture Notes in Mathematics}.
    \newblock Springer-Verlag, Berlin-New York, 1972.
    
    \bibitem[SGA 4$\rm_{II}$]{SGA4II}
    M.~{Artin}, A.~{Grothendieck}, and J.~L. {Verdier}, editors.
    \newblock {\em {S\'{e}minaire de G\'{e}om\'{e}trie Alg\'{e}brique du Bois-Marie 1963-1964. Th\'{e}orie des topos et cohomologie \'{e}tale des sch\'{e}mas (SGA 4). {T}ome 2.}}, volume 270 of {\em Lecture Notes in Mathematics}.
    \newblock Springer-Verlag, Berlin-New York, 1972.
    
    \bibitem[SGA 4$\rm_{III}$]{SGA4III}
    M.~{Artin}, A.~{Grothendieck}, and J.~L. {Verdier}, editors.
    \newblock {\em {S\'{e}minaire de G\'{e}om\'{e}trie Alg\'{e}brique du Bois-Marie 1963-1964. Th\'{e}orie des topos et cohomologie \'{e}tale des sch\'{e}mas (SGA 4). {T}ome 3.}}, volume 305 of {\em Lecture Notes in Mathematics}.
    \newblock Springer-Verlag, Berlin-New York, 1973.
    
    \bibitem[SGA 7$\rm_I$]{SGA7I}
    Alexander Grothendieck, M.~Raynaud, and D.~S. Rim, editors.
    \newblock {\em S{\'e}minaire de {G{\'e}om{\'e}trie} {Alg{\'e}brique} {Du} {Bois}-{Marie} 1967--1969. {Groupes} de monodromie en g{\'e}om{\'e}trie alg{\'e}brique ({SGA} 7). {T}ome I.}, volume 288 of {\em Lecture Notes in Mathematics}.
    \newblock Springer, Cham, 1972.
    
    \bibitem[Spa88]{Spaltenstein88Unbounded}
    N.~Spaltenstein.
    \newblock Resolutions of unbounded complexes.
    \newblock {\em Compositio Mathematica}, 65(2):121--154, 1988.
    
    \bibitem[SS03]{SS03Albanese}
    Michael Spie\ss\ and Tam\'{a}s Szamuely.
    \newblock On the {A}lbanese map for smooth quasi-projective varieties.
    \newblock {\em Mathematische Annalen}, 325(1):1--17, 2003.
    
    \bibitem[Stacks]{stacks-project}
    The {Stacks Project authors}.
    \newblock {Stacks Project}.
    \newblock \url{https://stacks.math.columbia.edu/}, 2024.
    
    \bibitem[Sus17]{Suslin17Nonperfect}
    Andrei Suslin.
    \newblock Motivic complexes over nonperfect fields.
    \newblock {\em Annals of K-Theory}, 2(2):277--302, 2017.
    
    \bibitem[{Voe}00]{Voevodsky00DM}
    Vladimir {Voevodsky}.
    \newblock {Triangulated categories of motives over a field}.
    \newblock In {\em {Cycles, transfers, and motivic homology theories}}, pages 188--238. Princeton, NJ: Princeton University Press, 2000.
    
    \bibitem[{Yu}19]{Yu19ChowSAV}
    Chia~Fu {Yu}.
    \newblock {Chow's theorem for semi-abelian varieties and bounds for splitting fields of algebraic tori}.
    \newblock {\em {Acta Mathematica Sinica. English Series}}, 35(9):1453--1463, 2019.
    
    \end{thebibliography}
\end{document}